\documentclass[11pt]{amsart}
\usepackage{preamble}

\usepackage[headheight=15pt, headsep=15pt, footskip=27pt, bottom=2.4cm, left=2.4cm, right=2.4cm]{geometry}

\begin{document}

%\linenumbers

\title{On the Morrison-Kawamata dream space and its applications}



\begin{abstract}
We develop the theory of Morrison-Kawamata dream spaces, which axiomatizes varieties (not necessarily of Calabi-Yau type) that satisfy the Morrison-Kawamata cone conjecture. Using this theory, we establish the generic deformation invariance of various cones and apply it to the boundedness problem of algebraic varieties.
\end{abstract}

\author{Sung Rak Choi}
\address[Sung Rak Choi]{Department of Mathematics, Yonsei University, 50 Yonsei-ro, Seodaemun-gu, Seoul 03722, Republic of Korea}
\email{sungrakc@yonsei.ac.kr}

\author{Xingying Li}
\address[Xingying Li]{Department of Mathematics, Southern University of Science and Technology, 1088 Xueyuan Road, Shenzhen 518055, China} \email{xingyinglimath@gmail.com}

\author{Zhan Li}
\address[Zhan Li]{Department of Mathematics, Southern University of Science and Technology, 1088 Xueyuan Road, Shenzhen 518055, China} \email{lizhan@sustech.edu.cn, lizhan.math@gmail.com}

\author{Chuyu Zhou}
\address[Chuyu Zhou]{School of Mathematical Sciences, Xiamen University, Siming South Road 422, Xiamen, Fujian 361005, China}
\email{chuyuzhou@xmu.edu.cn, chuyuzhou1@gmail.com}

\maketitle

\setcounter{tocdepth}{2} % 2 includes subsections

\tableofcontents

\section{Introduction}\label{sec: intro}

This paper works over the field of complex numbers $\bC$.

We develop the theory of Morrison-Kawamata dream spaces, providing an axiomatic framework to characterize varieties that satisfy the Morrison-Kawamata cone conjecture (cf.  \cite{Mor93, Mor96, Kaw97, Tot09}). These spaces can also be viewed as a way to ``glue'' the local theory of Mori dream spaces under the action of birational automorphisms. Morrison-Kawamata dream spaces include Mori dream spaces (in particular, Fano type varieties), Calabi-Yau type varieties (under the assumption of the cone conjecture and the good minimal model conjecture), and many other varieties that are neither of Mori dream space type nor of Calabi-Yau type:
\[
\begin{array}{ccc}
\left\{\text{Fano type varieties}\right\}  &\subset& \{\text{Mori dream spaces}\}  \\
\cap && \cap \\
 \left\{\substack{\textstyle\text{Calabi-Yau type varieties}\\ \textstyle\text{satisfying the cone conjecture and}\\ \textstyle\text{ the good minimal model conjecture}}\right\}&\subset&  \{\text{Morrison-Kawamata dream spaces}\}.
\end{array}
\] It is remarkable that many properties of Mori dream spaces and Calabi-Yau type varieties can be extended under this unified framework, and the corresponding proofs become more conceptual. It is desirable that Morrison-Kawamata dream spaces serve as natural generalizations of Calabi-Yau type varieties in the study of their birational geometry, just as Mori dream spaces play an analogous role for varieties of Fano type.

To motivate the definition of Morrison-Kawamata dream spaces, we first recall the notion of Mori dream spaces. Please see Section \ref{sec: pre} for the meaning of the notation used in the sequel.

\begin{definition}[Mori dream space {\cite[1.10]{HK00}}]\label{def: MDS}
We will call a normal projective variety $X$ a Mori dream space provided the following hold:
\begin{enumerate}[label=(\roman*)]
    \item $X$ is $\mathbb{Q}$-factorial and $\mathrm{Pic}(X)_{\mathbb{Q}} = N^1(X)$,
    \item $\Nef(X)$ is the affine hull of finitely many semi-ample line bundles, and
    \item there is a finite collection of small $\bQ$-factorial modifications $f_i : X \dashrightarrow X_i$ such that each $X_i$ satisfies (ii) and $\mathrm{Mov}(X)$ is the union of the $f_i^*(\Nef(X_i))$.
\end{enumerate}
\end{definition}

If one wishes to retain the features of a Mori dream space while allowing birational actions, one may require the following conditions:

\begin{enumerate}[label=(\alph*)]
\item $X$ is $\mathbb{Q}$-factorial,
\item there exists a rational polyhedral cone $\Pi\subset \Mov(X)$ such that $\PsAut(X) \cdot \Pi = \Mov(X)$,
\item there is a finite collection of small $\bQ$-factorial modifications $f_i : X \dashrightarrow X_i, 1 \leq i \leq l$ such that $\Pi \subset \cup_{i=1}^lf_i^*(\Nef(X_i))$ and each $\Pi \cap f_i^*(\Nef(X_i))$ is a rational polyhedral cone, and
\item $f_{i,*}D$ is semi-ample for each effective $\bQ$-Cartier divisor $D$ with $[D]\in \Pi \cap f_i^*(\Nef(X_i))$.
\end{enumerate}

The guiding principle is that locally (i.e., inside $\Pi$), the variety behaves like a Mori dream space. The conditions (a)--(d) above are weaker than the corresponding ones in the definition of a Mori dream space. Indeed, a variety $X$ satisfying conditions (a)--(d) may have $h^1(X, \Oo_X) > 0$ (cf. Definition \ref{def: MDS} (i)), since we aim to include varieties with trivial canonical divisors. Moreover, we do not require that every small $\bQ$-factorial modification is again a Mori dream space (cf. Definition \ref{def: MDS} (ii)).

From the perspective of axiomatizing the Morrison-Kawamata cone conjecture, we introduce the following notion of Morrison-Kawamata dream spaces:

\begin{restatable}[Morrison-Kawamata dream fiber space]{definition}{MKDspaces}\label{def: MKD spaces}
Suppose that $X \to T$ is a projective morphism between normal quasi-projective varieties. Then $X/T$ is called a Morrison-Kawamata dream fiber space (MKD fiber space) if
    \begin{enumerate}
    \item $X$ is a $\bQ$-factorial variety,    
     \item every effective $\bR$-Cartier divisor admits a good minimal model$/T$,
        \item there exists a rational polyhedral cone $\Pi \subset \Mov(X/T)$ such that $\PsAut(X/T) \cdot \Pi = \Mov(X/T)$, and
        \item $\Eff(X/T)$ satisfies the local factoriality of canonical models$/T$.
    \end{enumerate}
    
    If $T$ is a closed point, then $X$ is called a Morrison-Kawamata dream space (MKD space).
\end{restatable}

For the definition of local factoriality of canonical models, see Definition \ref{def: local factoriality}. In the sequel, when there is no ambiguity, we also refer to an MKD fiber space simply as an MKD space.

This notion of a Morrison-Kawamata dream space is abstracted from our study of the Morrison-Kawamata cone conjecture in \cite{LZ25, Li26}. Although, at first sight, it does not appear to be a direct analogue of Mori dream spaces with birational group actions, namely conditions (a)--(d) above, Theorem \ref{thm: two def are same} shows that the two notions are indeed equivalent, up to mild variants, at least in the absolute setting.

\begin{theorem}\label{thm: two def are same}
Let $X$ be a normal projective variety. Assume that every effective $\bR$-Cartier divisor admits a minimal model. Then $X$ is an MKD space if and only if it satisfies the following conditions:
\begin{enumerate}[label=(\alph*)]
\item $X$ is $\mathbb{Q}$-factorial,
\item there exists a rational polyhedral cone $\Pi\subset \bMov^e(X)$ such that $\PsAut(X) \cdot \Pi = \bMov^e(X)$,
\item there is a finite collection of small $\bQ$-factorial modifications $f_i : X \dashrightarrow X_i, 1 \leq i \leq l$ such that $\Pi \subset \cup_{i=1}^lf_i^*(\Nef(X_i))$ and each $\Pi \cap f_i^*(\Nef(X_i))$ is a rational polyhedral cone, and
\item $f_{i,*}D$ is semi-ample for each effective $\bR$-Cartier divisor $D$ with $[D]\in \Pi \cap f_i^*(\Nef(X_i))$.
\end{enumerate}
\end{theorem}

Theorem \ref{thm: two def are same} should also serve as the definition of MKD fiber spaces, that is, in the relative setting. However, due to certain combinatorial issues when the movable cone is degenerate, we are unable to show that this relative version is equivalent to Definition \ref{def: MKD spaces}.

We remark that alternative formulations of what may be called MKD fiber spaces are possible, which are not necessarily equivalent to Definition~\ref{def: MKD spaces}. Furthermore, certain pathological phenomena may arise under the general notion adopted here, mainly due to the lack of the Cone Theorem (see \cite[Theorem 3.7]{KM98}). First, there is no analogue of the boundedness of lengths of $D$-negative extremal rays, which motivates the local factoriality assumption for canonical models in Definition \ref{def: MKD spaces} (4). Secondly, numerical triviality does not necessarily imply linear triviality for $D$-negative extremal contractions. To overcome this difficulty, \cite{KKL16} introduces the concept of a gen divisor; in our framework, we instead assume that every effective divisor admits a good minimal model.

Under the above definition, we develop the theory of MKD fiber spaces in the sequel.

First, we establish the chamber structure of minimal models (cf. \cite{Sho96, SC11, KKL16, LZ25}).

\begin{theorem}[Shokurov polytope for minimal models]\label{thm: Shokurov poly for minimal models}
Let $X/T$ be a normal $\bQ$-factorial variety. Let $\mathcal C \subset {\rm CDiv}(X)_\Rr$ be a cone generated by finitely many effective Cartier divisors. Suppose that every effective $\bR$-Cartier divisor in $\mathcal C$ admits a good minimal model$/T$, and $\mathcal C$ satisfies the local factoriality of canonical models$/T$. Then $\mathcal C$ can be written as a disjoint union of finitely many relatively open rational polyhedral cones 
\[
\mathcal C = \bigsqcup_{i=0}^m \mathcal C_i,
\]
such that for any effective $\bR$-Cartier divisors $B, D \in \mathcal C_i$, whenever $X \dto Y/T$ is a weak minimal model of $B$, it is also a weak minimal model of $D$.
\end{theorem}

The following version of the chamber structure for nef cones is analogous to the case of Calabi-Yau fiber spaces described in \cite[Theorem 2.7]{LZ25}. A statement for log pairs was proved in \cite[\S 6.2, First Main Theorem]{Sho96} (see also \cite[Proposition 3.2]{Bir11}) without assuming the existence of good minimal models. However, the existence of good minimal models is essential in the setting of MKD fiber spaces.

\begin{theorem}[Shokurov polytope for nef cones]\label{thm: Shokurov poly for nef}
Let $X/T$ be a normal $\bQ$-factorial variety. Let $\mathcal P \subset {\rm CDiv}(X)_\Rr$ be a cone generated by finitely many effective Cartier divisors. Suppose that every effective $\bR$-Cartier divisor in $\mathcal P$ admits a good minimal model$/T$, and $\mathcal P$ satisfies the local factoriality of canonical models$/T$. Then
\[
 \mathcal N_{\mathcal P}\coloneqq \{D \in \mathcal P \mid D \text{~is nef over~} T\}
 \] is a rational polyhedral cone.
\end{theorem}

We use $\Gamma_B(X/T)$ and $\Gamma_A(X/T)$ to denote the images of the pseudo-automorphism group $\PsAut(X/T)$ and the automorphism group $\Aut(X/T)$ under the natural group homomorphism 
\[
\rho: \PsAut(X/T, \Delta) \to {\rm GL}(N^1(X/T)).
\]
See Section~\ref{subsec: Movable cones and ample cones} for further explanation of the notation. The following extends the main result of \cite{GLSW26} to MKD fiber spaces.

\begin{theorem}\label{thm: 3 equivalences}
    Let $X/T$ be a normal $\bQ$-factorial variety. Assume that every effective $\bR$-Cartier divisor admits a good minimal model$/T$. Suppose that $\Eff(X/T)$ satisfies the local factoriality of canonical models$/T$. Then the following statements are equivalent:
    \begin{enumerate}
    \item $X/T$ is an MKD space.
    \item $\Nef^e(X'/T)$ admits a rational polyhedral fundamental domain under the action of $\Gamma_A(X'/T)$ for any small $\bQ$-factorial modification $X \dto X'/T$, and
        \[
        \{Y/T \mid X \dto Y/T \text{~is a birational contraction}\}
        \] is a finite set up to isomorphism of $Y/T$. 
        \item There is a rational polyhedral cone $P\subset \Eff(X/T)$ such that $\Gamma_B(X/T) \cdot P = \Eff(X/T)$. In particular, $\Eff(X/T)_+$ admits a weak rational fundamental domain under the action of $\Gamma_B(X/T)$. 
    \end{enumerate}
\end{theorem}

The following two results demonstrate that MKD fiber spaces retain the key properties of Mori dream spaces.

\begin{theorem}\label{thm: bir contraction is MKD}
Let $X/T$ be an MKD fiber space. If $f: X \dto Y/T$ is a birational contraction with $Y$ a $\bQ$-factorial variety, then $Y/T$ is still an MKD fiber space.
\end{theorem}

The next result concerns the minimal model program (MMP) for effective divisors on MKD fiber spaces.

\begin{theorem}\label{thm: MMP for MKD}
Let $X/T$ be an MKD fiber space. Then, for any effective $\bR$-Cartier divisor $D$, one can run a $D$-MMP$/T$ with scaling of an ample divisor, and this MMP terminates with a $D$-good minimal model$/T$. Moreover, every variety appearing in this MMP$/T$ is an MKD fiber space.
\end{theorem}

Unlike for Mori dream spaces, it is in general impossible to run MMPs for non-pseudo-effective divisors on MKD spaces. See Remark \ref{rmk: no MMP for non-effective div} for further discussion.

The following result constitutes the first step in applying MKD fiber spaces to moduli problems.

\begin{theorem}\label{thm: geometric MKD space}
Let $X$ be a fibration over $T$ such that the geometric generic fiber $X_{\bar \eta}$ is a klt MKD space. Then there exists a generically finite morphism $T' \to T$ such that, after shrinking $T'$, the following properties hold:
\begin{enumerate}
 \item $X_{T'}$ has klt singularities.
 \item The induced morphism $X_{T'} \to T'$ is an MKD fiber space.
 \item There exist natural isomorphisms 
 \[
 \Mov(X_{T'}/T') \simeq \Mov(X_{\bar\eta}) \quad \text{and} \quad \Eff(X_{T'}/T') \simeq \Eff(X_{\bar\eta}).
 \]
 \item $\Mov(X_{T'}/T'), \Eff(X_{T'}/T')$ are non-degenerate cones, and 
 \[
 \Mov(X_{T'}/T') = \Mov(X_{T'}/T')_+, \quad \Eff(X_{T'}/T') = \Eff(X_{T'}/T')_+.
 \] 
\end{enumerate}
Moreover, these properties remain valid for any generically finite morphism $T'' \to T$ factoring through $T' \to T$.
\end{theorem}

As a first application of the theory of MKD fiber spaces, we extend results on the deformation invariance of various cones from varieties of Fano type to MKD fiber spaces (cf. \cite{Wis91, Wis09, dFH11, HX15, Sho20, FHS24, CHHX25, CLZ25}). Such results play a crucial role in establishing the boundedness of moduli spaces of varieties (see \cite{HX15, HMX18, MST20, FHS24, CLZ25}, etc.) and the boundedness of complements (see \cite{Sho20, CHHX25}, etc.).

\begin{theorem}\label{thm: def of cone together}
Let $f: X \to T$ be a fibration. Suppose that $S \subset T$ is a Zariski dense subset such that for any $s \in S$, the fiber $X_s$ satisfies
\[
H^1(X_s, \mO_{X_s}) = H^2(X_s, \mO_{X_s}) = 0.
\]
Assume further that the geometric generic fiber $X_{\bar\eta}$ is a klt MKD space. Then, after a generically finite base change of $T$, there exists a non-empty Zariski open subset $T_0 \subset T$ such that for any Zariski open subset $U \subset T_0$ and $t\in U$, the natural map
\[
N^1(X_U/U) \to N^1(X_t), \quad [D] \mapsto [D|_{X_t}]
\] is an isomorphism which induces the isomorphisms
\[
\Nef(X_U/U) \simeq \Nef(X_t), \quad \Eff(X_U/U) \simeq \Eff(X_t), \quad \Mov(X_{U}/U) \simeq \Mov(X_t).
\] These isomorphisms also identify the Mori chamber decompositions of $ \Mov(X_{U}/U)$ and $ \Mov(X_t)$.

Moreover, the following statements hold:
\begin{enumerate}
\item For any $\bR$-divisor $\mD\in \Eff(X_U/U)$, a sequence of $\mD$-MMP$/U$ induces a sequence of $\mD|_{X_t}$-MMP of the same type for each $t\in U$. 
\item Conversely, for any $\bR$-divisor $D \in \Eff(X_t)$ with $t \in U$, any sequence of $D$-MMP on $X_t$ is induced by a sequence of $\mD$-MMP$/U$ on $X_U/U$ of the same type, where $\mD$ is an effective divisor satisfying $[\mD|_{X_t}] = [D]$.
\end{enumerate}
\end{theorem}

We remark that \cite{dFH11, HX15, FHS24, CHHX25} essentially rely on vanishing and extension theorems, which are not available in the general MKD setting. To establish Theorem \ref{thm: def of cone together}, we adopt a different perspective along the lines of \cite{CLZ25}  (it seems that \cite[\S 4]{Sho20} uses a similar method). Instead of considering deformations of a fixed $X_0$, we allow stratifications of the base. This relaxes the technical restrictions on $X_0$ while strengthening the resulting conclusions. For example, we establish the generic deformation invariance of nef cones, effective cones, movable cones, and Mori chamber decompositions, which do not hold without stratifications (see the discussion in \cite[\S 5]{FHS24}). This approach preserves the same strength in applications to boundedness problems through Noetherian induction.

The above results are established in Theorem \ref{thm: def of nef cone}, Lemma \ref{lem: uniform behavior in bc}, Theorem \ref{thm: MMP}, and Theorem \ref{thm: mov}. In fact, a slightly stronger statement holds: the conclusions remain valid for any MKD fiber space that is a birational contraction of $X/T$.

Together with \cite[Theorem~1.6]{Bir23}, this leads to the following folklore result.

\begin{theorem}\label{thm: bdd for rationally connected CY}
Let $\mathcal S_n$ be a set of rationally connected Calabi-Yau varieties of dimension $n$ with klt singularities. Assume that the Morrison-Kawamata cone conjecture holds for every $n$-dimensional rationally connected Calabi-Yau variety with klt singularities, and every effective $\bR$-Cartier divisor on such a variety admits a good minimal model. Then there exists a projective morphism $g$ between schemes of finite type such that for any $X \in \mathcal S_n$, the variety $X$ is isomorphic to some fiber of $g$.
\end{theorem}

In the proofs of the above results, we study various cones equipped with suitable group actions. One of the key ingredients is the generic nef cone.

\begin{definition}\label{def: generic nef cone}
Let $X$ be a normal variety projective over $T$. Then
\[
{\rm GNef}(X/T) \coloneqq \{[D] \in \Eff(X/T) \mid [D_\eta] \in \Nef(X_\eta)\}
\] is a convex cone inside $N^1(X/T)$ which is called the generic nef cone.
\end{definition}

The generic nef cone encodes the information about contractions $X \dto Y/T$ which become morphisms after restricting to a non-empty open subset $U\subset T$. ${\rm GNef}(X/T)$ admits a natural group action by the generic automorphism group
\begin{equation}\label{eq: generic auto}
{\rm GAut}(X/T) \coloneqq \{g\in \PsAut(X/T) \mid g_\eta \in \Aut(X_\eta)\}.
\end{equation}
Beyond their intrinsic interest, these examples illustrate that the geometric problem at hand dictates which cones and group actions one should consider.

As before, let $\Gamma_{GA}(X/T)$ be the image of ${\rm GAut}(X/T)$ under the natural group homomorphism $\rho: \PsAut(X/T) \to {\rm GL}(N^1(X/T))$. Then the analogous Morrison-Kawamata cone conjecture holds for the generic nef cone of an MKD fiber space.

\begin{theorem}\label{thm: cone for GNef}
    Let $X/T$ be an MKD fiber space. 
    \begin{enumerate}
    \item There is a rational polyhedral cone $Q\subset {\rm GNef}(X/T)$ such that $\Gamma_{GA}(X/T) \cdot Q \supset {\rm GNef}(X/T)$. 
    \item ${\rm GNef}(X/T)_+$ admits a weak rational polyhedral fundamental domain under the action of $\Gamma_{GA}(X/T)$. 
    \item If ${\rm GNef}(X/T)$ is non-degenerate, then ${\rm GNef}(X/T)_+={\rm GNef}(X/T)$.
    \item The set \[
    \begin{split}
    \{Y_\eta \mid &X \dto Y/T \text{~is a map such that~} X_U \to Y_U/U \\&\text{~is a contraction morphism for some non-empty open subset~} U \subset T\}
    \end{split}
    \]
    is finite up to isomorphism of $Y_\eta$.
    \end{enumerate}
\end{theorem}

Finally, we discuss the organization of the paper. Section~\ref{sec: pre} sets up the notation and terminology and provides the necessary background on convex geometry. Section~\ref{sec: MKD spaces} develops the foundations of MKD fiber spaces. More precisely, Theorems~\ref{thm: Shokurov poly for minimal models} and~\ref{thm: Shokurov poly for nef} are proved in Section~\ref{subsec: shokurov polytope}, Theorem~\ref{thm: 3 equivalences} is proved in Section~\ref{subsec: cone conj for nef and eff cones}, Theorem~\ref{thm: bir contraction is MKD} is proved in Section~\ref{subsec: MKD spaces under birational contractions}, Theorem~\ref{thm: MMP for MKD} is proved in Section~\ref{subsec: MMP for MKD}, Theorem~\ref{thm: two def are same} is proved in Section~\ref{subsec: equivalence}, and Theorem~\ref{thm: geometric MKD space} is proved in Section~\ref{subsec: geometric MKD space}. Section~\ref{sec: variant of cones} introduces cones other than the usual nef, effective, and movable cones, and studies the corresponding cone conjectures. The geometric consequences are derived from these cones, and Theorem~\ref{thm: cone for GNef} is proved in this section. Section~\ref{sec: examples and questions} provides examples of MKD fiber spaces and lists open questions about them. Section~\ref{sec: deformation of cones of MKD spaces} presents applications of the theory of MKD fiber spaces; in particular, Theorems~\ref{thm: def of cone together} and~\ref{thm: bdd for rationally connected CY} are proved there.

\subsection*{Acknowledgements}  
We thank Sheng Meng for pointing out the reference \cite{Ogu00}. S. Choi is partially supported by Samsung Science and Technology Foundation under Project Number SSTF-BA2302-03. Z. Li is partially supported by the NSFC (No.12471041), the Guangdong Basic and Applied Basic Research Foundation (No.2024A1515012341). C. Zhou is supported by the NSFC (No.12501058) and a grant from Xiamen University (No.X2450214).

\section{Preliminaries}\label{sec: pre}

\subsection{Notation and terminology}
We write $X/T$, or say that $X$ is a variety over $T$, for a projective morphism $f\colon X\to T$ between normal quasi-projective varieties over $\bC$. In this case, we also say that $X$ is a fiber space over $T$. We call $X/T$, or $f$, a fibration if $f$ is surjective with connected fibers. Thus, under our convention, a fiber space need not be a fibration. This convention is chosen to be compatible with our definition of MKD fiber spaces.

By divisors, we mean Weil divisors. For $\mathbb K=\Zz, \Qq, \Rr$ and two $\mathbb K$-divisors $A, B$ on $X/T$, $A \sim_{\mathbb K} B/T$ means that $A$ and $B$ are $\mathbb K$-linearly equivalent over $T$. If $A, B$ are $\Rr$-Cartier divisors, then $A \equiv B/T$ means that $A$ and $B$ are numerically equivalent over $T$. An $\bR$-Cartier divisor $D$ on $X$ is called effective over $T$ if there exists an effective $\bR$-Cartier divisor $B$ such that $D \sim_\bR B/T$. A Cartier divisor $D$ is called movable over $T$ if the base locus of the relative linear system $|D/T|$ has codimension greater than $1$. It is called semi-ample over $T$ if there exists $m\in \Zz_{>0}$ such that $|mD/T|$ is base-point free. An $\bR$-Cartier divisor $D$ on $X$ is said to be semi-ample over $T$ if it can be written as an $\bR_{>0}$-linear combination of semi-ample Cartier divisors over $T$. In the following, we use $|D/T|_\bR$ to denote the relative $\bR$-linear system.

We use $\Supp(E)$ to denote the support of the divisor $E$. For a birational map $g: X \dto Y$, we use $\Exc(g)$ to denote the support of the exceptional divisors. If $D$ is an $\Rr$-Cartier divisor on $X$, we denote by $g_*D$ its strict transform on $Y$, defined as follows. Let $p: W \to X, q: W \to Y$ be birational morphisms such that $q \circ p^{-1}=g$, then $g_*D \coloneqq q_*(p^*D)$. This is independent of the choice of $p$ and $q$. A birational map $g: X \dto Y$ is called a birational contraction if $g^{-1}$ does not extract any divisor. 

Let $X/T$ be a normal complex variety and $\De \geq 0$ be an $\Rr$-divisor on $X$, then $(X/T, \De)$ is called a log pair. We assume that $K_X+\De$ is $\Rr$-Cartier for a log pair $(X, \De)$, where $K_X$ is the canonical divisor of $X$. A log pair $(X,\Delta)$ has klt singularities if there exists a log resolution $\pi: Y \to X$ such that in the expression
\begin{equation}\label{eq: klt}
    K_Y=\pi^*(K_X+\Delta)+D,
\end{equation}
the coefficients of $D$ are greater than $-1$. Note that in \eqref{eq: klt}, $K_Y$ is chosen to be the unique Weil divisor on $Y$ such that $\pi_*K_Y=K_X$. Similarly, if the coefficients of $D$ are greater than or equal to $-1$, then $(X,\Delta)$ is said to have lc singularities. See \cite[\S 2.3]{KM98} for a more detailed discussion. Throughout this paper, a log pair $(X/T, \De)$ is called a Calabi-Yau pair over $T$ if $(X, \De)$ has klt singularities with $K_X+\De \sim_\Rr 0/T$. We say that $X/T$ is of Calabi-Yau type over $T$ if there exists a $\De$ such that $(X/T, \De)$ is a Calabi-Yau pair. A Calabi-Yau pair $(X/T, \De)$ is called a Calabi-Yau fibration if $X \to T$ is a fibration. 

\subsection{Minimal models of $\bR$-Cartier divisors}\label{subsec: mm}

\begin{definition}[{\cite[Definition 3.6.1]{BCHM10}}] 
Let $\phi: X \dto Y$ be a proper birational contraction of normal
quasi-projective varieties and let $D$ be an $\bR$-Cartier divisor on $X$ such that $D_Y=\phi_*D$ is also $\bR$-Cartier. We say that $\phi$ is $D$-non-positive (resp. $D$-negative) if for some common resolution $p: W \to X$ and $q: W \to Y$, we may write 
\[
p^*D= q^*D_Y + E,
\]
where $E\geq 0$ is $q$-exceptional (resp. $E \geq 0$ is $q$-exceptional and the support of $E$ contains (the strict transforms of) the $\phi$-exceptional divisors).
\end{definition}

We include the following well-known lemma (for example, see \cite[Lemma 1.7]{HK00}).

\begin{lemma}\label{lem: common mm}
Let $X$ and $Y$ be varieties over $T$, and suppose that $p\colon W\to X$ and $q\colon W\to Y$ are projective birational contraction morphisms over $T$. Assume that there exist a nef$/T$ $\bR$-Cartier divisor $B$ on $X$ and a nef$/T$ $\bR$-Cartier divisor $D$ on $Y$ such that
\[
p^*B+E\equiv q^*D+F/T,
\] where $E$ is an effective $p$-exceptional divisor and $F$ is an effective $q$-exceptional divisor. Then we have $E=F$. Furthermore, 
\begin{enumerate} 
\item if $D=q_*p^*B$, then $p^*B= q^*D$, and
\item if $D$ is ample$/T$, then $q\circ p^{-1}: X \dto Y$ is a morphism.
\end{enumerate}
\end{lemma}

   Minimal models and ample models can also be defined analogously to those in \cite[Definition 3.6.5, Definition 3.6.7]{BCHM10}. Note that the minimal model below is referred to as an optimal model in \cite[Definition 2.3]{KKL16}.

\begin{definition}\label{def: minimal model}
Let $X \to T$ be a projective morphism of normal quasi-projective varieties and $\phi: X \dto Y/T$ be a birational contraction with $Y$ projective over $T$. Suppose that $D$ is an $\bR$-Cartier divisor on $X$ with $D_Y=\phi_*D$ an $\bR$-Cartier divisor on $Y$.
\begin{enumerate}
\item $Y/T$ (or $\phi$) is a weak minimal model$/T$ of $D$ if $\phi$ is $D$-non-positive and $D_Y$ is nef$/T$.
\item $Y/T$ (or $\phi$)  is a minimal model$/T$ of $D$ if $Y$ is $\bQ$-factorial, $\phi$ is $D$-negative and $D_Y$ is nef$/T$.
\item A minimal model $Y/T$ (or $\phi$) is called a good minimal model$/T$ of $D$ if $D_Y$ is semi-ample$/T$.
\end{enumerate}

Moreover, we say that $g\colon X \dto Z/T$ is the ample model$/T$ of $D$ if $Z$ is normal and projective$/T$, and there exists an ample$/T$ divisor $H$ on $Z$ satisfying the following property: if $p\colon W \to X$ and $q\colon W \to Z$ resolve $g$, then $p$ is a birational contraction morphism and
\[
p^*D \sim_{\bR} q^*H+E/T,
\]
where $E \geq 0$, and every divisor $B \in |p^*D/T|_\bR$ satisfies $B \geq E$.
\end{definition}

To simplify the notation, we usually omit ``$/T$'' from expressions such as ``(weak) minimal models$/T$'' when the base is clear from the context.

For a variety of Calabi-Yau type, the good minimal model conjecture predicts that every effective $\bR$-Cartier divisor admits a good minimal model. We say that such a variety satisfies the good minimal model conjecture if this prediction holds for all effective $\bR$-Cartier divisors on it. This conjecture is known to hold in dimension at most $3$.

\begin{lemma}\label{lem: iso in codim 1 preserves mm}
Let $h: X \dto X'/T$ be a birational map which is isomorphic in codimension $1$. Suppose that $D$ is an $\Rr$-Cartier divisor on $X$ and that $D' = h_*D$ is an $\Rr$-Cartier divisor on $X'$. If $g: X' \dto Y$ is a minimal model$/T$ of $D'$, then $g \circ h$ is a minimal model$/T$ of $D$.
\end{lemma}
\begin{proof}
Let $p: W \to X, q: W \to X'$, and $r: W \to Y$ be projective birational morphisms such that $h=q \circ p^{-1}$ and $g=r \circ q^{-1}$. By assumption, we have
\[
q^*D'=r^*D_Y + E,
\] where $D_Y=g_*D'$ is nef$/T$ and $E$ is an $r$-exceptional divisor such that $\Supp E$ contains the strict transforms of divisors contracted by $g$. As $h$ is isomorphic in codimension $1$, we have
\[
p^*D+F=q^*D',
\] where $F$ is a $p$-exceptional divisor. Combining with the above equation, we have
\[
p^*D=r^*D_Y + (E-F).
\] As $F-E$ is nef over $X$ and $p_*(E-F)=p_*E \geq 0$, by the negativity lemma \cite[Lemma 3.39]{KM98}, we have $E-F \geq 0$. As $F$ is $p$-exceptional and $h$ is isomorphic in codimension $1$, $\Supp (E-F)$ contains the strict transforms of divisors contracted by $g\circ h$. This shows that $g \circ h$ is a minimal model$/T$ of $D$.
\end{proof}

\subsection{Cones in N\'eron-Severi spaces and the Morrison-Kawamata cone conjecture}\label{subsec: Movable cones and ample cones}

Let $X$ be a $\bQ$-factorial variety over $T$. Let $\operatorname{WDiv}(X)$ and $\operatorname{CDiv}(X)$ denote the free abelian groups generated by the Weil divisors and the Cartier divisors on $X$, respectively, and let $\Pic(X/T)$ denote the relative Picard group. Since $X$ is $\bQ$-factorial, 
\[
N^1(X/T)_\bZ \coloneqq {\rm WDiv}(X)/{\equiv}
\]
forms a lattice. For $\mathbb K = \Qq$ or $\Rr$, set  
\[
\begin{split}
&{\rm WDiv}(X)_{\mathbb K} \coloneqq {\rm WDiv}(X) \otimes_\Zz \mathbb K, \quad 
{\rm CDiv}(X)_{\mathbb K} \coloneqq {\rm CDiv}(X) \otimes_\Zz \mathbb K, \\
&\Pic(X/T)_{\mathbb K} \coloneqq \Pic(X/T) \otimes_\Zz \mathbb K, \quad N^1(X/T)_{\mathbb K} \coloneqq N^1(X/T)_\bZ \otimes_\Zz \mathbb K.
\end{split}
\]
For simplicity, we write  
\[
N^1(X/T) \coloneqq N^1(X/T)_\bR.
\]

If $D$ is an $\bR$-Cartier divisor, then $[D]\in N^1(X/T)$ denotes its numerical class. By abuse of terminology, we also refer to $[D]$ as an $\bR$-Cartier divisor. For a subset $S\subset N^1(X/T)$, let $\Conv(S)$ denote the convex hull of $S$.

We list relevant cones inside $N^1(X/T)$ which appear in the paper:

\begin{enumerate}
\item $\Eff(X/T)\coloneqq$ the cone generated by effective Cartier divisors;
\item $\bEff(X/T)\coloneqq$ the closure of $\Eff(X/T)$ (i.e., the cone generated by pseudo-effective divisors);
\item $\Eff(X/T)_+ \coloneqq \Conv(\bEff(X/T) \cap N^1(X/T)_\Qq)$;
\item $\Mov(X/T)\coloneqq$ the cone generated by movable divisors;
\item $\bMov(X/T)\coloneqq$ the closure of $\Mov(X/T)$;
\item $\bMov^e(X/T) \coloneqq \bMov(X/T) \cap \Eff(X/T)$;
\item ${\Mov(X/T)_+}\coloneqq \Conv(\bMov(X/T) \cap N^1(X/T)_\Qq)$;
\item $\Amp(X/T)\coloneqq$ the cone generated by ample divisors;
\item $\Nef(X/T)\coloneqq$ the closure of $\Amp(X/T)$  (i.e., the cone generated by nef divisors);
\item $\Nef^e(X/T) \coloneqq \Nef(X/T) \cap \Eff(X/T)$;
\item  $\Nef(X/T)_+ \coloneqq\Conv(\Nef(X/T) \cap N^1(X/T)_\Qq)$.
\end{enumerate}

If $K$ is a field of characteristic zero and $X$ is a variety over $K$, then the above cones still make sense for $X$. Let $N_1(X/T)$ be the dual vector space of $N^1(X/T)$. Then the Mori cone $\overline{\NE}(X/T)\subset N_1(X/T)$ is the dual cone of $\Nef(X/T)$.

Let $\De$ be a divisor on a $\Qq$-factorial variety $X$. We denote by $\Bir(X/T, \De)$ the birational automorphism group of $(X/T, \De)$ over $T$. More precisely, $\Bir(X/T, \De)$ consists of birational maps $g: X \dashrightarrow X/T$ such that $g_* \Supp \De = \Supp \De$. A birational map is called a pseudo-automorphism if it is an isomorphism in codimension $1$. Denote by $\PsAut(X/T, \De)$ the subgroup of $\Bir(X/T, \De)$ consisting of pseudo-automorphisms.  
Similarly, let $\Aut(X/T, \De)$ be the subgroup of $\Bir(X/T, \De)$ consisting of automorphisms of $X/T$.
 
Let $g\in \Bir(X/T, \De)$, and let $D$ be an $\Rr$-Cartier divisor on a $\Qq$-factorial variety $X$. Because the pushforward map $g_*$ preserves numerical equivalence classes, there is a linear map
\[
g_*: N^1(X/T)\to N^1(X/T), \quad [D] \mapsto [g_*D].
\] However, if $g\in \Bir(X/T)$ is not isomorphic in codimension $1$, then for $[D] \in \Mov(X/T)$, $[g_*D]$ may not be in $\Mov(X/T)$. Moreover, $(g , [D]) \mapsto [g_*D]$ may not be a group action of $\Bir(X/T, \De)$ on $N^1(X/T)$. For instance, if $D$ is a divisor contracted by $g$, then $(g^{-1})_*(g_*[D])= 0 \neq (g^{-1}\circ g)_*[D]$. On the other hand, it is straightforward to verify that
\[
\begin{split}
\PsAut(X/T, \De) \times N^1(X/T) &\to N^1(X/T),\\
(g, [D]) &\mapsto [g_*D],
\end{split}
\]
defines a natural group action. We use $g \cdot D$ and $g \cdot [D]$ to denote $g_*D$ and $[g_*D]$, respectively. Let $\Gamma_B(X/T, \De)$ and $\Gamma_A(X/T, \De)$ be the images of $\PsAut(X/T, \De)$ and $\Aut(X/T, \De)$, respectively, under the natural group homomorphism 
\[
\rho: \PsAut(X/T, \De) \to {\rm GL}(N^1(X/T)).
\] By abusing the notation, we also write $g$ for $\rho(g) \in \Gamma_B(X/T, \De)$, and denote $\rho(g)([D])$ by $g\cdot [D]$. 

The lattice $N^1(X/T)_\Zz$ is invariant under the action of $\PsAut(X/T,\De)$. Moreover, the cones $\Mov(X/T)$, $\bMov(X/T)$, $\bMov^e(X/T)$, and $\Mov(X/T)_+$ are invariant under this action. Similarly, the cones $\Amp(X/T)$, $\Nef(X/T)$, $\Nef^e(X/T)$, and $\Nef(X/T)_+$ are invariant under the action of $\Aut(X/T,\De)$. When $\De=0$, we omit $\De$ from the above notation for simplicity.

The following Morrison-Kawamata cone conjecture is one of the most important conjectures for Calabi-Yau fiber spaces, and it serves as the main motivation for defining MKD fiber spaces in this paper. It was formulated in increasing generality in \cite{Mor93, Mor96, Kaw97, Tot09}. See \cite{LOP18} for a survey of the conjecture. For the precise definitions of rational polyhedral fundamental domains and rational polyhedral cones, see Section~\ref{sec: Geometry of convex cones}.

\begin{conjecture}[Morrison-Kawamata cone conjecture]
Let $(X/T, \Delta)$ be a Calabi-Yau fiber space. 
\begin{enumerate}
    \item $\bMov^e(X/T)$ admits a rational polyhedral fundamental domain under the action of $\Gamma_B(X/T, \De)$.
    \item $\Nef^e(X/T)$ admits a rational polyhedral fundamental domain under the action of $\Gamma_A(X/T, \De)$.
\end{enumerate}
\end{conjecture}

Under the good minimal model conjecture in dimension $\dim(X/T)$, at least when $\bMov(X/T)$ is non-degenerate, we know that (1) is equivalent to the existence of a rational polyhedral cone $Q \subset \Eff(X/T)$ such that 
\[
\PsAut(X/T,\De) \cdot Q \supset \Mov(X/T)
\] 
(see \cite[Theorem~1.3 (2)]{LZ25}). Moreover, by \cite[Theorem~14]{Xu24} and \cite[Theorem~1.5]{GLSW26}, we know that (2) follows from (1).

The following well-known fact will be used in the sequel.

\begin{lemma}\label{lem: movable give iso in codim 1}
Let $D$ be an effective divisor on $X/T$ such that $[D] \in \bMov^e(X/T)$. If $D$ admits a minimal model$/T$ $h: X \dto Y/T$, then $h$ is isomorphic in codimension $1$.
\end{lemma}
\begin{proof}
Let $p: W\to X, q: W \to Y$ be projective birational morphisms such that $h=q\circ p^{-1}$. By assumption, we have
\[
p^*D=q^*D_Y+E,
\] where $D_Y\coloneqq h_*D$ and $E \geq 0$ is a $q$-exceptional divisor whose support contains $\Exc(h)$. Then $q^*D_Y$ and $E$ are the positive and negative parts of the Nakayama–Zariski decomposition$/T$ of $p^*D$, respectively (see \cite{Nak04}). As $[D] \in \bMov^e(X/T)$, we see that $E$ is $p$-exceptional. Hence, $\Exc(h)=0$, and thus $h$ is isomorphic in codimension $1$.
\end{proof}

For a normal variety $X$ which is projective over a variety $T$, let $K \coloneqq K(T)$ be the field of rational functions on $T$ and $\bar K$ be the algebraic closure of $K$. Set $X_K \coloneqq X \times_T \Spec K$ and $X_{\bar K} \coloneqq X \times_T \Spec \bar K$. The following proposition will be used in the sequel. 

\begin{proposition}[{\cite[Proposition 4.3]{LZ25}}]\label{prop: Generic property}
Let $f: X \to T$ be a fibration with $X$ a $\Qq$-factorial variety. 
\begin{enumerate}
\item There exist natural maps
\[
\begin{split}
\iota_{\bar K}: &N^1(X/T) \to N^1(X_{\bar K}), \quad [D] \mapsto [D_{\bar K}],\\
\iota_K: &N^1(X/T) \to N^1(X_{K}), \quad [D] \mapsto [D_{K}].
\end{split}
\] Moreover, $\iota_K$ is a surjective map.
\item For any sufficiently small open set $U \subset T$, there exists a natural inclusion
\[
N^1(X_U/U) \hookrightarrow N^1(X_{\bar K}), \quad [D] \mapsto [D_{\bar K}].
\]
\end{enumerate}
\end{proposition}

\begin{remark}\label{rmk: iota-K not need fibration}
By the same proof of \cite[Proposition 4.3]{LZ25}, the natural map 
\[
\iota_K: N^1(X/T) \to N^1(X_K), \quad [D] \mapsto [D_K],
\]
is a well-defined surjective linear map for any normal variety $X$ over $T$. Here $X\to T$ is not assumed to be a fibration.
\end{remark}

\subsection{Geometry of convex cones}\label{sec: Geometry of convex cones}

Let $V(\Zz)$ be a lattice of finite rank. Set $V(\Qq) \coloneqq V(\Zz) \otimes_\Zz \Qq$ and $V \coloneqq V(\Qq) \otimes_\Qq \Rr$. Elements of $V(\bQ)$ are called rational points.

If $S \subset V$ is a subset, then $\Conv(S)$ denotes the convex hull of $S$. A polytope $P$ in $V$ is defined as the convex hull of finitely many points in $V$. In particular, a polytope is always closed. If $P$ is the convex hull of finitely many rational points, then $P$ is called a rational polytope. By an open subset of $P$, we mean a subset that is open in the topology induced on $P$ from the minimal affine subspace containing it. We denote by $\Int(P)$ the relative interior of $P$. For a rational point $\De \in P$, by shrinking $P$ around $\De$ we mean replacing $P$ with a sufficiently small rational polytope $P' \subset P$ such that 
\[
P' \supseteq P \cap \mathbb{B}(\De, \ep),
\]
where $\mathbb{B}(\De, \ep)$ is the ball centered at $\De$ of radius $\ep \in \Rr_{>0}$.

A set $C\subset V$ is called a cone if for any $x\in C$ and $\lambda\in \Rr_{>0}$, we have $\lambda \cdot x \in C$. We use $\Int(C)$ to denote the relative interior of $C$. A cone $C$ is said to be relatively open if $C = \Int(C)$. A full-dimensional relatively open cone is simply called an open cone. By convention, the origin is a relatively open cone.  If $S \subset V$ is a subset, then $\Cone(S)$ denotes the closed convex cone generated by $S$. A cone is called a polyhedral cone (resp. rational polyhedral cone) if it is a closed convex cone generated by finitely many points (resp. rational points). As we are only concerned with convex cones in this paper, we also refer to them as cones. A cone $C \subset V$ is non-degenerate if it does not contain an affine line. This is equivalent to saying that its closure $\bar C$ does not contain a non-trivial vector space. A face of a convex cone $C$ is a convex cone $F \subset C$ such that for any closed line segment $I \subset C$ with $I \cap F \neq \emptyset$, either $I \subset F$ or $I \cap F$ consists of exactly one endpoint of $I$. A face of dimension $\dim C - 1$ is called a facet.

In the following, let $\Gamma$ be a group and let $\rho: \Gamma \to \mathrm{GL}(V)$ be a group homomorphism. Then $\Gamma$ acts on $V$ via $\rho$. For $\gamma \in \Gamma$ and $x \in V$, we write $\gamma \cdot x$ or simply $\gamma x$ for this action. For a subset $S \subset V$, set
\[
\Gamma \cdot S \coloneqq \{\gamma \cdot x \mid \gamma \in \Gamma,\, x \in S\}.
\] 
Suppose that this action preserves a convex cone $C$ and the lattice $V(\Zz)$. We further assume that $\dim C = \dim V$. The following definition slightly generalizes \cite[Proposition–Definition~4.1]{Loo14}.

\begin{definition}\label{def: polyhedral type} 
Under the above notation and assumptions, we introduce the following definitions.
\begin{enumerate}
\item Suppose that $C \subset V$ is a convex cone (possibly degenerate). Let
\[
C_+ \coloneqq \Conv(\overline{C} \cap V(\Qq))
\]
be the convex hull of the rational points in $\overline{C}$.

\item We say that $(C_+, \Gamma)$ is \emph{of polyhedral type} if there exists a polyhedral cone $\Pi \subset C_+$ such that $\Gamma \cdot \Pi \supset \Int(C)$.
\end{enumerate}
\end{definition}

\begin{proposition}[{\cite[Proposition-Definition 4.1]{Loo14}}]\label{prop: prop-def}
Under the above notation and assumptions, if $C$ is non-degenerate, then the following conditions are equivalent:
\begin{enumerate}
\item There exists a polyhedral cone $\Pi \subset C_+$ with $\Gamma \cdot  \Pi = C_+$.
\item There exists a polyhedral cone $\Pi \subset C_+$ with $\Gamma \cdot  \Pi \supset \Int(C)$.
\end{enumerate}
Moreover, in case (2), we necessarily have $\Gamma \cdot  \Pi = C_+$.
\end{proposition}

\begin{definition}\label{def: fundamental domain}
Let $\rho: \Gamma \hookrightarrow {\rm GL}(V)$ be an injective group homomorphism and $C \subset V$ be a convex cone. Let $\Pi \subset C $ be a (rational) polyhedral cone. Suppose that $\Gamma$ acts on $C$. Then $\Pi$ is called a weak (rational) polyhedral fundamental domain for $C$ under the action $\Gamma$ if 
\begin{enumerate}
\item $\Gamma \cdot \Pi = C$, and
\item for each $\gamma \in \Gamma$, either $\gamma \Pi = \Pi$ or $\gamma\Pi \cap \Int(\Pi) = \emptyset$. 
\end{enumerate}

Moreover, for $\Gamma_\Pi \coloneqq \{\gamma \in \Gamma \mid \gamma \Pi = \Pi\}$, if $\Gamma_\Pi = \{{\rm id}\}$, then $\Pi$ is called a (rational) polyhedral fundamental domain.
\end{definition}

\begin{lemma}[{\cite[Theorem 3.8 \& Application 4.14]{Loo14}; see also \cite[Lemma 3.5]{LZ25}}]\label{le: existence of fun domain}
Under the notation and assumptions of Definition \ref{def: polyhedral type}, suppose that $\rho: \Gamma \hookrightarrow {\rm GL}(V)$ is injective. Let $(C_+, \Gamma)$ be of polyhedral type with $C$ non-degenerate. Then under the action of $\Gamma$, $C_+$ admits a rational polyhedral fundamental domain.
\end{lemma}

The following proposition is proved in \cite{GLSW26} and remains valid for degenerate cones by the same argument.

\begin{proposition}[{\cite[Proposition 3.6]{GLSW26}}]\label{prop: fundamental domain for surface}
Let $(C_+, \Gamma)$ be a cone of polyhedral type (possibly degenerate). Then for each face $F$ of $C_+$, $(F_+, {\rm Stab}_{F}\Gamma)$ is still of polyhedral type, where 
\[
{\rm Stab}_{F}\Gamma \coloneqq \{\gamma \in \Gamma \mid \gamma F= F\}
\] is a subgroup of $\Gamma$.
\end{proposition}

The following consequence of having a polyhedral fundamental domain is well-known (see \cite[Corollary 4.15]{Loo14}).

\begin{theorem}\label{thm: finite presented}
Let $\rho: \Gamma \hookrightarrow {\rm GL}(V)$ be an injective group homomorphism and $C \subset V$ be a non-degenerate cone. Suppose that $C$ is $\Gamma$-invariant. If $C$ admits a polyhedral fundamental domain under the action of $\Gamma$, then $\Gamma$ is finitely presented. 
\end{theorem}

For a possibly degenerate open convex cone $C$, let $W \subset \bar C$ be the maximal $\Rr$-linear subspace. We say that $W$ is defined over $\Qq$ if $W = W(\Qq) \otimes_\Qq \Rr$, where $W(\Qq) \coloneqq W \cap V(\Qq)$.

\begin{proposition}[{\cite[Proposition 3.8]{LZ25}}]\label{prop: degenerate cone}
Let $(C_+, \Gamma)$ be of polyhedral type, and let $W \subset \bar C$ be the maximal subspace. Suppose that $W$ is defined over $\Qq$. Then there is a rational polyhedral cone $\Pi \subset C_+$ such that $\Gamma \cdot \Pi = C_+$, and for each $\gamma \in \Gamma$, either $\gamma \Pi\cap \Int(\Pi) = \emptyset$ or $\gamma \Pi = \Pi$. That is, $C_+$ admits a weak rational polyhedral fundamental domain under the action of $\Gamma$.
\end{proposition}

\section{Foundations on Morrison-Kawamata dream spaces}\label{sec: MKD spaces}

 \subsection{Definition of Morrison-Kawamata dream spaces}\label{subsec: definition of MKD}

Before giving the definition of Morrison-Kawamata dream spaces, we introduce the following notion, which provides a weak substitute for the Cone Theorem (see \cite[Theorem~3.7]{KM98}). This notion rules out certain pathological phenomena that may occur in the general definition of minimal models.
 
\begin{definition}[Local factoriality of canonical models]\label{def: local factoriality}
Suppose that $X$ is a variety over $T$. Let $\bS \subset {\rm CDiv}(X)_\Rr$ be a set of divisors.  
We say that the local factoriality of canonical models$/T$ holds for $\bS$ if, for any effective $\bR$-Cartier divisor $D \in \bS$ and any rational polytope 
\[
P = \Conv(E_i \mid 1 \leq i \leq m)
\]
such that $D \in P$, where each $E_i$ is an effective$/T$ $\bQ$-Cartier divisor, there exists an open neighborhood $U$ of $D$ (in the topology induced on $P$) such that, for every effective $\Rr$-Cartier divisor $B \in \bS \cap U$, if
\[
f_D \colon X \dashrightarrow Z_D/T 
\quad \text{and} \quad 
f_B \colon X \dashrightarrow Z_B/T
\]
are the canonical models$/T$ of $D$ and $B$, respectively, then there exists a morphism $h \colon Z_B \to Z_D$ such that $f_D = h \circ f_B$.

Similarly, we can define the local factoriality of canonical models$/T$ for a subset $\bS\subset N^1(X/T)$ assuming that effective $\Rr$-Cartier divisors in $\bS$ admit good minimal models/$T$. To be precise, this means that for any effective $\Rr$-Cartier divisor $D$ with $[D] \in \bS$ and any rational polytope 
\[
P = \Conv([E_i] \mid 1 \leq i) \subset N^1(X/T)_\bQ
\]
such that $[D] \in P$, where each $E_i$ is an effective$/T$ $\bQ$-Cartier divisor, there exists an open neighborhood $U$ of $[D]$ (in the topology induced on $P$) such that, for every effective $\Rr$-Cartier divisor $B$ with $[B] \in \bS \cap U$, if
\[
f_D \colon X \dashrightarrow Z_D/T 
\quad \text{and} \quad 
f_B \colon X \dashrightarrow Z_B/T
\]
are the canonical models$/T$ of $D$ and $B$, respectively, then there exists a morphism $h \colon Z_B \to Z_D$ such that $f_D = h \circ f_B$.
\end{definition}

\begin{remark}
Since we assume that effective $\Rr$-Cartier divisors in $\bS$ admit good minimal models, if $B, D$ are effective $\Rr$-Cartier divisors such that $[B]=[D]\in \bS$, then the canonical models$/T$ of $B$ and $D$ coincide. Hence, the above notion is well-defined for subsets of $N^1(X/T)$.
\end{remark}

\begin{remark}\label{rmk: lf is for polyhedral open neighborhood}
We emphasize that the local factoriality of canonical models does not refer to an arbitrary open neighborhood (in the induced topology). Rather, it refers specifically to an open set within an arbitrary polytope generated by effective divisors (see Figure \ref{fig:geometry}). Hence, this requirement is weaker, and the two notions are not the same when an effective divisor lies on the boundary of $\bEff(X/T)$ that is not locally polyhedral.
\end{remark}

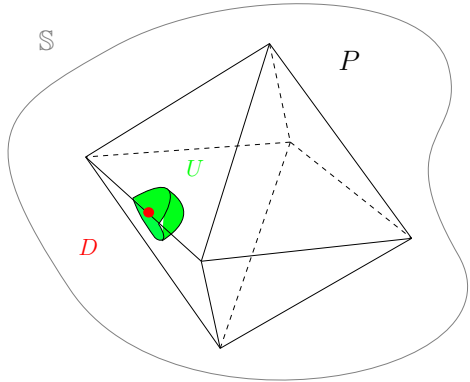
\begin{figure}[H]
\centering
\tikzset{every picture/.style={line width=0.75pt}} %set default line width to 0.75pt        
\resizebox{0.45\textwidth}{!}{
\begin{tikzpicture}[x=0.75pt,y=0.75pt,yscale=-1,xscale=1]
%uncomment if require: \path (0,480); %set diagram left start at 0, and has height of 480

%Straight Lines [id:da30945313208064507] 
\draw  [dash pattern={on 4.5pt off 4.5pt}]  (390.44,188.46) -- (268,197.37) -- (262.61,197.77) -- (148.43,206.08) ;
%Straight Lines [id:da6415579454549147] 
\draw    (285.12,329.44) -- (533.85,302.32) ;
%Straight Lines [id:da9901587950825712] 
\draw  [dash pattern={on 4.5pt off 4.5pt}]  (390.44,188.46) -- (533.85,302.32) ;
%Straight Lines [id:da18739280787180113] 
\draw    (365.79,71.88) -- (285.12,329.44) ;
%Straight Lines [id:da2206496667633301] 
\draw    (365.79,71.88) -- (148.43,206.08) ;
%Straight Lines [id:da7887943069723985] 
\draw  [dash pattern={on 4.5pt off 4.5pt}]  (365.79,71.88) -- (390.44,188.46) ;
%Straight Lines [id:da3211706208622829] 
\draw    (365.79,71.88) -- (533.85,302.32) ;
%Straight Lines [id:da7829132557729483] 
\draw    (148.43,206.08) -- (307.53,432.46) ;
%Straight Lines [id:da48206473861013655] 
\draw    (307.53,432.46) -- (533.85,302.32) ;
%Straight Lines [id:da5842075278578367] 
\draw    (285.12,329.44) -- (307.53,432.46) ;
%Straight Lines [id:da20496503051062143] 
\draw  [dash pattern={on 4.5pt off 4.5pt}]  (390.44,188.46) -- (307.53,432.46) ;
%Shape: Free Drawing [id:dp3887103969396517] 
\draw  [line width=0.75] [line join = round][line cap = round] (320.97,216.92) .. controls (320.97,216.92) and (320.97,216.92) .. (320.97,216.92) ;
%Shape: Free Drawing [id:dp5076995909234465] 
\draw  [color={rgb, 255:red, 128; green, 128; blue, 128 }  ,draw opacity=1 ][line width=0.75] [line join = round][line cap = round] (230.33,426) .. controls (230.33,426) and (230.33,426) .. (230.33,426) ;
%Curve Lines [id:da6202729638081255] 
\draw [fill={rgb, 255:red, 0; green, 255; blue, 24 }  ,fill opacity=1 ]   (242.83,243.02) .. controls (262.15,254.96) and (281.48,277.84) .. (238.62,305.29) ;
%Curve Lines [id:da3667775157805686] 
\draw [fill={rgb, 255:red, 0; green, 255; blue, 24 }  ,fill opacity=1 ]   (204.17,254.71) .. controls (272.24,216.33) and (244.51,276.31) .. (234.42,284.96) ;
%Curve Lines [id:da32511814712754117] 
\draw [fill={rgb, 255:red, 0; green, 255; blue, 67 }  ,fill opacity=1 ]   (204.17,254.71) .. controls (213.42,280.89) and (225.18,293.6) .. (229.38,300.21) .. controls (233.58,306.82) and (249.55,308.85) .. (234.42,284.96) ;
%Straight Lines [id:da6581973007209233] 
\draw    (148.43,206.08) -- (285.12,329.44) ;
%Shape: Polygon Curved [id:ds9739036537331628] 
\draw  [color={rgb, 255:red, 128; green, 128; blue, 128 }  ,draw opacity=1 ] (181.23,75.95) .. controls (315.07,1.93) and (551.98,7.45) .. (576.71,96.64) .. controls (601.45,185.82) and (504.23,175.49) .. (588.85,321.85) .. controls (673.47,468.21) and (252.78,546.29) .. (130.27,363.22) .. controls (7.76,180.14) and (47.38,149.98) .. (181.23,75.95) -- cycle ;
%Shape: Ellipse [id:dp30216398256298693] 
\draw  [color={rgb, 255:red, 255; green, 0; blue, 0 }  ,draw opacity=1 ][fill={rgb, 255:red, 255; green, 0; blue, 0 }  ,fill opacity=1 ] (217.13,271.6) .. controls (217.13,268.55) and (219.74,266.08) .. (222.96,266.08) .. controls (226.17,266.08) and (228.78,268.55) .. (228.78,271.6) .. controls (228.78,274.64) and (226.17,277.11) .. (222.96,277.11) .. controls (219.74,277.11) and (217.13,274.64) .. (217.13,271.6) -- cycle ;

% Text Node
\draw (138.62,301.1) node [anchor=north west][inner sep=0.75pt]  [font=\LARGE,color={rgb, 255:red, 255; green, 0; blue, 0 }  ,opacity=1 ]  {$D$};
% Text Node
\draw (265.61,208.17) node [anchor=north west][inner sep=0.75pt]  [font=\LARGE,color={rgb, 255:red, 0; green, 255; blue, 6 }  ,opacity=1 ]  {$U$};
% Text Node
\draw (445.99,78.46) node [anchor=north west][inner sep=0.75pt]  [font=\Huge]  {$P$};
% Text Node
\draw (91.73,67.53) node [anchor=west] [inner sep=0.75pt]  [font=\Huge,color={rgb, 255:red, 128; green, 128; blue, 128 }  ,opacity=1 ]  {$\mathbb{S}$};

\end{tikzpicture}}
\caption{Neighborhood of $D$ in the context of the local factoriality of canonical models}
  \label{fig:geometry}
\end{figure}

Recall the definition of MKD fiber spaces given in the introduction.

\MKDspaces*

\begin{remark}
The terminology fiber space in the notion of an MKD fiber space is slightly more general than that used in \cite{LZ25, Li26}, where it always refers to a fibration. However, one can always pass to the Stein factorization to eliminate this difference.
\end{remark}

\begin{remark}
In Definition~\ref{def: MKD spaces}~(4), we may assume the local factoriality of canonical models only on $\Pi$. Corollary \ref{cor: local factoriality for Pi} shows that, at least in the absolute setting (in fact, it suffices for $\Eff(X/T)$ to be non-degenerate), the local factoriality of canonical models on $\Eff(X)$ can be deduced from that of $\Pi$. To establish this implication in general, one needs to prove the Siegel property (see \cite[Theorem~3.8]{Loo14}) for degenerate cones, which is currently unavailable. Moreover, there are also possible ways to slightly weaken the assumptions in Definition~\ref{def: MKD spaces} (see Remark~\ref{rmk: weaken assumption for MKD}).
\end{remark}

\begin{remark}\label{rmk: not for codition (2)}
It is possible that for two numerically equivalent divisors, one admits a good minimal model but the other does not. See \cite[Example 3.10]{KKL16} for such examples. Thus, in Definition \ref{def: MKD spaces}~(2), we require that the good minimal model exists for a divisor instead of for an arbitrary divisor in its numerical class.
\end{remark}

\begin{remark}
The notion of MKD fiber spaces still makes sense if one considers the log pair $(X/T, \Delta)$ and replaces $\PsAut(X/T)$ by $\PsAut(X/T,\Delta)$. With this modification, the discussion in the paper remains valid after making appropriate (and simple) changes.
\end{remark}

The following are immediate properties of MKD fiber spaces.

\begin{proposition}\label{prop: property of MKD from definition}
Let $X/T$ be an MKD fiber space. 
\begin{enumerate}
\item We have $\Mov(X/T) = \bMov^e(X/T)$.
\item If $\bMov(X/T)$ is non-degenerate, then we have 
\[
\Mov(X/T) =\bMov^e(X/T) = \Mov(X/T)_+.
\] 
\end{enumerate}
\end{proposition}
\begin{proof}
By assumption, an effective divisor $D$ with $[D] \in \bMov^e(X/T)$ admits a good minimal model$/T$, which is isomorphic to $X$ in codimension $1$ by Lemma \ref{lem: movable give iso in codim 1}. Thus the natural inclusion $\Mov(X/T) \subset \bMov^e(X/T)$ is indeed an equality. This shows (1).

If $\bMov(X/T)$ is non-degenerate, then by Proposition~\ref{prop: prop-def} and Definition~\ref{def: MKD spaces}~(3), we have 
\[
\PsAut(X/T) \cdot \Pi = \Mov(X/T)_+.
\] Hence, the natural inclusions
\[
\Mov(X/T) \subset \bMov^e(X/T) \subset \Mov(X/T)_+
\] 
are equalities. (However, if $\bMov(X/T)$ is degenerate, it may happen that $\Mov(X/T) \subsetneqq \Mov(X/T)_+$. For example, see \cite[Example~3.8 (2)]{Kaw97}.)
\end{proof}

\subsection{Local factoriality for canonical models}

The local factoriality condition for canonical models requires some clarification. We discuss why this condition is natural from the viewpoint of standard conjectures in birational geometry.

Let $\mathcal C \subset {\rm CDiv}(X)_{\bR}$ be a cone. Set
\[
R(X/T, \mathcal C) \coloneqq \bigoplus_{D \in \mathcal C \cap {\rm CDiv}(X)} f_*\mO(D)
\] as a sheaf of $\mO_T$-algebras. If $\mathcal C$ is the cone generated by effective $\bQ$-Cartier divisors $D_1, \cdots, D_l$, then we also write $R(X/T, D_1, \cdots, D_l)$ for $R(X/T, \mathcal C)$.

\begin{proposition}\label{prop: condition (4)}
Let $X$ be a variety over $T$.
\begin{enumerate}
\item If $\bS \subset \Nef(X/T)$, then the local factoriality of canonical models$/T$ holds for $\bS$.
\item  Assume that every effective $\bR$-Cartier divisor on $X/T$ admits a good minimal model$/T$. Then the following hold:
\begin{enumerate}
\item If $\mathcal C \subset {\rm CDiv}(X)_{\bR}$ is a cone generated by finitely many effective$/T$ $\bQ$-Cartier divisors and $R(X/T, \mathcal C)$ is a finitely generated sheaf of $\Oo_T$-algebras, then the local factoriality of canonical models$/T$ holds for $\mathcal C$.
\item If $X/T$ is a fiber space of Calabi-Yau type, then the local factoriality of canonical models$/T$ holds for any subset $\bS \subset {\rm CDiv}(X)_{\bR}$.
\end{enumerate}
\item Suppose that $h: X \dto Y/T$ is a small $\bQ$-factorial modification between $\bQ$-factorial varieties. If local factoriality of canonical models$/T$ holds for $\bS \subset {\rm CDiv}(X)_\bR$, then it also holds for $h_*\bS \subset {\rm CDiv}(Y)_\bR$.
\end{enumerate}
\end{proposition}
\begin{proof}
For (1), note that, by the definition of local factoriality of canonical models$/T$ for a subset $\bS\subset N^1(X/T)$, every effective divisor $D$ with $[D]\in\bS$ admits a good minimal model$/T$. Suppose that $[D]\in P$, where $P$ is a rational polytope given as the convex hull of finitely many effective divisors$/T$. Then $[D]$ lies in the relative interior of a face $F$ of $P$. Hence, the canonical model of $D$ corresponds to the contraction of curves intersecting trivially with $B$ such that $[B]\in F$. Shrink $P$ around $[D]$, we can assume that any face of $P$ containing $[D]$ must contain $F$ as a face. Therefore, locally around $[D]$, any canonical model must contract curves intersecting trivially with $B, [B]\in F$. Therefore, every canonical model of a divisor in $\bS \cap P$ maps to the canonical model of $D$.

\medskip

(2) (a) can be proved by the same argument as \cite[Theorem 4.2]{KKL16}, and we reproduce here for the reader's convenience. By  \cite[Theorem 3.2]{KKL16}, 
\[
\Supp\mathfrak R \coloneqq \{B \in \mathcal C \mid |B/T|_\bR \neq \emptyset\} \subset  {\rm CDiv}(X/T)_{\bR}\] is a union of finitely many rational polyhedral cones $\mathcal C_i, 1 \leq i \leq m$ such that  for any geometric valuation $\Gamma$ over $X$, the function
\[
\sigma_\Gamma: B \to \inf\{{\rm mult}_\Gamma B' \mid B' \in |B/T|_{\bR}\}
\] is linear on each $\mathcal C_i$. Moreover, there is a positive integer $d$ and a resolution $\pi: \ti X \to X$ such
that the movable$/T$ part of the divisor ${\rm Mob}(\pi^*(dB))$ is base-point free$/T$ for every $B\in \Supp\mathfrak R \cap  {\rm CDiv}(X)$, and ${\rm Mob}(\pi^*(kdB))=k{\rm Mob}(\pi^*(dB))$ for every positive integer $k$. Therefore, it suffices to work in a fixed rational polyhedral cone $\mathcal C_i$. By the above property, 
\[
\mathcal M_i \coloneqq \Cone \{{\rm Mob}(\pi^*(dB))\mid B\in \mathcal C_i \cap  {\rm CDiv}(X)\}
\] is a rational polyhedral cone. Moreover, for any $B\in \mathcal C_i \cap  {\rm CDiv}(X)$, the natural map
\[
X \dto \ti X \to Z_{B}/T
\] is the canonical model of $B$, where $\ti X \to Z_B$ is the contraction morphism induced by the base-point free$/T$ divisor ${\rm Mob}(\pi^*(dB))$. Thus, locally around any $\bR$-Cartier divisor $D\in \mathcal C_i$,  we can assume that $D$ lies in the interior of a face whose dimension is minimal among all the faces of $\mathcal C_i$. This implies that, locally around any rational polytope $P$ which is the convex hull of effective$/T$ divisors such that $D\in P$, every canonical model for a divisor in $\mathcal C \cap P$ maps to the canonical model of $D$.
\medskip

For (2) (b), since $X$ is of Calabi-Yau type over $T$, there exists a divisor $\Delta$ such that $(X, \Delta)$ is a klt pair with $K_X + \Delta \sim_\bR 0/T$. Let $D \in \bS$ be an effective divisor, and let $P = \Conv(E_i \mid 1 \le i \le l)$ be a rational polytope with $E_i, 1 \leq i \leq l$ effective$/T$ $\bQ$-Cartier divisors such that $D \in P$. As the canonical model$/T$ of $D$ coincides with that of $\epsilon D$ for any $\epsilon > 0$, after rescaling $D$ and $E_i$, we may assume that each pair $(X, \Delta + E_i)$ has klt singularities. Note that the canonical model$/T$ for $D$ is the same as that for $K_X + \Delta + D$. On the other hand, by the assumption on the existence of good minimal models,
\[
R(X/T, K_X + \Delta + E_1, \dots, K_X + \Delta + E_l)
\]
is a finitely generated sheaf of $\Oo_T$-algebras by \cite[Theorem~8.10]{DHP13}. Hence, (2) (b) follows from (2) (a).

\medskip

For (3), it suffices to show that if $g: X \dto Z/T$ is the ample model$/T$ of $D$, then $g\circ h^{-1}: Y \dto Z/T$ is also the ample model$/T$ of $D_Y \coloneqq h_*D$. By Definition \ref{def: minimal model}, $Z$ is normal and projective$/T$, and there exists an ample$/T$ divisor $H$ on $Z$ with the following property. If $p\colon W\to X$ and $q\colon W\to Z$ resolve $g$, then $p$ is a birational contraction morphism and
\[
p^*D\sim_{\bR} q^*H+E/T,
\]
where $E\geq 0$, and every divisor $B\in |p^*D/T|_\bR$ satisfies $B\geq E$. Note that this property holds for $W$ if and only if it holds for any birational model higher than $W$. Hence, replacing $W$ by a higher model, we may assume that there is a birational morphism $r: W \to Y$ such that $h=r \circ p^{-1}$. As $h$ is isomorphic in codimension $1$ between $\bQ$-factorial varieties, there exist effective $p$-exceptional (and hence also $r$-exceptional) divisors $F_1, F_2$ such that
\[
p^*D + F_1 = r^*D_Y + F_2.
\] Hence, we have
\[
r^*D_Y \sim_\Rr q^*H+E+F_1-F_2/T.
\] As $-(E+F_1-F_2)$ is nef over $Y$ and $r_*(E+F_1-F_2)=r_* E \geq 0$, we have $E+F_1-F_2 \geq 0$ by the negativity lemma. Moreover, if $B' \in |r^*D_Y/T|_\bR$, then
\[
B'+F_2-F_1 \sim_\bR p^*D/T.
\] Thus we have $B'+F_2-F_1 = p^*(p_*(B'+F_2-F_1))$ by the negativity lemma. As $p_*(B'+F_2-F_1)=p_*B' \geq 0$, we have $(B'+F_2-F_1) \in |p^*D/T|_\bR$. Hence, we have $(B' + F_2 - F_1) \geq E$. This implies that $B' \geq E + F_1 - F_2$. Therefore, $g \circ h^{-1}$ is the ample model$/T$ of $D_Y$.
\end{proof}

As a direct corollary of Proposition~\ref{prop: condition (4)}, we obtain the following result. For the definition of relative Mori dream spaces, see \cite{Oht22}.

\begin{corollary}\label{cor: Mori dream space, CY are MKD space}
Let $X$ be a $\bQ$-factorial variety over $T$. Then $X/T$ is an MKD fiber space in each of the following cases:
\begin{enumerate}
\item $X/T$ is a relative Mori dream space (in particular, $X$ is of Fano type over $T$).
\item $X$ is of Calabi-Yau type over $T$ such that the Morrison-Kawamata cone conjecture holds for $X/T$, and every effective $\bR$-Cartier divisor on $X$ admits a good minimal model$/T$.
\end{enumerate}
\end{corollary}

\subsection{Shokurov polytopes of minimal models}\label{subsec: shokurov polytope}

As a cornerstone for later applications, we establish the following result concerning Shokurov polytopes (that is, the chamber decomposition of minimal models). Results of this type for ordinary log pairs were established in \cite{Sho96, Cho08, SC11}. The proof of Theorem \ref{thm: Shokurov poly for minimal models} combines \cite[Theorems~2.4 and~2.6]{LZ25} with a simplification of the argument given in \cite[Theorem~2.4]{HPX24}. The proof follows the ideas of \cite[Lemma~7.1]{BCHM10}.

We first prepare several auxiliary lemmas.

\begin{lemma}\label{lem: factor bir contraction}
Assume that $X/T$ is a $\bQ$-factorial variety such that every effective $\bQ$-Cartier divisor on $X/T$ admits a minimal model$/T$. If $f: X \dto Y/T$ is a birational contraction, then $f$ factors into a small $\bQ$-factorial modification $h: X \dto X'/T$ followed by a birational morphism $g: X' \to Y/T$.
\end{lemma}
\begin{proof}
Let $A$ be an ample$/T$ divisor on $Y$ and let $A_X$ be the strict transform of $A$ on $X$. Then some positive multiple of $A_X$ is a movable divisor over $T$. By assumption, let $h: X \dto X'/T$ be a minimal model$/T$ of $A_X$ such that $X'$ is $\bQ$-factorial. Then $h$ is isomorphic in codimension $1$ by Lemma \ref{lem: movable give iso in codim 1}. 

Let $p: W \to X, q: W \to Y, r: W \to X'$ be projective birational morphisms such that $h=r \circ p^{-1}$ and $f=q \circ p^{-1}$. Since $A_{X'} = r_*(q^*A)$, we have
\[
r^*A_{X'} = q^*A + \tilde E,
\]
where $\tilde E$ is $r$-exceptional (and hence also $q$-exceptional). Since $A_{X'}$ is nef$/T$, $g\coloneqq q \circ r^{-1}: X' \dto Y$ is a morphism by Lemma \ref{lem: common mm}~(2). Hence, we have $f= g \circ h$.
\end{proof}

\begin{lemma}\label{lem: birational model preserves good minimal model and local factoriality}
Let $X/T$ be a $\Qq$-factorial variety and  $P \subset \Eff(X/T)$ be a rational polyhedral cone. Assume that
\begin{enumerate}
\item every effective $\bR$-Cartier divisor in $P$ admits a good minimal model over $T$, and
\item $P$ satisfies the local factoriality of canonical models$/T$.
\end{enumerate}
Let $f: X \dto Y/T$ be a good minimal model of an effective $\bR$-Cartier divisor $D$ with $[D] \in P$. Then, after shrinking $P$ around $[D]$, the properties (1) and (2) above are preserved for $P_Y \coloneqq f_*P \subset \Eff(Y/T)$.
\end{lemma}

\begin{proof}
After shrinking $P$ around $[D]$, we claim that if $g: X \dto X'/T$ is a good minimal model$/T$ of $B$ with $[B]\in P$, then the natural map $\tau=g\circ f^{-1}: Y \dto X'/T$ is a good minimal model$/T$ of $B_Y\coloneqq f_*B$. 

Let $p: W \to X, q: W \to Y$ be projective birational morphisms such that $f=q \circ p^{-1}$. By assumption, we have 
\[
p^*D=q^*D_Y+E,
\] where $D_Y=f_*D$ and $E \geq 0$ is a $q$-exceptional divisor containing the divisors contracted by $f$. After shrinking $P$ around $[D]$, we may assume that for each $[B] \in P$, we have
\[
p^*B + E(B)^- = q^*B_Y + E(B)^+,
\]
where $E(B)^- \geq 0$ is a $p$-exceptional divisor and $E(B)^+ \geq 0$ is a $q$-exceptional divisor such that $E(B)^-$ and $E(B)^+$ have no common irreducible components, and $\Supp E(B)^+$ contains all divisors contracted by $f$. Replacing $W$ by a higher model, we can assume that there exists a birational morphism $r: W \to X'$ such that $g=r\circ p^{-1}$.

We claim that $g: X \dto X'/T$ contracts $\Exc(f)$. As $g$ is a good minimal model$/T$ of $B$, we have
\[
p^*B=q^*B_Y+E(B)^+-E(B)^-=r^*B_{X'}+F,
\] where $B_{X'}=g_*B$ and $F \geq 0$ is an $r$-exceptional divisor which contains the divisors contracted by $g$. As $B_{X'}$ is nef$/T$ and $E(B)^-$ is also $r$-exceptional, we see that $F+E(B)^-$ is the negative part of the Nakayama-Zariski decomposition$/T$ of 
\[
p^*B+E(B)^-=q^*B_Y+E(B)^+=r^*B_{X'}+F+E(B)^-.
\] Therefore, we have $\Supp (E(B)^+) \subset \Supp (F+E(B)^-)$. This implies
\[
\Supp(E(B)^+) \subset \Supp F,
\]
since $E(B)^-$ and $E(B)^+$ have no common irreducible components. As $F$ is $r$-exceptional, we see that $\Supp(\Exc(f)) \subset \Supp (E(B)^+)$ is also $r$-exceptional. Therefore, $g$ contracts $\Exc(f)$. As
\[
q_*\left(F-(E(B)^+-E(B)^-)\right)=q_*F \geq 0, 
\] we have $F-(E(B)^+-E(B)^-) \geq 0$ by the negativity lemma. As $\Supp F$ contains the divisor contracted by $\tau$ and $(E(B)^+-E(B)^-)$ is $q$-exceptional, we see that $F -(E(B)^+-E(B)^-)$ still contains the divisor contracted by $\tau$. This shows that $\tau$ is a good minimal model of $B_Y$.

By the above discussion, if $f_B: X \dto X' \to Z_B$ is the canonical model of $B$ over $T$, then 
\[
f_B \circ \tau: Y \dto X' \to Z_B
\] is the canonical model of $B_Y$ over $T$. Thus, the local factoriality of canonical models$/T$ holds for $P_Y$ as the local factoriality of canonical models$/T$ holds for $P$.
\end{proof}

The following lemma upgrades Lemma~\ref{lem: birational model preserves good minimal model and local factoriality} from local factoriality over $T$ to local factoriality over the canonical model of $D$.

\begin{lemma}\label{lem: inductive step}
Let $X/T$ be a $\Qq$-factorial variety and  $P \subset \Eff(X/T)$ be a rational polyhedral cone. Assume that
\begin{enumerate}
\item every effective $\bR$-Cartier divisor in $P$ admits a good minimal model over $T$, and
\item $P$ satisfies the local factoriality of canonical models$/T$.
\end{enumerate}
Let $f: X \dto Y/T$ be a good minimal model for an effective $\bR$-Cartier divisor $D$ with $[D] \in P$. If $Y \to Z/T$ is the contraction morphism induced by $f_*D$, then, after shrinking $P_Y \coloneqq f_*P$ around $[f_*D]$,
\begin{enumerate}
\item every effective $\bR$-Cartier divisor $B$ with $[B]\in P_Y$ admits a good minimal model over $Z$, and
\item $P_Y$ satisfies the local factoriality of canonical models$/Z$.
\end{enumerate}
\end{lemma}
\begin{proof}
By Lemma~\ref{lem: birational model preserves good minimal model and local factoriality}, since $P$ is a polytope, the local factoriality of canonical models$/T$ implies that, after shrinking $P$ around $[D]$, the canonical model$/T$ of any effective $\bR$-Cartier divisor $B$ with $[B]\in P$ maps to $Z/T$. Hence, after this shrinking, $P$ satisfies the local factoriality of canonical models$/Z$.

By the proof of Lemma \ref{lem: birational model preserves good minimal model and local factoriality}, after shrinking $P$ around $[D]$, if $g: X \dto X'/T$ is a good minimal model of $B$ over $T$, then the natural map $\tau=g\circ f^{-1}: Y \dto X'/T$ is a good minimal model of $B_Y$ over $T$. Therefore, $\tau$ is also a good minimal model of $B_Y$ over $Z$. Moreover, $f_*P$ satisfies the local factoriality of canonical models$/Z$ as $P$ satisfies the local factoriality of canonical models$/Z$.
\end{proof}

\begin{remark}
Compared with Lemma \ref{lem: birational model preserves good minimal model and local factoriality}, in the conclusion (2) of Lemma \ref{lem: inductive step}, the local factoriality of canonical models is over $Z$ instead of $T$.
\end{remark}

With the above preparations, we are now ready to prove Theorem~\ref{thm: Shokurov poly for minimal models}.

\begin{proof}[Proof of Theorem \ref{thm: Shokurov poly for minimal models}]
Let $P\subset {\rm CDiv}(X)_\Rr$ be a rational polytope such that $\Cone(P)=\mathcal C$. It suffices to show that $P$ can be written as a disjoint union of finitely many relatively open rational polytopes
\[
P = \bigsqcup_{i=0}^{m} P_i,
\]
such that for any effective $\bR$-Cartier divisors $B, D \in P_i$, whenever $X \dto Y/T$ is a weak minimal model of $B$, it is also a weak minimal model of $D$. Indeed, the desired decomposition can be taken as $\mathcal C_i = \Cone(P_i)\setminus \{0\}$ for $1 \leq i \leq m$, together with an extra cone $\mathcal C_m = \{0\}$.

We proceed by induction on the dimension of $P$. The statement trivially holds if $\dim P=0$. Hence, we can assume that it holds for any polytope of dimension $\leq d$. Assume that we have $\dim P=d+1$.

\medskip

\noindent Step 1. If there exists some $\De_0 \in P$ such that $\De_0 \equiv 0/T$, then we show the claim. Let $P'$ be the union of facets of $P$. By the induction hypothesis, $P'= \sqcup_j \ti Q^\circ_j$ is a finite union such that each $\ti Q^\circ_j$ is a relatively open rational polytope, and for $B, D \in \ti Q^\circ_j$, if $X \dto Y/T$ is a weak log canonical model of $B$, then it is also a weak log canonical model of $D$. Note that for any facet $F$ of $P$, each $\ti Q_j^\circ$ either lies entirely in $F$ or is disjoint from $F$. For  $t, t'\in(0,1]$, we have
\[
tB+(1-t)\De_0 \equiv tB/T,\quad t'D+(1-t')\De_0 \equiv t'D/T.
\] Hence, $X \dto Y/T$ is a weak log canonical model of $tB+(1-t)\De_{0}$ if and only if it is a weak log canonical model of $B$ if and only if it is a weak log canonical model of $D$  if and only if it is a weak log canonical model of $t'D+(1-t')\De_{0}$. Therefore, if $\De_0 \in \Int(P)$, then the decomposition
\[
P=\left(\bigsqcup_j \ti Q_j^\circ\right) \bigsqcup \left(\bigsqcup_j \Int(\Conv(\ti Q_j^\circ, \De_0))\right) \bigsqcup \{\De_0\}
\] satisfies the claim. If $\De_0$ lies on the boundary of $P$, define
\[
P'' \coloneqq \bigcup_{\substack{\De_0\in F \\ F\subset P \text{~is a facet}}} F
\]
to be the union of facets of $P$ that contain $\De_0$.  Then the decomposition
\[
P=\left(\bigsqcup_{j} \ti Q_j^\circ\right) \bigsqcup \left(\bigsqcup_{\ti Q_j^\circ \not\subset P''} \Int(\Conv(\ti Q_j^\circ, \De_0))\right)
\] satisfies the claim.

\medskip

 \noindent Step 2. We show the general case. By the compactness of $P$, it suffices to show the result locally around any point $\De_0\in P$. 

Let $h: X \dto X'/T$ be a good minimal model of $\De_0$. Then there exist a contraction $\pi: X' \to Z'/T$ and an ample$/T$ $\Rr$-Cartier divisor $A$ on $Z'$ such that $\De_0' \sim_\Rr \pi^*A$, where $\De_0'$ is the strict transform of $\De_0$. In particular, we have $\De_0' \equiv 0/Z'$. 

Let $P'$ be the image of $P$ under the natural map ${\rm CDiv}(X)_\Rr \to {\rm CDiv}(X')_\Rr$. By Lemma \ref{lem: inductive step}, after shrinking $P'$, we can assume that every effective $\bR$-Cartier divisor $B\in P'$ admits a good minimal model over $Z'$, and
$P'$ satisfies the local factoriality of canonical models over $Z'$. By Step~1, there exists a finite disjoint union of relatively open rational polytopes
\[
P'=\bigsqcup_j P'_j
\]
such that the following property holds. For any effective $\bR$-Cartier divisors $B$ and $D$ with $[B],[D]\in P'_j$, if $X'\dto Y/Z'$ is a weak minimal model$/Z'$ of $B$, then it is also a weak minimal model$/Z'$ of $D$. We emphasize that the weak minimal model is over $Z'$ instead of $T$. We claim that they are also weak minimal models over $T$ after shrinking $P'$. It suffices to work with a fixed $P_j'$ containing $\De_0'$. 

Let $B$ be an effective $\bR$-Cartier divisor with $B \in P_j'$. Suppose that $h': X' \dto Y/Z'$ is a weak minimal model$/Z'$ of $B$. Let $\mu: Y \to Z_{B}$ be the canonical model$/Z'$ which is induced by the semi-ample divisor $h'_*{B}$. Thus, there exists an ample$/Z'$ $\bR$-Cartier divisor $H$ on $Z_{B}$ such that $h_*'B\sim_\bR \mu^*H$. 
Let $\theta: Z_{B} \to Z'$ be the natural morphism such that $\pi=\theta \circ \mu \circ h'$.
For $\lambda \in [0,1]$, we have
\[
h_*'(\lambda \De_0 + (1-\lambda) B)\sim_\Rr \mu^*(\lambda \theta^*A+(1-\lambda) H).
\] As $A$ is ample$/T$ and $H$ is ample$/Z'$, there exists a positive rational number $r<1$ such that whenever $\lambda \in [r,1]$, the divisor $\lambda \theta^*A+(1-\lambda) H$ is nef$/T$. That is, $h': X' \dto Y$ is also a weak minimal model of $\lambda \De_0 + (1-\lambda) B'$ over $T$. By Lemma \ref{lem: common mm} (1), not only $h'$, but also any weak minimal model of $\lambda \De_0 + (1-\lambda) B'$ over $Z'$ is a weak minimal model of $\lambda \De_0 + (1-\lambda) B'$ over $T$.

Take $B'$ to be a vertex of $\bar P_j'$ (here $\bar P_j'$ denotes the closure of $P_j'$), and replace $B'$ by $r\De_0 + (1-r)B'$ as above. Replace $P_j'$ by the relatively open polytope generated by these $B'$. The above argument shows that for any divisor $\Theta \in P_j'$, $h_*'\Theta$ is nef over $T$. Hence, $h': X' \dto Y$ is also a weak minimal model of $\Theta$ over $T$. By Lemma \ref{lem: common mm} (1) again, not only $h'$, but also any weak minimal model of $\Theta$ over $Z'$ is also a weak minimal model of $\Theta$ over $T$. This completes the inductive step.
\end{proof}

Theorem~\ref{thm: Shokurov poly for nef}, which gives the chamber structure for nef cones, is a consequence of Theorem~\ref{thm: Shokurov poly for minimal models}.

\begin{proof}[Proof of Theorem \ref{thm: Shokurov poly for nef}]
    Let $\mathcal P = \sqcup_{i=0}^m \mathcal P_i$ be a decomposition into finitely many relatively open rational polyhedral cones satisfying Theorem \ref{thm: Shokurov poly for minimal models}. We claim that if there exists a nef$/T$ divisor $D$ such that $D \in \mathcal{P}_i$, then $\bar{\mathcal{P}_i} \subset \mathcal{N}_{\mathcal P}$, where $\bar{\mathcal{P}_i}$ is the closure of $\mathcal P_i$. 
    
   By definition, ${\rm Id}: X \to X/T$ is a minimal model$/T$ of $D$. Hence, we have $\mathcal P_i \subset \mathcal{N}_{\mathcal P}$ by the property of $\mathcal P_i$ in Theorem \ref{thm: Shokurov poly for minimal models}. This also implies that $\bar{\mathcal{P}_i} \subset \mathcal N_P$. Therefore, we have
    \[
     \mathcal N_P=\Cone \{\bar{\mathcal{P}_i} \mid {\mathcal{P}_i} \text{~contains a nef$/T$ divisor}\},
    \] which is a rational polyhedral cone.
\end{proof}

\subsection{Fundamental domains for nef cones and effective cones}\label{subsec: cone conj for nef and eff cones}

For varieties of Calabi-Yau type, assuming the good minimal model conjecture, the movable cone conjecture implies the nef cone conjecture, and the effective cone conjecture is equivalent to the movable cone conjecture. The former implication was established independently in \cite{Xu24} and \cite{GLSW26}, while the latter equivalence was proved in \cite{GLSW26}. We extend these results to MKD fiber spaces.

The following can be proved by a similar argument as \cite[Proposition 3.8]{Li26}.

\begin{proposition}\label{prop: max vector defined over Q}
Suppose that $E$ and $M$ are the maximal linear subspaces contained in $\bEff(X/T)$ and $\bMov(X/T)$, respectively. Then the following hold.
\begin{enumerate}
\item $E$ is defined over $\Qq$.
\item Assume that $X$ is $\bQ$-factorial and every effective $\bR$-Cartier divisor admits a good minimal model$/T$. Then $M$ is defined over $\Qq$.
\end{enumerate}
\end{proposition}
\begin{proof}
Let $X \to S/T$ be the Stein factorization of $X \to T$. Then we have the natural map
\[
N^1(X/S) \to N^1(X/T), \quad [D] \mapsto [D]
\] which is an isomorphism as $S \to T$ is a finite map. Moreover, under this map, we have $\Eff(X/S) \simeq \Eff(X/T)$ and $\Mov(X/S) \simeq \Mov(X/T)$. Hence, replacing $X \to T$ by its Stein factorization, we can assume that $X \to T$ is a fibration. 

For (1), by Proposition \ref{prop: Generic property} (1), there exists a natural map 
\[
\iota_{\bar\eta}: N^1(X/T) \to N^1(X_{\bar\eta}), \quad [D] \mapsto [D_{\bar\eta}],
\] where $\bar\eta$ is the geometric generic point of $T$. We claim that $\Ker (\iota_{\bar\eta}) = E$. Hence, $E$ is defined over $\Qq$ as $\iota_{\bar\eta}$ is defined over $\Qq$. Indeed, for $[D] \in \Ker (\iota_{\bar\eta})$, we have $D_{\bar\eta} \equiv 0$ on $X_{\bar\eta}$, hence $D_{\bar\eta}+tA_{\bar\eta}$ is big for any ample$/T$ divisor $A$ on $X$ and $t \in \Rr_{>0}$. Thus, $D + tA$ is also big over $T$. Taking $t \to 0$, we obtain $[D] \in \bEff(X/T)$. For the same reason, $[-D] \in \bEff(X/T)$. Hence, we have $[D] \in E$. Conversely, if $[D] \in E$, then $\pm [D_{\bar\eta}] \in \bEff(X_{\eta})$. Elements in $\bEff(X_{\bar\eta})$ intersect with movable curves non-negatively (see \cite{BDPP13}). As movable curves form a full dimensional cone in $N_1(X_{\bar\eta})$, we have $D_{\bar\eta} \equiv 0$.

\medskip

For (2), suppose that $X'/T$ is $\Qq$-factorial and $X \dto X'/T$ is isomorphic in codimension $1$. If $D$ is a divisor on $X$, then let $D'$ be the strict transform of $D$ on $X'$. If $C$ is an irreducible curve on $X'$, then we say that its class covers a divisor if the Zariski closure 
\[
\overline{\bigcup_{[C']=[C] \in N_1(X'/T)} C'}
\] in $X'$  has codimension $\leq 1$. 

We claim that 
\begin{equation}\label{eq: 1}
\begin{split}
M = E ~\cap ~&\{[D] \mid D'\cdot C=0, \text{~where~} X \dto X'/T \text{~is isomorphic in codimension~}1, \\
&\qquad X' \text{~is~} \Qq\text{-factorial, and~} C \text{~is a curve whose class covers a divisor}\}.
\end{split}
\end{equation}

For ``$\subset$'', if $[D] \in M$, then $\pm [D'] \in \bMov(X'/T)$, and thus $\pm D' \cdot C \geq 0$ for any curve $C$ whose class covers a divisor in $X'$. Hence, we have $D' \cdot C=0$.

For ``$\supset$'', let $[D]$ belong to the right-hand side of \eqref{eq: 1}. As $[D] \in E$ is an $\bR$-Cartier divisor, $[D+A]\in \Eff(X/T)$ for any ample$/T$ divisor $A$. Let $B= D+A$. By assumption, let $f: X \dto Y/T$ be a good minimal model of $B$. By Lemma \ref{lem: factor bir contraction}, there exists a small $\bQ$-factorial modification $h: X \dto X'$ and a contraction morphism $g: X' \to Y$ such that $f=g \circ h$. 

We claim that $f$ is isomorphic in codimension $1$. Otherwise, we can assume that $g$ is a non-identical divisorial contraction. Let $p: W \to X, q: W \to X', r: W \to Y$ be projective birational morphisms such that $f=r\circ p^{-1}, g=r \circ q^{-1}, h=q\circ p^{-1}$. Then we have
\[
p^*B=r^*B_Y+E,
\] where $B_Y$ is the strict transform of $B$ on $Y$ and $E \geq 0$ is an $r$-exceptional divisor whose support contains $\Exc(f)$. Moreover, we have $p^*B=q^*B'+F$, where $B'$ is the strict transform of $B$ on $X'$ and $F$ is $q$-exceptional. Thus, we have $q^*B'+F=r^*B_Y+E$. Pushing forward by $q_*$, we obtain
\[
B' = g^*B_Y + q_*E,
\] where $p_*E > 0$ is a $g$-exceptional divisor which contains $\Exc(g)>0$. Hence, there exists a family of curves $l$, contracted by $g$ and covering an irreducible component of $\Supp(p_*E)$, such that $(q_*E)\cdot l < 0$. This implies that $B' \cdot l<0$, which contradicts the choice of $B$. Indeed, if $A'$ is the strict transform of $A$ on $X'$, then we have $A' \in \Mov(X'/T)$. Hence, we have $A' \cdot l \geq 0$. Since $B' = D' + A'$, we obtain $B'\cdot l = D' \cdot l + A' \cdot l \geq 0$.

 Therefore, $f$ is isomorphic in codimension $1$ and thus $[B]=[D+A] \in \Mov(X/T)$. By the arbitrariness of $A$, we have $[D]=\lim_{\ep \to 0^+} [D+\ep A] \in \bMov(X/T)$. The same argument shows that $[-D]\in \bMov(X/T)$. This implies that $[D]\in M$. 

Because $N^1(X/T) \to N^1(X'/T)$ is a linear map defined over $\Qq$, and the curve class $[C]$ in $N_1(X'/T)$ is a rational point, the vector space
\[
\begin{split}
&\{[D] \mid D'\cdot C=0, \text{~where~} X \dto X'/T \text{~is isomorphic in codimension~}1, \\
&\qquad X' \text{~is~} \Qq\text{-factorial, and~} C \text{~is a curve whose class covers a divisor}\}
\end{split}
\] is defined over $\Qq$. Since $E$ is defined over $\Qq$, $M$ is also defined over $\Qq$ by \eqref{eq: 1}.
\end{proof}

\begin{lemma}\label{lem: small Q-fact is MKD}
Let $X/T$ be an MKD fiber space. If $h: X \dto Y/T$ is a small $\bQ$-factorial modification, then $Y/T$ is still an MKD fiber space.
\end{lemma}
\begin{proof}
By assumption, $Y/T$ satisfies Definition \ref{def: MKD spaces} (1). To see that $Y/T$ satisfies Definition~\ref{def: MKD spaces}~(3), it suffices to observe that $h$ naturally identifies $\Mov(X/T)$ with $\Mov(Y/T)$, $\PsAut(X/T)$ with $\PsAut(Y/T)$, and the corresponding group actions. Definition \ref{def: MKD spaces} (2) follows from Lemma \ref{lem: iso in codim 1 preserves mm} and Definition \ref{def: MKD spaces} (4) follows from Proposition \ref{prop: condition (4)} (3).
\end{proof}

We are now ready to establish the equivalence between the existence of weak rational polyhedral fundamental domains for the movable, nef, and effective cones. The proof below follows the line of \cite[Theorem 1.5]{GLSW26} and \cite[Lemma 5.2 (1)]{LZ25}, with suitable adaptations to the MKD space setting.

\begin{proof}[Proof of Theorem \ref{thm: 3 equivalences}]
The implication $(1) \Rightarrow (2)$ can be proved essentially in the same way as in \cite[Theorem~14]{Xu24} (see also \cite[Theorem~1.5]{GLSW26}). 
Assume that $\Pi \subset \Mov(X/T)$ is a rational polyhedral cone such that $\Gamma_B(X/T) \cdot \Pi = \Mov(X/T)$. 
Note that in this step we only use the local factoriality of canonical models$/T$ for $\Pi$, instead of for $\Eff(X/T)$. This observation will be used in the proof of Corollary \ref{cor: local factoriality for Pi} below.

By Lemma~\ref{lem: small Q-fact is MKD} and Proposition \ref{prop: condition (4)} (3), if $h: X \dto X'/T$ is a small $\bQ$-factorial modification, then $X'/T$ is also an MKD fiber space and $h\cdot \Pi \coloneqq h_*\Pi$ still satisfies the local factoriality of canonical models$/T$. Without loss of generality, we may assume that $X'/T$ is $X/T$. We show the existence of a rational polyhedral fundamental domain for $\Nef^e(X/T)$ under the action of $\Gamma_A(X/T)$.

Applying Theorem \ref{thm: Shokurov poly for minimal models} to $\Pi$, we see that $\Pi = \cup_{l\in J} \Pi_{l}$ can be decomposed into finitely many relatively open rational polyhedral cones such that divisors in the same cone share the same weak minimal models$/T$. If $g\cdot \bar\Pi_{l} \cap \Amp(X/T) \neq \{0\}$ for some $g\in \PsAut(X/T)$, then 
\[
g\cdot \Pi_{l} \cap (\Amp(X/T)\setminus \{0\}) \neq \emptyset
\] as $\Amp(X/T)\setminus \{0\}$ is open. By Theorem  \ref{thm: Shokurov poly for minimal models}, we have $g\cdot \bar\Pi_{l} \subset \Nef^e(X/T)$. 

As $\Amp(X/T) \subset\Mov(X/T)$, we have
    \[
\Amp(X/T) = \bigcup_{l\in J}\bigcup_{g\in \PsAut(X/T)} \left(g\cdot \bar\Pi_{l} \cap \Amp(X/T)\right).
    \] Therefore, we have
    \begin{equation}\label{eq: inclusion}
    \Amp(X/T) \subset \bigcup_{l\in J}\bigcup_{\substack{g\in \PsAut(X/T)\\ g\cdot \bar\Pi_{l} \subset \Nef^e(X/T)}} g\cdot \bar\Pi_{l}.
    \end{equation}

    Moreover, if both $g\cdot \bar\Pi_{l}\cap \Amp (X/T) \neq \{0\}$ and $h\cdot \bar\Pi_{l}\cap \Amp (X/T) \neq \{0\}$, then 
    \begin{equation}\label{eq: in iso}
    h\circ g^{-1} \in \Aut(X/T).
    \end{equation} In fact, take an ample$/T$ divisor $A$ such that $[A] \in g\cdot \bar\Pi_{l}$, then we have
    \[
    (h\circ g^{-1}) \cdot [A] \in h\cdot \bar\Pi_{l}\subset \Nef(X/T).
    \] By Lemma \ref{lem: common mm} (2),  $g\circ h^{-1}$ is a morphism. As $g\circ h^{-1}$ is a small $\bQ$-factorial modification, $g\circ h^{-1}$ must be an isomorphism.

Therefore, for each $l \in J$, if there exists some $g \in \PsAut(X/T)$ such that $g \cdot \bar\Pi_{l} \subset \Nef^e(X/T)$, we fix one such $g$, denoted by $g_l$. Let 
\[
P' \coloneqq \Cone (g_l \cdot \bar\Pi_{l} \mid \text{~there exists~}l \in J \text{~such that~} g_l \cdot \bar\Pi_{l} \subset \Nef^e(X/T))
\] 
be the cone generated by finitely many rational polyhedral cones $g_l \cdot \bar\Pi_{l}$. In particular, $P'\subset \Nef^e(X/T)$ is a rational polyhedral cone. Moreover, we have
\[
\Aut(X/T)\cdot P' \supset \Amp(X/T)
\] by \eqref{eq: inclusion} and \eqref{eq: in iso}. Hence, $(\Nef(X/T)_+, \Gamma_A(X/T))$ is of polyhedral type. As $\Nef(X/T)$ is a non-degenerate cone, by Lemma \ref{le: existence of fun domain}, there exists a rational polyhedral cone which is the fundamental domain for $\Nef(X/T)_+$ under the action of $\Gamma_A(X/T)$. 

To show that $\Nef^e(X/T)$ admits a rational polyhedral fundamental domain under the action of $\Gamma_A(X/T)$, it suffices to prove that
\begin{equation}\label{eq: +=e}
\Nef(X/T)_+ = \Nef^e(X/T).
\end{equation}
Indeed, by Proposition \ref{prop: prop-def}, we always have $\Gamma_A(X/T) \cdot P' = \Nef(X/T)_+$. As $P'\subset \Nef^e(X/T)$, this implies that $\Nef(X/T)_+\subset \Nef^e(X/T)$. Conversely, let $[D] \in \Nef^e(X/T)\subset \Mov(X/T)$. Then there exists $g\in \PsAut(X/T)$ such that $[D] \in g\cdot \Pi$. As $\Pi$ satisfies the local factoriality of canonical models$/T$, so is $g\cdot \Pi$ by Proposition \ref{prop: condition (4)} (3). Applying Theorem \ref{thm: Shokurov poly for nef} to $g\cdot \Pi$, we see that
\[
\mathcal N_{g\cdot \Pi}\coloneqq \{[B] \in g\cdot \Pi \mid B \text{~is nef over~}T\}
\] is a rational polyhedral cone. Hence, we have $[D] \in \Nef(X/T)_+$. This shows \eqref{eq: +=e}, and thus $\Nef^e(X/T)$ admits a rational polyhedral fundamental domain under the action of $\Gamma_A(X/T)$.

Next, we show the finiteness of 
\begin{equation}\label{eq: finiteness of bir contractions}
        \{ Y/T \mid X \dto Y/T \text{~is a birational contraction}\}
\end{equation} up to isomorphism of $Y/T$. It follows from a more precise statement that there exist finitely many birational contractions 
$g_j: X \dto Y_j/T, 1 \leq j \leq n$, such that for any birational contraction 
$u: X \dto Y/T$, there exists some $\mu \in \PsAut(X/T)$ and an index $1 \leq j \leq n$ satisfying 
\begin{equation}\label{eq: decomposition of bir contraction}
    u = g_j \circ \mu.
\end{equation}

Indeed, let $\Pi$ be a weak rational polyhedral fundamental domain for $\Mov(X/T)$ under the action of $\Gamma_B(X/T)$. Let $\Pi = \cup_l \Pi_l$ be a decomposition of $\Pi$ into relatively open rational polyhedral cones satisfying Theorem~\ref{thm: Shokurov poly for minimal models}. For each $l$, we fix a birational contraction
\[
\ti g_l: X \dto \ti Y_l/T
\]
which is a minimal model$/T$ for some effective divisor in $\Pi_l$. By Lemma \ref{lem: movable give iso in codim 1}, $g_l$ is isomorphic in codimension $1$. By Theorem~\ref{thm: Shokurov poly for minimal models}, $\ti g_l$ is in fact a weak minimal model$/T$ for every effective divisor in $\Pi_l$.

Let $A_Y$ be an ample$/T$ divisor on $Y$, and let $A$ denote its strict transform on $X$. Then we have $[A]\in \Mov(X/T)$. Hence, there exist 
some $\mu\in \PsAut(X/T)$ and $l$ such that $\mu\cdot [A] \in \Pi_l$. Then $\ti g_l: X \dto \ti Y_l$ is a weak minimal model$/T$ of $\mu_*A$. 
\[
\begin{tikzcd}
X \arrow[dr,dashed, "u"'] \arrow[r, dashed, "\mu"] & X \arrow[r, dashed, "\tilde g_l"] & \tilde Y_l \arrow[dl, "\tau"] \\
&Y & 
\end{tikzcd}
\]
By Lemma \ref{lem: common mm} (2), the natural map 
\[
\tau \coloneqq u \circ \mu^{-1} \circ {\ti g_l}^{-1}: \ti Y_l \dto Y
\] is indeed a morphism and $\tau^*A_Y=\ti A$, where $\ti A =  ({\ti g_l} \circ \mu)_*A$. Moreover, $\tilde A$ lies in $(\tilde g_l)_* \Pi_l$, whose closure is a rational polyhedral cone in $\Nef(\tilde Y_l/T)$. Hence, there are only finitely many possibilities for $\tau$, which in turn implies that there are only finitely many possibilities for $g_j \coloneqq \tau \circ \ti g_l$.

\medskip

For the implication $(2) \Rightarrow (3)$, let $D$ be an effective $\bR$-Cartier divisor on $X$ with $u: X \dto Y/T$ a good minimal model$/T$ of $D$. Then Lemma \ref{lem: factor bir contraction} shows that it factors into a small $\bQ$-factorial modification $h: X \dto X'/T$ followed by a contraction morphism $X' \to Y/T$. By the assumption of (2), $X'/T$ belongs to a finite set 
\[
\{X_j/T \mid X \dto X_j/T \text{~is a small $\bQ$-factorial modification,~} 1 \leq j\leq n\}
\] up to isomorphism. Moreover, each $X_j/T$ admits a weak rational polyhedral fundamental domain $P_j \subset \Nef^e(X_j/T)$ under the action of $\Gamma_A(X_j/T)$. Thus, $u$ is a composition of natural maps
\begin{equation}\label{eq: maps}
u: X \overset{h}{\dashrightarrow}  X' \overset{\sigma}{\simeq}  X_j \xrightarrow{\theta} Y/T,
\end{equation} 
where $\sigma$ is an isomorphism$/T$ and $\theta$ is a birational contraction morphism.

By assumption, $\Nef^e(X_j/T)$ admits a rational polyhedral fundamental domain $P_j$
under the action of $\Gamma_A(X_j/T)$. In particular, we have $P_j \subset \Nef(X/T)_+$ and $\Gamma_A(X_j/T) \cdot P_j \supset \Amp(X_j/T)$. Thus, we have $\Gamma_A(X_j/T) \cdot P' = \Nef(X_j/T)_+$ by Proposition \ref{prop: prop-def}. This implies that 
\[
\Nef^e(X_j/T)=\Nef(X_j/T)_+.
\] By Lemma \ref{lem: iso in codim 1 preserves mm}, effective divisors in $\Nef^e(X_j/T)$ are semi-ample$/T$. Hence, $\theta^*\Nef^e(Y/T)$  is a face of $\Nef^e(X_j/T)=\Nef(X_j/T)_+$. By Proposition \ref{prop: fundamental domain for surface} and Lemma \ref{le: existence of fun domain}, $\theta^*\Nef^e(Y/T)$ admits a rational polyhedral fundamental domain $P'_{j,\theta}$ under the action of 
\[
{\rm Stab}_{\theta^*\Nef^e(Y/T)} \Gamma_{A}(X_j/T) = \{[g]\in \Gamma_{A}(X_j/T) \mid [g](\theta^*\Nef^e(Y/T))=\theta^*\Nef^e(Y/T) \}.
\] 

Let 
\begin{equation}\label{eq: P_j, theta}
P_{j,\theta} \coloneqq \Cone(P'_{j,\theta}, [E] \mid E \text{~is~} \theta\text{-exceptional}) \subset \Eff(X_j/T)
\end{equation}
be the rational polyhedral cone generated by $P'_{j,\theta}$ and $\theta$-exceptional divisors. By definition, for any $g$ with $[g] \in {\rm Stab}_{\theta^*\Nef^e(Y/T)} \Gamma_{A}(X_j/T)$, the map $g$ sends each $\theta$-exceptional divisor to another $\theta$-exceptional divisor. As a result, 
\begin{equation}\label{eq: small cone}
{\rm Stab}_{\theta^*\Nef^e(Y/T)} \Gamma_{A}(X_j/T) \cdot P_{j,\theta}\supset \Cone(\theta^*\Nef^e(Y/T), [E] \mid E \text{~is~} \theta\text{-exceptional}).
\end{equation}

Since $P_j$ is a rational polyhedral fundamental domain for $\Nef^e(X_j/T)$ under the action of $\Gamma_A(X_j/T)$, there exists a finite set of morphisms
\[
\{\theta_j^k\colon X_j\to Y_j^k/T \mid 1\leq k\leq m(j)\}
\]
such that, for any contraction morphism $\theta\colon X_j\to Y/T$, there exist $k$ and $\tau\in \Aut(X_j/T)$ such that $\theta=\theta_j^k\circ \tau$ up to isomorphism. By \eqref{eq: maps}, $u$ is a composition of natural maps
\[
X \overset{h}{\dashrightarrow} X' \overset{\sigma}{\simeq}  X_j \overset{\tau}{\simeq} X_j \xrightarrow{\theta^k_j} Y/T.
\]

Set $D_j = (\tau \circ \sigma \circ h)_* D$. By construction, we have
\[
[D_j] \in \Cone\big((\theta^k_j)^* \Nef^e(Y/T), \; [E] \mid E \text{ is }\theta^k_j\text{-exceptional}\big).
\]

By \eqref{eq: small cone}, as $\tau \in \Aut(X_j/T)$, we have
\begin{equation}\label{eq: second level}
[(\sigma\circ h)_*D] \in \Gamma_A(X_j/T) \cdot P_{j, \theta}.
\end{equation}

For each $X_j/T, 1 \leq j \leq n$, fixed a small $\bQ$-factorial modification $h_j: X \dto X_j/T$. Set 
\begin{equation}\label{eq: Q}
Q_{j,k} \coloneqq h_j^*P_{j, \theta^k_j}\subset \Eff(X/T), \quad 1 \leq k \leq m(j).
\end{equation} We claim that
\begin{equation}\label{eq: cover eff}
\bigcup_{j, 1 \leq k\leq m(j)}\Gamma_B(X/T) \cdot Q_{j, k} = \Eff(X/T).
\end{equation}
Indeed, as $h, \sigma$ and $h_j$ are all isomorphic in codimension $1$, we have $(\sigma\circ h)^{-1} \circ h_j\in \PsAut(X/T)$. By \eqref{eq: second level}, we have
\[
[D] \in \Gamma_B(X/T) \cdot Q_{j, k}.
\] This shows ``$\supset$'' in \eqref{eq: cover eff}. The converse inclusion follows from $Q_{j,k}\subset \Eff(X/T)$.

As there are finitely many rational polyhedral cones $Q_{j, k}$, 
\begin{equation}\label{eq: combine}
Q \coloneqq \Cone(Q_{j,k}  \mid 1 \leq j \leq n, 1 \leq k \leq m(j))
\end{equation} is a rational polyhedral cone which satisfies 
\[
\Gamma_B(X/T) \cdot Q = \Eff(X/T)
\] by \eqref{eq: cover eff}. Besides, by Proposition \ref{prop: degenerate cone} and Proposition \ref{prop: max vector defined over Q}, $\Eff(X/T)_+$ admits a weak rational polyhedral fundamental domain under the action of $\Gamma_B(X/T)$.

\medskip

The implication $(3) \Rightarrow (1)$ can be shown by the same argument as in \cite[Lemma 5.2 (1)]{LZ25}. Note that it suffices to verify Definition \ref{def: MKD spaces} (3). 

Assume that $Q\subset \Eff(X/T)$ is a rational polyhedral cone such that $\Gamma_B(X/T) \cdot Q =\Eff(X/T)$.  By Theorem \ref{thm: Shokurov poly for minimal models}, $Q = \cup_s Q_{s}$ can be decomposed into finitely many relatively open rational polyhedral cones such that divisors in the same cone share the same weak minimal model$/T$. 

We claim that $Q_{s} \cap \Mov(X/T) \neq \emptyset$ implies that $\bar Q_{s} \subset \Mov(X/T)$. Indeed, if $D$ is an effective divisor such that $[D]\in Q_{s} \cap \Mov(X/T)$, then let $h: X \dto Y/T$ be a minimal model$/T$ of $D$. Let $p: W \to X, q: W \to Y$ be projective birational morphisms such that $h=q\circ p^{-1}$. Then
\[
p^*D=q^*D_Y+E,
\] where $D_Y = h_*D$ is semi-ample$/T$ and $E\geq 0$ is a $q$-exceptional divisor whose support contains $\Exc(h)$. As $[D] \in \Mov(X/T)$, $h$ is isomorphic in codimension $1$ by Lemma \ref{lem: movable give iso in codim 1}. If $B$ is an effective divisor with $[B]\in \bar Q_s$, then $h_*B$ is nef$/T$. By assumption, $B$ admits a good minimal model$/T$. Therefore, by Lemma \ref{lem: common mm} (1), $h_*B$ is semi-ample$/T$. As $h$ is isomorphic in codimension $1$, we see that $\bar Q_{s} \subset \Mov(X/T)$. 

Let $P$ be the rational polyhedral cone defined by
\[
P \coloneqq \Cone\{\bar Q_s \mid Q_{s} \cap \Mov(X/T) \neq \emptyset\}.
\] By the above discussion, we have $P \subset \Mov(X/T)$, and thus $\Gamma_B(X/T) \cdot P \subset \Mov(X/T)$. Conversely, if $D$ is an effective $\bR$-Cartier divisor such that $[D] \in \Mov(X/T) \subset \Eff(X/T)$, then there exist $g\in \PsAut(X/T)$ and $Q_s$ such that $[g_*D] \in Q_s$. As $[g_*D] \in \Mov(X/T)$, we have $[g_*D] \in Q_s \cap \Mov(X/T)$. This shows
\[
\Gamma_B(X/T) \cdot P = \Mov(X/T),
\]
which verifies Definition \ref{def: MKD spaces} (3) for $X/T$.
\end{proof}

We record several results established in the proof of Theorem~\ref{thm: 3 equivalences}, as well as some easy consequences of it, for later use.

In the following corollary, (1) can be viewed as an SYZ-type statement, while (3) corresponds to the numerical non-vanishing for MKD fiber spaces.

\begin{corollary}\label{cor: of 3 equiv}
Assume that $X/T$ is an MKD fiber space.
\begin{enumerate} 
\item $\Nef(X/T)_+ = \Nef^e(X/T)$. That is, any nef $\Qq$-Cartier divisor on $X/T$ is numerically equivalent to a semi-ample $\Qq$-Cartier divisor over $T$.
\item There exist finitely many birational contractions 
$g_j: X \dto Y_j/T, 1 \leq j \leq n$, such that for any birational contraction 
$u: X \dto Y/T$, there exist $\mu \in \PsAut(X/T)$ and an index $1 \leq j \leq n$ satisfying $u = g_j \circ \mu$.
\item If $\Eff(X/T)$ is a non-degenerate cone, then $\Eff(X/T)_+=\Eff(X/T)$. That is, any pseudo-effective $\Qq$-Cartier divisor on $X/T$ is numerically equivalent to an effective $\Qq$-Cartier divisor over $T$.
\end{enumerate}
\end{corollary}
\begin{proof}
Note that (1) is established in \eqref{eq: +=e}, and (2) is established in \eqref{eq: decomposition of bir contraction}. By Theorem~\ref{thm: 3 equivalences}, there exists a rational polyhedral cone $P \subset \Eff(X/T)$ such that $\Gamma_B(X/T) \cdot P = \Eff(X/T)$. If $\Eff(X/T)$ is non-degenerate, then Proposition~\ref{prop: prop-def} gives $\Gamma_B(X/T)\cdot P=\Eff(X/T)_+$. This proves (3).
\end{proof}

The following corollary of the proof of Theorem~\ref{thm: 3 equivalences} shows that, at least in the absolute setting, local factoriality of canonical models on $\Eff(X)$ can be deduced from local factoriality on $\Pi$. Thus, it refines the definition of MKD spaces and better reflects their guiding philosophy: MKD spaces should be obtained by gluing Mori dream spaces as local models.

\begin{corollary}\label{cor: local factoriality for Pi}
Let $X/T$ be a fibration satisfying all the properties of an MKD fiber space in Definition \ref{def: MKD spaces}, except that in (4) we replace the condition by requiring that $\Pi$ satisfies the local factoriality of canonical models$/T$. If $\Eff(X/T)$ is a non-degenerate cone, then $\Eff(X/T)$ also satisfies the local factoriality of canonical models$/T$.
\end{corollary}
\begin{proof}
We follow the notation used in the proof of Theorem~\ref{thm: 3 equivalences}. First, as remarked in the first paragraph of the proof of Theorem~\ref{thm: 3 equivalences}, in the proof of $(1)\Rightarrow(2)$, we only use the fact that $\Pi$ satisfies local factoriality of canonical models$/T$. Hence, under the assumption of Corollary \ref{cor: local factoriality for Pi}, we still have Theorem \ref{thm: 3 equivalences} (2). 

Next, let us recall several constructions from the proof of $(2)\Rightarrow(3)$. We have the following maps as in \eqref{eq: maps}:
\[
X \overset{h}{\dashrightarrow} X' \overset{\sigma}{\simeq} X_j \xrightarrow{\theta} Y/T.
\] Moreover, we define the rational polyhedral cone
\[
P_{j,\theta} = \Cone(P'_{j,\theta}, [E] \mid E \text{~is~} \theta\text{-exceptional}) \subset \Eff(X_j/T),
\] in \eqref{eq: P_j, theta}, where $P'_{j,\theta}$ is a rational polyhedral cone inside $\theta^*\Nef^e(Y/T)$. Therefore, $P_{j,\theta}$ satisfies the local factoriality of canonical models$/T$ by Proposition \ref{prop: condition (4)} (1). As $h_j: X \dto X_j/T$ is isomorphic in codimension $1$, 
\[
Q_{j,k}= h_j^*P_{j, \theta_l^k}\subset \Eff(X/T) \quad \text{ (see \eqref{eq: Q})}
\] satisfies the local factoriality of canonical models$/T$ by Proposition \ref{prop: condition (4)} (3). By \eqref{eq: cover eff}, we have
\begin{equation}\label{eq: cover 2}
\bigcup_{j, 1 \leq k\leq m(j)}\Gamma_B(X/T) \cdot Q_{j, k} = \Eff(X/T),
\end{equation}
which is a finite union of sets $\Gamma_B(X/T) \cdot Q_{j,k}$. Moreover, we have 
\[
\Gamma_B(X/T) \cdot Q = \Eff(X/T),
\] where $Q= \Cone(Q_{j,k}  \mid 1 \leq j \leq n, 1 \leq k \leq m(j))$ is a rational polyhedral cone (see \eqref{eq: combine}). This implies that $(\Eff(X/T)_+, \Gamma_B(X/T))$ is of polyhedral type. From the above construction, we can conclude that $\Eff(X/T)$ satisfies the local factoriality of canonical models$/T$ in the following.

Let $P \subset \Eff(X/T)$ be a rational polyhedral cone. As $\Eff(X/T)$ is non-degenerate by assumption and $(\Eff(X/T)_+, \Gamma_B(X/T))$ is of polyhedral type, the set
\begin{equation}\label{eq: Siegel 1}
\{(\gamma \cdot Q_{j,k}) \cap P \mid \gamma \in \Gamma_B(X/T)\}
\end{equation} is a finite set for each $Q_{j,k}$ by \cite[Theorem 3.8]{Loo14} (this is called the Siegel property). Moreover, since each $Q_{j,k}$ satisfies the local factoriality of canonical models$/T$, so does $\gamma \cdot Q_{j, k}$ by Proposition \ref{prop: condition (4)} (3). Hence, $(\gamma \cdot Q_{j, k}) \cap P$ also satisfies the local factoriality of canonical models$/T$. By \eqref{eq: cover 2} and \eqref{eq: Siegel 1}, $P$ is a finite union of rational polyhedral cones satisfying the local factoriality of canonical models$/T$, and hence $P$ also satisfies this property. This completes the proof.
\end{proof}

\begin{remark}\label{rmk: Pi is fund of Eff is ok}
A proof similar to that of Corollary~\ref{cor: local factoriality for Pi} shows the following. Suppose that $\Pi'\subset \Eff(X/T)$ is a rational polyhedral cone satisfying local factoriality of canonical models$/T$. If $\Eff(X/T)$ is non-degenerate and $\PsAut(X/T)\cdot \Pi'=\Eff(X/T)$, then $\Eff(X/T)$ also satisfies local factoriality of canonical models$/T$.
\end{remark}

\begin{remark}\label{rmk: siegel property}
The difficulty in extending the proof of Corollary \ref{cor: local factoriality for Pi} to the case where $\Eff(X/T)$ is degenerate lies in extending the Siegel property (i.e., \cite[Theorem 3.8]{Loo14}) to the degenerate case. This extension appears plausible (at least in the above geometric setting), but it is currently unavailable.
\end{remark}

\begin{proposition}\label{prop: generic property for MKD space}
Let $X/T$ be an MKD fiber space. If $U \subset T$ is a non-empty open subset of $T$, then $X_U/U$ is still an MKD fiber space.
\end{proposition}
\begin{proof}
If $D$ is a prime divisor on $X_U$, let $\bar D$ denote its Zariski closure in $X$. Since $X$ is $\bQ$-factorial, $\bar D$ is a $\bQ$-Cartier divisor. Hence $D$ is also $\bQ$-Cartier, and therefore $X_U$ remains $\bQ$-factorial. 

If $B = \sum a_i B_i$ is an effective $\bR$-divisor on $X_U$ with prime divisors $B_i$, let 
$\bar B = \sum a_i \bar B_i$. Since $\bar B|_U = B$, we see that if $X \dto Y/T$ is a good minimal model of $\bar B$ over $T$, then $X_U \dto Y_U/U$ is a good minimal model of $B$ over $U$. Consequently, the local factoriality of canonical models over $U$ for $\Eff(X_U/U)$ follows from that over $T$ for $\Eff(X/T)$. 

Applying Theorem \ref{thm: 3 equivalences} (3) to $X/T$, there exists a rational polyhedral cone $P\subset \Eff(X/T)$ such that $\PsAut(X/T) \cdot P = \Eff(X/T)$. Let $P'$ be the image of $P$ under the natural map $N^1(X/T) \to N^1(X_U/U)$. It follows that $P'\subset \Eff(X_U/U)$ and 
\[
\PsAut(X_U/U) \cdot P' =\Eff(X_U/U).
\]
By Theorem \ref{thm: 3 equivalences}, we see that $X_U/U$ is an MKD fiber space.
\end{proof}

\subsection{MKD fiber spaces under birational contractions}\label{subsec: MKD spaces under birational contractions}

We prove Theorem \ref{thm: bir contraction is MKD}, which states that a birational contraction of an MKD fiber space is still an MKD fiber space.

\begin{proof}[Proof of Theorem \ref{thm: bir contraction is MKD}]
By Lemma \ref{lem: factor bir contraction}, there is a small $\bQ$-factorial modification $h: X \dto X'/T$ and a contraction morphism $g: X' \to Y$ such that $f= g\circ h$. By Lemma \ref{lem: small Q-fact is MKD}, $X'/T$ is still an MKD fiber space. Replacing $X'$ by $X$, we can assume that $f$ is a morphism.

First, we show that if $D$ is an effective $\bR$-Cartier divisor on $Y$, then $D$ admits a good minimal model$/T$. Let $B\coloneqq f^*D+\Exc(f)$ be an effective divisor on $X$. By assumption, let $\theta: X \dto X'/T$ be a good minimal model$/T$ of $B$. We will show that 
\[
u \coloneqq \theta \circ f^{-1}: Y \dto X'/T
\] is a good minimal model$/T$ of $D$. Let $p: W \to X, q: W \to Y, r: W \to X'$ be projective birational morphisms such that $f=q\circ p^{-1}$ and $\theta=r\circ p^{-1}$. Then, we have
\[
p^*(f^*D+\Exc(f))=r^*B_Y+E,
\] where $B_Y =q_*B$ and $E \geq 0$ is an exceptional $r$-divisor whose support contains $\Exc(\theta)$. As $B_Y$ is semi-ample$/T$, $E$ is the negative part of the Nakayama-Zariski decomposition$/T$ of $r^*B_Y+E$. As $p^*\Exc(f)$ is $(f\circ p)$-exceptional, we have $\Supp(p^*\Exc(f)) \subset \Supp E$. This implies that $p^*\Exc(f)$ is $r$-exceptional. Hence, $u: Y \dto X'/T$ is a birational contraction. We have
\[
p^*(f^*D)=r^*B_Y+E-p^*\Exc(f),
\] where $E-p^*\Exc(f)$ is $r$-exceptional. Since $(f\circ p)_*(E-p^*\Exc(f))=(f\circ p)_*E \geq 0$, we have $E-p^*\Exc(f) \geq 0$ by the negativity lemma. Moreover, since $\Supp(p^*\Exc(f))$ does not contain $\Exc(u)$, $\Supp(E-p^*\Exc(f))$ still contains  $\Exc(u)$. This shows that $u$ is a good minimal model$/T$ of $D$. 

Next, we show that $\Eff(Y/T)$ satisfies the local factoriality of canonical models$/T$. Indeed, if $Y \dto Z/T$ is the canonical model$/T$ of $D$, then the composition map $X \to Y \dto Z/T$  is the canonical model$/T$ of $f^*D$. Hence, the local factoriality of canonical models$/T$ for $\Eff(Y/T)$ follows from that of $\Eff(X/T)$.

Finally, it suffices to verify that $Y/T$ satisfies the conditions in Theorem~\ref{thm: 3 equivalences} (2) to conclude that $Y/T$ is an MKD fiber space.

As $X/T$ is an MKD fiber space, by Theorem \ref{thm: 3 equivalences}, we see that
\[
\{ Z/T \mid Y \dto Z/T \text{~is a birational contraction}\}
\] is a finite set up to isomorphism of $Z/T$. Suppose that $Y\dto Y'/T$ is a small $\bQ$-factorial modification. Since $X\dto Y'/T$ remains a birational contraction, Lemmas~\ref{lem: factor bir contraction} and~\ref{lem: small Q-fact is MKD} imply that, after replacing $Y$ by $Y'$, it suffices to show that $\Nef^e(Y/T)$ admits a rational polyhedral fundamental domain under the action of $\Gamma_A(Y/T)$.

By Corollary \ref{cor: of 3 equiv} (1), we have $\Nef^e(X/T)=\Nef(X/T)_+$. As effective divisors in $\Nef^e(X/T)$ are semi-ample$/T$, $f^*\Nef^e(Y/T)$ is a face of $\Nef^e(X/T)$. We claim that 
\begin{equation}\label{eq: e=+}
\Nef^e(Y/T)=\Nef(Y/T)_+.
\end{equation} Indeed, if $F$ is a face of the cone $C_+$, then $F = F_+$ (see \cite[Remark 3.5]{GLSW26}). Therefore, we have $f^*\Nef^e(Y/T) = (f^*\Nef(Y/T))_+$. Since $f^*: N^1(Y/T) \to N^1(X/T)$ is an injective linear map defined over $\bQ$, it follows that $(f^*\Nef(Y/T))_+ = f^*(\Nef(Y/T)_+)$. This implies $\Nef^e(Y/T) = \Nef(Y/T)_+$. 

By Theorem~\ref{thm: 3 equivalences} (2), $(\Nef^e(X/T), \Gamma_A(X/T))$ is of polyhedral type. By Proposition \ref{prop: fundamental domain for surface} and the above discussion,
\begin{equation}\label{eq: face is of polyhedral type}
(f^*\Nef^e(Y/T), ~{\rm Stab}_{f^*\Nef^e(Y/T)} \Gamma_{A}(X/T))
\end{equation} is still of polyhedral type. By Lemma \ref{lem: common mm} (2), any $g \in \Aut(X/T)$ with $[g] \in {\rm Stab}_{f^*\Nef^e(Y/T)} \Gamma_{A}(X/T)$ descends to an element $\ti g \in \Aut(Y/T)$. In other words, we have $\ti g \coloneqq f \circ g \circ f^{-1} \in \Aut(Y/T)$. As $f^*[\ti g \cdot D]=g\cdot [f^*D]$, we see that
\[
(\Nef^e(Y/T), ~\Gamma_A(Y/T))
\] is of polyhedral type. Then Lemma \ref{le: existence of fun domain} shows that there exists a rational polyhedral fundamental domain for the action of $\Gamma_A(Y/T)$ on $\Nef^e(Y/T)$. Hence, all the conditions in Theorem~\ref{thm: 3 equivalences} (2) are verified, and we can conclude that $Y/T$ is an MKD fiber space.
\end{proof}

We record the following result, established in the above proof (see \eqref{eq: e=+} and \eqref{eq: face is of polyhedral type}), for later use.

\begin{proposition}\label{prop: face is of polyhedral type}
Let $X/T$ be an MKD fiber space. If $f: X \to Y/T$ is a contraction morphism (not necessarily birational), then 
\[
(f^*\Nef^e(Y/T), \; {\rm Stab}_{f^*\Nef^e(Y/T)} \Gamma_{A}(X/T))
\]
is of polyhedral type. Moreover, we have $\Nef^e(Y/T)=\Nef(Y/T)_+$.
\end{proposition}

\subsection{Minimal model program for MKD fiber spaces}\label{subsec: MMP for MKD}

We show that, for any effective $\bR$-Cartier divisor $D$ on an MKD fiber space $X/T$, one can always perform steps of the minimal model program (MMP), and a $D$-MMP with scaling of an ample divisor always terminates. 

\subsubsection{Extremal contractions and flips in the category of MKD fiber spaces}

The following proposition establishes the existence of extremal contractions and flips for effective divisors in the category of MKD fiber spaces.

\begin{proposition}\label{prop: extremal contractions and flips}
Let $X/T$ be an MKD fiber space with $D$ an effective $\bR$-Cartier divisor. If $D$ is not nef$/T$, then there exists a $D$-negative extremal contraction $f: X \to Y/T$ with the relative Picard number $\rho(X/Y)=1$. Moreover, 
\begin{enumerate}
\item if $f$ is a divisorial contraction, then $Y/T$ is still an MKD fiber space, and
\item  if $f$ is a small contraction, then there exists a $D$-flip $X \dto X^+/Y$ such that $X^+/T$ is still an MKD fiber space.
\end{enumerate}
\end{proposition}
\begin{proof}
Let $A$ be an ample$/T$ $\bR$-Cartier divisor on $X$. If $D$ is not nef$/T$, then there exists $r>0$ such that $[D+rA]$ lies on the boundary of $\Nef(X/T)$. Since $D+rA$ is big$/T$, there exists a full-dimensional rational polyhedral cone $P\subset \Eff(X/T)$ such that $[D+rA]\in \Int(P)$. By Theorem~\ref{thm: Shokurov poly for nef},
\[
P\cap \Nef(X/T)
\]
is a rational polyhedral cone. Hence, after perturbing $A$, we may assume that $[D+rA]$ lies in the relative interior of a facet of $P\cap \Nef(X/T)$. Since $X/T$ is an MKD fiber space, the effective nef divisor $D+rA$ is semi-ample over $T$. Let $f: X \to Y/T$ be the birational contraction induced by $D+rA$. Then $f$ is a $D$-negative extremal contraction with $\rho(X/Y)=1$.

If $f$ is a divisorial contraction, then we claim that $Y$ is still $\bQ$-factorial. Let $B_Y$ be a Weil divisor on $Y$ with $B$ the strict transform of $B_Y$ on $X$. Let $E=\Exc(f)$ be the exceptional divisor. Then there exists some $c\in \bQ$ such that $B+cE \equiv 0/Y$. Let $H$ be an ample$/T$ divisor on $Y$. By the local factoriality of canonical models$/T$, there exists a positive rational number $\ep$, such that the canonical model$/T$ of $\ep(B+cE)+f^*H$ admits a morphism to $Y$. In particular, there exists a birational contraction $h: X \dto X'/Y$ such that $h_*(\ep(B+cE)+f^*H)$ is semi-ample over $Y$. On the other hand, as $h_*(\ep(B+cE)+f^*H) \equiv 0/Y$, there exists a $\bQ$-Cartier divisor $\Theta$ on $Y$ such that 
\[
f'^*\Theta \sim_\bQ h_*(\ep(B+cE)+f^*H)/T,
\] where $f'=f\circ h^{-1}: X' \to Y$. This implies that $f^*\Theta \sim_\bQ \ep(B+cE)+f^*H/T$. Therefore, we have
\[
B_Y = f_*(B+cE)\sim_\bQ \frac 1 \ep (\Theta-H)/T,
\] which is a $\bQ$-Cartier divisor. This shows that $Y$ is still $\bQ$-factorial. By Theorem \ref{thm: bir contraction is MKD}, $Y/T$ is still an MKD space.

\medskip

Next, assume that $f$ is a small $D$-negative contraction with $\rho(X/Y)=1$. Then $f^*\Nef(Y/T)$ is a facet of $\Nef(X/T)$. Let $H$ be an ample$/T$ divisor on $Y$. Then $f^*H$ is a big$/T$ divisor on $X$ and $[f^*H] \in f^*\Nef(Y/T)$. Hence, there exists a full-dimensional rational polyhedral cone $Q \subset \Eff(X/T)$ such that $[f^*H]\in {\rm Int}(Q)$. Let 
\[
Q = \bigsqcup_{i=1}^m Q_i
\] be a disjoint union of finitely many relatively open rational polyhedral cones as in Theorem \ref{thm: Shokurov poly for minimal models}. As $[f^*H]$ lies on the boundary of $\Nef(X/T)$, there exists a full-dimensional cone $Q_i$ and a facet $F \subset \bar Q_i$ such that 
\begin{enumerate}
\item $Q_i \cap \Nef(X/T)=\emptyset$,
\item $F \subset \Nef(X/T)$, and 
\item $[f^*H] \in F$.
\end{enumerate}
By the local factoriality of canonical models$/T$, shrinking $Q_i$ around $[f^*H]$, we can assume that the canonical model$/T$ of each effective divisor in $Q_i$ admits a morphism to $Y$. Let $g: X \dto X^+/T$ be a good minimal model corresponding to a divisor in $Q_i$. We claim that $g$ is the $D$-flip. 

By the construction, $X^+/T$ is $\bQ$-factorial and there exists a morphism $f^+: X^+ \to Y/T$. Therefore, $g$ is isomorphic in codimension $1$. As $f$ is a small contraction, $Y$ is not $\bQ$-factorial. Hence, $f^+$ is not an isomorphism. If $B^+$ is a $\bQ$-Cartier divisor on $X^+$ with $B$ the strict transform on $X$, then we have $B+cD \equiv 0/Y$ for some $c\in \bQ$. By the same argument as before, we have $B+cD \sim_{\bQ} 0/Y$ and thus $B^++cD^+\sim_{\bQ} 0/Y$. This implies $\rho(X^+/Y)=1$ and $D^+ \not\equiv 0/Y$. As $Q_i \cap \Nef(X/T)=\emptyset$, $g$ is not an isomorphism. Thus $D^+$ is ample$/Y$. This implies that $g: X \dto X^+/T$ is the $D$-flip. By Lemma \ref{thm: bir contraction is MKD}, $Y/T$ is still an MKD fiber space.
\end{proof}

\begin{remark}\label{rmk: no MMP for non-effective div}
Unlike for Mori dream spaces, it is in general impossible to run MMPs for non-pseudo-effective divisors on MKD spaces. For example, let $A$ be a simple abelian variety with $\rho(A)=2$. We first recall the following general fact: if $A$ is a simple abelian variety, then every nonzero rational nef class in $N^1(A)$ is ample.

Indeed, let $L$ be a nef Cartier divisor. After replacing $L$ by a numerically equivalent divisor, we may assume that $L$ is effective by \cite[Lemma~1.1]{Bau98}. Consider the homomorphism associated to $L$,
\[
\phi_L: A\longrightarrow \widehat A,\quad
x\longmapsto t_x^*L\otimes L^{-1},
\]
where $ \widehat A$ is the dual abelian variety and $t_x$ is the translation by $x$. Then the kernel $\ker(\phi_L)$ is a Zariski closed subgroup of $A$ by \cite[Page~57]{Mum08}. Since $A$ is simple, either $\ker(\phi_L)=A$ or $\ker(\phi_L)$ is finite. If $\ker(\phi_L)=A$, then $c_1(L)=0$ in $N^1(A)$. Thus, if the numerical class of $L$ is nonzero, then $\ker(\phi_L)$ is finite. Hence $L$ is ample by \cite[Page~57]{Mum08}.

It follows that the two boundary rays of $\Nef(A)$ are not rational rays. Consequently, the two boundary rays $\bR_{\geq 0}[\ell_1]$ and $\bR_{\geq 0}[\ell_2]$ of the dual cone $\overline{NE}(A)$ are not rational rays either. In particular, for every curve $C$ on $A$, the class $[C]$ lies in the interior of $\overline{NE}(A)$. Now take a Cartier divisor $D$ such that $D\cdot \ell_1>0$ and $D\cdot \ell_2<0$. Then $D$ is not pseudo-effective. If $C$ is a curve such that $D\cdot C<0$, then any morphism contracting $C$ must contract all curves on $A$, since $[C]\in \operatorname{Int}(\overline{NE}(A))$. In particular, it must also contract curves $C'$ satisfying $D\cdot C'>0$. Hence no step of a $D$-MMP can be performed on $A$.
\end{remark}

\subsubsection{MMP with scaling for MKD fiber spaces}

Theorem \ref{thm: bir contraction is MKD} allows us to run the MMP in the category of MKD fiber spaces. Moreover, we can run a special MMP for MKD fiber spaces, called an MMP with scaling, as follows (cf. \cite[\S 3.10]{BCHM10}). 

Let $X/T$ be an MKD fiber space. Suppose that $D$ is an effective $\bR$-Cartier divisor on $X$ which is not nef$/T$. Let $A$ be a big$/T$ $\bR$-Cartier divisor such that $D+tA$ is nef$/T$ for some $t\in \Rr_{\geq 0}$. Set 
\[
r \coloneqq \inf \{t \in \Rr_{\geq 0} \mid D+tA \text{~is nef over~}T\}
\] to be the nef threshold of $D$ with respect to $A$. 

Since $D+rA$ is big over $T$, there exists a full-dimensional rational polyhedral cone $P \subset \Eff(X/T)$ such that $[D+rA] \in \Int(P)$. By Theorem \ref{thm: Shokurov poly for nef}, $P\cap \Nef(X/T)$ is a rational polyhedral cone. Note that $[D+rA]$ lies on the boundary of $\Nef(X/T)$.

Let $F_j$, $1\leq j\leq n$, be the facets of $P\cap \Nef(X/T)$ containing $[D+rA]$. Since $X/T$ is an MKD fiber space, every effective divisor whose class lies in $F_j$ is semi-ample over $T$. Then, for each $j$, there exists $[\ell_j]\in {\rm NE}(X/T)$ such that
\[
F_j \subset \{[B] \in N^1(X/T) \mid [B] \cdot \ell_j=0\}.
\]
In an open neighborhood of $[D+rA]$, the nef cone $\Nef(X/T)$ is given by
\[
\bigcap_{j=1}^n \{[B]\in N^1(X/T)\mid [B]\cdot \ell_j\geq 0\}.
\]
Since $[D+(r-\epsilon)A]\notin \Nef(X/T)$ for any $0<\epsilon\leq r$, there exists $k$ such that
\[
(D+(r-\epsilon)A)\cdot \ell_k<0.
\] As $(D+rA)\cdot \ell_k =0$, this shows $D\cdot \ell_k<0$. Let $f\colon X\to Y/T$ be the contraction morphism defined by any effective divisor whose class lies in $\operatorname{Int}(F_k)$. Then $f$ is a $D$-negative extremal contraction with $\rho(X/Y)=1$. As $D+rA$ is semi-ample$/T$, we have $D+rA \sim_\Rr 0/Y$.

If $f$ is a divisorial contraction, then $Y/T$ is an MKD fiber space by Proposition \ref{prop: extremal contractions and flips}~(1). Moreover, $f_*(D+rA)$ is nef $/T$. Hence, the nef threshold of $f_*D$ with respect to $f_*A$ is less than or equal to $r$. We can repeat the above process starting from $Y/T$. 

If $f$ is a small contraction, let $\theta: X\dto X^+/Y$ be the $D$-flip given by Proposition~\ref{prop: extremal contractions and flips}~(2). Then $X^+/T$ is still an MKD fiber space. We have $\theta_*(D+rA) \sim_\Rr 0/Y$. In particular, $\theta_*D + r \theta_*A$ is still nef over $T$. Hence, the nef threshold of $\theta_*D$ with respect to $\theta_*A$ is still less than or equal to $r$. We can repeat the above process starting from $X^+/T$. 

The above discussion implies that we can obtain a sequence of $D$-MMP$/T$,
\[
X = X_1 \dto X_2 \dto \cdots \dto X_n \dto \cdots ,
\] 
and a sequence of non-increasing non-negative real numbers 
\[
r_1 \geq r_2 \geq \cdots \geq r_n \geq \cdots,
\] 
which are the nef thresholds at each step. This process is called an MMP$/T$ with scaling of $A$. If there exists some $n$ such that $r_n = 0$, then the strict transform of $D$ on $X_n$ is nef over $T$. Hence, this MMP terminates at $X_n$. Moreover, $X \dto X_n/T$ is a good minimal model$/T$ of $D$. Indeed, set $Y = X_n$ and let $g: X \dto Y/T$ be this MMP$/T$ such that the strict transform of $D$, denoted by $D_Y$, is nef$/T$.  Since $X/T$ is an MKD fiber space, $D$ admits a good minimal model $h: X \dto X'$ over $T$. If $p: W \to X, q: W \to Y, r: W \to X'$ are projective birational morphisms such that $g=q\circ p^{-1}, h=r \circ p^{-1}$, then we have
\[
p^*D=q^*D_Y+E, \quad p^*D=r^*D'+F,
\] where $D'$ is the strict transform of $D$ on $X'$ and $E\geq 0$ (resp. $F\geq 0$) is a $q$-exceptional (resp. $r$-exceptional) divisor. By Lemma \ref{lem: common mm}, we have $q^*D_Y=r^*D'$, which implies that $D_Y$ is also semi-ample over $T$. Finally, as each step of this MMP is $D$-negative, we see that $g$ is also $D$-negative. This shows that $g$ is a good minimal model$/T$ of $D$.

\medskip

Now we establish Theorem \ref{thm: MMP for MKD}, which, combined with Proposition \ref{prop: extremal contractions and flips}, ensures the termination of the MMP for an MKD fiber space with scaling of an ample divisor.

\begin{proof}[Proof of Theorem \ref{thm: MMP for MKD}]
Based on the above discussion, we can run a $D$-MMP$/T$ with scaling of an ample$/T$ divisor $A$. Moreover, every variety appearing in this MMP$/T$ is an MKD fiber space by Proposition \ref{prop: extremal contractions and flips}.  If this MMP does not terminate, then we deduce a contradiction.

As Picard numbers strictly decrease under divisorial contractions, after finitely many steps, this MMP consists only of flips. Without loss of generality, we may assume that this MMP$/T$ involves only flips.

As $D$ is an effective $\bR$-Cartier divisor, there exists a rational polyhedral cone $Q \subset \Eff(X/T)$ such that $[D], [A] \in Q$. By Theorem \ref{thm: Shokurov poly for minimal models}, $Q$ decomposes into a finite disjoint union of relatively open rational polyhedral cones such that divisors in the same cone share the same weak minimal models$/T$. This decomposition induces a decomposition of the interval $[[D], [D+rA]]$ into finitely many closed intervals 
\begin{equation}\label{eq: decomp interval}
\{~[[D+c_iA], [D+c_{i+1}A]] ~\mid~ 1 \leq i \leq m~\},
\end{equation}
such that the divisors in each open interval $([D+c_iA], [D+c_{i+1}A])$ share the same weak minimal models$/T$. Suppose that $r_l$ is the nef threshold on $X_l$ for this MMP with scaling. Note that 
\[
g_l: X \dto X_l/T
\] is a weak minimal model$/T$ of $D+r_lA$. If $[D+r_lA] \in \left[[D+c_iA], [D+c_{i+1}A]\right)$, then we have $r_l=c_i$. Indeed, by the property of the decomposition in \eqref{eq: decomp interval}, $X \dto X_l/T$ is also a weak minimal model$/T$ of $D+c_iA$. In particular, there are finitely many possibilities for the nef thresholds.

We claim that 
\[
I \coloneqq \{i \mid r_i =r_l \text{~and~} i>l\}
\] is a finite set. 

Indeed, as $D+r_lA$ is big$/T$, there exists a rational polyhedral cone $R \subset \Eff(X/T)$ such that $[D+r_lA] \in \Int(R)$. Let $R= \cup_j R_j$ be a finite union satisfying Theorem \ref{thm: Shokurov poly for minimal models}. For each $i \in I$, as $X \dto X_i/T$ is a minimal model$/T$ of $D + r_l A$ and $[D + r_l A] \in \Int(R)$, there exists $[A_i] \in R_j$ such that the strict transform of $A_i$ on $X_i$ is ample over $T$. Since there are only finitely many $R_j$, we may assume that there exist distinct indices $i_1,i_2\in I$ such that $[A_{i_1}]$ and $[A_{i_2}]$ both lie in the same $R_j$. By the choice of $R_j$, $g_{i_1}: X \dto X_{i_1}/T$ is also a weak minimal model$/T$ of $A_{i_2}$. By Lemma \ref{lem: common mm}, the natural map $g_{i_1} \circ g_{i_2}^{-1}: X_{i_2} \dto X_{i_1}$ is a morphism. For the same reason, we see that $g_{i_2} \circ g_{i_1}^{-1}: X_{i_1} \dto X_{i_2}$ is also a morphism. Hence, there exists an isomorphism $h: X_{i_1} \simeq X_{i_2}/T$ such that $h\circ g_{i_1}=g_{i_2}$. We can assume that $i_1<i_2$. Let $p: W \to X_{i_1}, q: W \to X_{i_2}$ be projective birational morphisms such that $h=q\circ p^{-1}$. As $h: X _{i_1} \dto X_{i_2}$ consists of a sequence of $D$-flips, we have
\[
p^*D_{i_1}=q^*D_{i_2}+E
\] with $E>0$, where $D_{i_1}$ and $D_{i_2}$ are strict transforms of $D$ on $X_{i_1}$ and $X_{i_2}$, respectively. On the other hand, as $h$ is an isomorphism, we have
\[
p^*D_{i_1}=q^*D_{i_2}.
\] This leads to a contradiction, and thus the claim is verified. Therefore, the nef thresholds can occur only finitely many times for each fixed value, which contradicts the assumption that the MMP does not terminate.

Finally, by the discussion preceding the proof, this MMP terminates with a good minimal model$/T$.
\end{proof}

\subsection{Equivalence with the notion of Mori dream spaces under pseudo-automorphism actions}\label{subsec: equivalence}

In this section, we establish Theorem \ref{thm: two def are same}, which demonstrates that the MKD spaces defined in Definition \ref{def: MKD spaces} are indeed analogues of Mori dream spaces under pseudo-automorphism actions. Note that, for technical reasons, we consider the cone $\bMov^e(X)$ in Theorem \ref{thm: two def are same}. Once the theorem is established, it follows immediately that $\bMov^e(X) = \Mov(X)$ by Lemma \ref{lem: movable give iso in codim 1}.

\begin{proof}[Proof of Theorem \ref{thm: two def are same}]
First, assume that $X$ is an MKD space. Note that (a) and (d) follow directly from the definition of MKD spaces (see Definition \ref{def: MKD spaces}). By Lemma \ref{lem: movable give iso in codim 1}, we have $\bMov^e(X)=\Mov(X)$, which verifies (b). Finally, (c) follows from Theorem \ref{thm: Shokurov poly for minimal models} together with Lemma \ref{lem: movable give iso in codim 1}.

\medskip

Next, assuming that $X$ satisfies the above (a), (b), (c), and (d), we need to check (1), (2), (3), and (4) for the definition of MKD spaces (see Definition \ref{def: MKD spaces}). Note that (1) is exactly (a). 

First, we show that every effective $\bR$-Cartier divisor admits a good minimal model. By assumption, if $D$ is an effective $\bR$-Cartier divisor, then it admits a minimal model $h: X \dto X'$. Let $D'$ be the strict transform of $D$ on $X'$. We need to show that $D'$ is semi-ample. Let $p: W \to X, q: W \to X'$ be projective birational morphisms such that $h=q \circ p^{-1}$. Then we have 
\[
p^*D=q^*D'+E,
\] where $E \geq 0$ is a $q$-exceptional divisor such that $\Supp E$ contains the divisors contracted by $h$. Let $B=p_*(q^*D')$, which is an $\bR$-Cartier divisor as $X$ is $\bQ$-factorial. Moreover, we have $[B] \in \bMov^e(X)$ as $q^*(D'+\ep A)$ is a semi-ample divisor for any $\ep \in \Qq_{>0}$, where $A$ is an ample divisor on $X'$. By (b), there exists some $g \in \PsAut(X)$ such that $g_*B \in \Pi$. By (c) and (d), $g_*B$ admits a good minimal model $f: X \dto Y$ which is isomorphic in codimension $1$. Therefore,
\[
\theta: X  \stackrel{g}{\dashrightarrow} X  \stackrel{f}{\dashrightarrow} Y
\] is a good minimal model of $B$ by Lemma \ref{lem: iso in codim 1 preserves mm}. As $B=p_*(q^*D')$, we have   
\[
p^*B=q^*D'+F,
\] where $F$ is a $p$-exceptional divisor. As $-F \equiv q^*D'/X$ is nef over $X$, we have $F \geq 0$ by the negativity lemma. Replacing $W$ by a higher model, we can assume that there exists a morphism $r: W \to Y$ such that $\theta=r\circ p^{-1}$. Then we have
\[
p^*B=r^*B_Y+F_Y=q^*D'+F,
\] where $B_Y=\theta_*B$ and $F_Y \geq 0$ is an $r$-exceptional divisor.  By Lemma \ref{lem: common mm}, we have $F_Y = F$. Thus, $D'$ is semi-ample as $B_Y$ is semi-ample. This verifies (2).

By Lemma \ref{lem: movable give iso in codim 1}, we have $\bMov^e(X)=\Mov(X)$, hence (3) follows from (b).

Next, we show that $\Pi$ satisfies the local factoriality of canonical models. By (c), $\Pi$ is a finite union of $\Pi \cap f_i^*(\Nef(X_i))$, each of which is a rational polyhedral cone. Since $(f_{i,*}\Pi) \cap \Nef(X_i) $ satisfies the local factoriality of canonical models by Proposition \ref{prop: condition (4)} (1), the same holds for $\Pi \cap f_i^*(\Nef(X_i))$ by Proposition \ref{prop: condition (4)} (3). Hence, $\Pi$ satisfies the local factoriality of canonical models.

Finally, we show that $\Eff(X)$ satisfies the local factoriality of canonical models (i.e., Definition \ref{def: MKD spaces} (4)). The above discussion has checked conditions (1), (2), and (3) in Definition \ref{def: MKD spaces}. Moreover, we have shown that $\Pi$ satisfies the local factoriality of canonical models. As we are in the absolute case, $\Eff(X)$ is non-degenerate. Then the local factoriality of canonical models for $\Eff(X)$ follows from Corollary \ref{cor: local factoriality for Pi}. This verifies (4).
\end{proof}

\begin{remark}\label{rmk: existence of minimal model}
Suppose that $D$ is an effective divisor and $D = P + N$ is its Nakayama-Zariski decomposition, where $P$ is the positive part and $N$ is the fixed part. Then we have $P \in \bMov^e(X)$. By (b) and (c) of Theorem \ref{thm: two def are same}, there exists a good minimal model $X \dashrightarrow X'$ of $P$, which is isomorphic in codimension $1$. Let $D', P', N'$ denote the strict transforms of $D, P, N$ on $X'$. Then $D$ admits a good minimal model if and only if $D'$ does. Moreover, we have $D' = P' + N'$, which is the Nakayama-Zariski decomposition of $D'$, with $P'$ semi-ample as its positive part. In particular, $D$ birationally admits a Nakayama-Zariski decomposition with semi-ample positive part. 

In \cite{BH14}, it was shown that minimal models exist if a (usual) log pair birationally admits a Nakayama-Zariski decomposition with nef positive part. Thus, it is natural to ask whether the method of \cite{BH14} can be adapted to the above setting to remove the extra assumption that every effective divisor admits a minimal model. Pursuing this would likely require developing a complete theory in the present framework. Since Theorem \ref{thm: two def are same} serves only to justify that MKD spaces are analogues of Mori dream spaces, we do not pursue this direction here.
\end{remark}

\begin{remark}\label{rmk: weaken assumption for MKD}
The proof of Theorem \ref{thm: two def are same} shows that, at least in the absolute setting, the assumptions in the definition of MKD spaces can be slightly weakened by requiring only that
\begin{enumerate}
    \item effective divisors admit minimal models and divisors in $\bMov^e(X)$ admit good minimal models (instead of assuming that all effective divisors admit good minimal models), and
    \item $\Pi$ satisfies the local factoriality of canonical models (instead of requiring that $\Eff(X)$ satisfies the local factoriality of canonical models).
\end{enumerate}
\end{remark}

\subsection{Fibrations whose geometric generic fiber is an MKD space}\label{subsec: geometric MKD space}

In this section, we establish Theorem \ref{thm: geometric MKD space}, which constitutes the first step toward applying MKD fiber spaces to moduli problems.

In the sequel, we will use the following spreading-out and specialization techniques, whose underlying principle is well known (see, for example, \cite[Chapter 3.2]{Poo17}). The following result is taken from \cite[Lemma 4.2]{LZ25}.

\begin{lemma}[{\cite[Lemma 4.2]{LZ25}}]\label{lem: spread out and specialization}
Suppose that $X \to S$ is a morphism between varieties. Let $K = K(S)$ be the field of rational functions on $S$, and let $\bar K$ be the algebraic closure of $K$. If 
\[
\bar g \colon \bar Y \to X_{\bar K}
\] 
is a morphism of varieties over $\bar K$, and $\bar M$ is a coherent sheaf on $\bar Y$, then, after shrinking $S$, there exist a finite \'etale Galois base change $T \to S$, a variety $Y/T$, a morphism 
\[
g \colon Y \to X_T/T,
\] 
and a coherent sheaf $M$ on $Y$ such that $\bar Y \simeq Y_{\bar K}$, $\bar g = g_{\bar K}$, and $\bar M \simeq M_{\bar K}$.
\end{lemma}

In what follows, we write $\eta \coloneqq \Spec(K)$ and $\bar\eta \coloneqq \Spec(\bar K)$ for the generic point and the geometric generic point of $S$, respectively. Hence, $X_K$ and $X_{\bar K}$ are also denoted by $X_\eta$ and $X_{\bar\eta}$, respectively.

\begin{proof}[Proof of Theorem \ref{thm: geometric MKD space}]
By the standard spreading-out and specialization techniques (see Lemma \ref{lem: spread out and specialization}), there exists a generically finite morphism $T' \to T$ such that, after shrinking $T'$, the induced morphism $X_{T'} \to T'$ is a fibration and $X_{T'}$ is of klt singularities. For any generically finite morphism $T'' \to T$ factoring through $T' \to T$, after shrinking $T''$, this property remains valid.

\medskip

For (2), as $X_{T'}$ has klt singularities, there exists a small $\bQ$-factorial modification $Y \to X_{T'}$ (see \cite{BCHM10}). Thus, $Y_{\bar\eta} \to (X_{T'})_{\bar\eta}=X_{\bar\eta}$ is still a small modification. As $X_{\bar\eta}$ is a $\bQ$-factorial variety, we have $Y_{\bar\eta} \simeq (X_{T'})_{\bar\eta}$. Hence there exists a Zariski open set $T_0 \subset T'$ such that $Y_{T_0} \simeq X_{T_0}$. As $Y_{T_0}$ is still $\bQ$-factorial, after shrinking $T'$, we can assume that $X_{T'}$ is $\bQ$-factorial.

By Proposition \ref{prop: Generic property} (2), after shrinking $T'$, the natural linear map
\begin{equation}\label{eq: equal space}
\iota: N^1(X_{T'}/T') \to N^1(X_{\bar\eta}),\quad [D] \mapsto [D_{\bar\eta}]
\end{equation} is injective. Since $N^1(X_{\bar\eta})$ is a finite-dimensional vector space, by applying Lemma \ref{lem: spread out and specialization} to each of its generators, we may assume that \eqref{eq: equal space} becomes an isomorphism after a further generically finite base change of $T'$. Moreover, $\iota$ induces inclusions of the following cones
\begin{equation}\label{eq: inj for Mov and Eff}
\Mov(X_{T'}/T') \hookrightarrow \Mov(X_{\bar\eta}), \quad \Eff(X_{T'}/T') \hookrightarrow \Eff(X_{\bar\eta}).
\end{equation} Possibly after a further generically finite base change, we show that the above injections become isomorphisms.

As $X_{\bar\eta}$ is an MKD space, $\Mov(X_{\bar\eta})$ admits a rational polyhedral fundamental domain $\bar \Pi$ under the action of $\Gamma_B(X_{\bar\eta})$. By Lemma \ref{lem: spread out and specialization}, replacing $T' \to T$ by a further generically finite morphism and shrinking $T'$, we can assume that there exists a rational polyhedral cone $\Pi \subset \Mov(X_{T'}/T')$ such that $\iota(\Pi)=\bar\Pi$. As $\Mov(X_{\bar\eta})$ is a non-degenerate cone, by Theorem \ref{thm: finite presented}, $\Gamma_B(X_{\bar\eta})$ is finitely presented. Assume that $\bar \gamma_1, \cdots, \bar\gamma_l \in \PsAut(X_{\bar\eta})$ generate $\Gamma_B(X_{\bar\eta})$. By Lemma \ref{lem: spread out and specialization}, replacing $T' \to T$ by a further generically finite morphism and shrinking $T'$, there are $\gamma_1, \cdots, \gamma_l \in \PsAut(X_{T'}/T')$ such that $(\gamma_j)_{\bar\eta} = \bar\gamma_j, 1 \leq j \leq l$. We claim that 
\[
\Mov(X_{T'}/T') \hookrightarrow \Mov(X_{\bar\eta})
\] is a surjection. Let $[\bar D] \in  \Mov(X_{\bar\eta})$. Then there exists $[\bar B] \in \bar \Pi$ such that
\[
\bar\gamma \cdot [\bar B] = [\bar D],
\] where $\bar\gamma = \bar\gamma_i \cdots \bar\gamma_j \cdots \bar\gamma_k$ is a finite product of those generators. By the above construction, there exist $[B]\in \Pi$ with $\iota([B])=[\bar B]$ and elements $\gamma_j$ with $(\gamma_j)_{\bar\eta}=\bar\gamma_j$. Therefore, if 
\[
\gamma = \gamma_i \cdots \gamma_j \cdots \gamma_k
\] is the corresponding pseudo-automorphism in $\PsAut(X_{T'}/T')$, then we have 
\[
\iota(\gamma \cdot [B]) = \bar\gamma \cdot [\bar B] = [\bar D].
\] As $\PsAut(X_{T'}/T') \cdot \Pi \subset \Mov(X_{T'}/T')$, this shows the desired claim. This argument also implies
\begin{equation}\label{eq: fun for Mov}
\PsAut(X_{T'}/T') \cdot \Pi = \Mov(X_{T'}/T').
\end{equation} Note that the above argument remains valid for any generically finite morphism $T'' \to T$ factoring through $T' \to T$.

The surjection of $\Eff(X_{T'}/T') \to \Eff(X_{\bar\eta})$ can be shown by the same argument as above. In fact, the existence of a fundamental domain $\bar\Xi$ for $\Eff(X_{\bar\eta})$ under the action of $\Gamma_B(X_{\bar\eta})$ follows from Theorem \ref{thm: 3 equivalences} (3). Moreover, if $\Xi$ and $T'$ are defined similarly to the movable case above, then we still have
\begin{equation}\label{eq: cover eff1}
\PsAut(X_{T'}/T') \cdot \Xi = \Eff(X_{T'}/T').
\end{equation}

We now show that, after possibly replacing $T' \to T$ by a further generically finite morphism and shrinking $T'$, every effective $\bR$-Cartier divisor on $X_{T'}/T'$ admits a good minimal model over $T'$. By \eqref{eq: cover eff1} and Lemma \ref{lem: iso in codim 1 preserves mm}, it suffices to show the claim for effective $\bR$-Cartier divisors in $\Xi$. Let
\[
\bar\Xi=\bigsqcup_{i=1}^m \bar\Xi_i
\] be a disjoint union of finitely many relatively open rational polyhedral cones as in Theorem \ref{thm: Shokurov poly for minimal models}. Fix a birational map $\bar f_i \colon X_{\bar\eta} \dashrightarrow \bar Y_i$ associated with each $\bar\Xi_i$, such that for every effective divisor $\bar D$ with $[\bar D] \in \bar\Xi_i$, the map $\bar f_i$ is a good minimal model of $\bar D$. Let $\bar D_{i}^j, 1 \leq j\leq v$ be effective divisors such that each $[\bar D_{i}^j]$ is a vertex of the closure of $\bar\Xi_i$. Let $\bar\theta_i^j: \bar Y_i \to \bar Z_{i}^j$ be the morphism induced by $\bar D_i^j$. Then there exists an ample divisor $\bar A_{i}^j$ on $\bar Z_{i}^j$ such that $\bar D_{i}^j = (\bar\theta_i^j)^*(\bar A_{i}^j)$.

By Lemma \ref{lem: spread out and specialization}, replacing $T' \to T$ by a further generically finite morphism and shrinking $T'$, the above construction is modeled on $X_{T'}/T'$. To be precise, there exist effective divisors $D_{i}^j, 1 \leq j\leq v$ on $X_{T'}$, birational maps $f_i: X_{T'} \dto Y_i/T'$, morphisms $\theta_i^j: Y_i \to Z_{i}^j/T'$, and ample$/T'$ divisors $A_{i}^j$ on $Z_{i}^j$ such that $D_{i}^j = (\theta_i^j)^*(A_{i}^j)$, $(D_{i}^j)_{\bar\eta}=\bar D_{i}^j$, $(f_i)_{\bar\eta}=\bar f_i$, $(\theta_i^j)_{\bar\eta}=\bar\theta_i^j$, and $(A_{i}^j)_{\bar\eta}=\bar A_{i}^j$. After shrinking $T'$, this implies that $f_i$ is a good minimal model$/T'$ of any divisor of the form
\[
D=\sum_j r_i^j D_i^j, \quad \text{where~} \sum_j r_i^j = 1 \text{~and~} r_i^j > 0.
\] Note that 
\[
\Xi_i \coloneqq \left\{\sum_j r_i^j [D_i^j] ~\mid~ \sum_j r_i^j = 1 \text{~and~} r_i^j > 0, 1 \leq j \leq v\right\}
\] is the relatively open cone such that $\iota(\Xi_i)=\bar\Xi_i$. Hence, it suffices to show that if $B$ is an effective $\bR$-Cartier divisor on $X_{T'}$ such that $[B]=[D] \in \Xi_i$, then $f_i: X_{T'} \dto Y_i/T'$ is a good minimal model of $B$. Set $B_i\coloneqq (f_i)_*B$. We need to show that $B_i$ is semi-ample$/T'$.

By shrinking $T'$, we may assume that $T'$ is smooth and that there exists no very exceptional divisor on $Y_i$ over $T'$. This implies that if $E$ is a prime divisor on $Y_i$ which is vertical over $T'$, then 
\[
E=\tau_i^*(\tau_i(E)) \sim_\Qq 0/T',
\] where $\tau_i: Y_i \to T'$ is the natural morphism. As effective divisors on $X_{\bar\eta}$ admit good minimal models, we see that $(B_i)_{\bar\eta}$ is semi-ample. As $\bar\eta \to \eta$ is a flat morphism, we have the natural identification
\[
H^0(X_{\eta}, \Oo_{X_\eta}(m(B_i)_{\eta})) \otimes_{\eta} \bar\eta \simeq H^0(X_{\bar\eta}, \Oo_{X_{\bar\eta}}(m(B_i)_{\bar\eta}))
\] for any sufficiently divisible $m\in \Zz$ (see \cite[III Proposition 9.3]{Har77}). Therefore,  $(B_i)_{\eta}$ is semi-ample$/\eta$ and thus $B_i$ is semi-ample over an open subset $U \subset T'$. On the other hand, by the previous discussion, any vertical divisor is linearly trivial over $T'$. Hence, $B_i$ is semi-ample over $T'$.

The above discussion shows that $X_{T'}/T'$ is an MKD fiber space, except for the local factoriality of canonical models$/T'$ for $\Eff(X_{T'}/T')$. Since $\Eff(X_{T'}/T')$ is a non-degenerate cone, by Corollary \ref{cor: local factoriality for Pi} (see Remark \ref{rmk: Pi is fund of Eff is ok}), it remains to verify the local factoriality of canonical models$/T'$ for $\Xi$. 

By the above discussion, there exists a disjoint union of finitely many relatively open rational polyhedral cones
\[
\Xi=\bigsqcup_{i=1}^m \Xi_i,
\] 
such that each divisor $D$ with $[D] \in \Xi_i$ admits a good minimal model$/T'$ 
\[
f_i: X_{T'} \dto Y_i/T'.
\] Thus, $(f_i)_*\Xi_i \subset \Nef^e(Y_i/T')$ is a rational polyhedral cone which satisfies the local factoriality of canonical models$/T'$ by Proposition \ref{prop: condition (4)} (1). If $g: Y_i \to Z/T'$ is the canonical model$/T'$ of $(f_i)_*D$ for some $[D] \in \Xi_i$, then $g\circ f_i$ is the canonical model$/T'$ of $D$. Thus, $\Xi_i$ also satisfies the local factoriality of canonical models$/T'$ and so does $\Xi$. This completes the verification that $X_{T'}/T'$ is an MKD fiber space. Note that the above discussion remains valid for any generically finite morphism $T'' \to T$ factoring through $T' \to T$. This shows (2).

\medskip

Note that (3) has already been established in the proof of (2) (see the discussion following \eqref{eq: inj for Mov and Eff}).

\medskip

For (4), as $X_{\bar\eta}$ is projective, $\Eff(X_{\bar\eta})$ and $\Eff(X_{\bar\eta})$ are non-degenerate cones and so do $\Mov(X_{T'}/T')$ and $\Eff(X_{T'}/T')$ by (3). As shown in the proof of (2) (see \eqref{eq: fun for Mov}), there exists a rational polyhedral cone $\Pi \subset \Mov(X_{T'}/T')$ such that $\PsAut(X_{T'}/T') \cdot \Pi = \Mov(X_{T'}/T')$. By Proposition \ref{prop: prop-def}, we have $\PsAut(X_{T'}/T') \cdot \Pi = \Mov(X_{T'}/T')_+$. Hence, we have $\Mov(X_{T'}/T') = \Mov(X_{T'}/T')_+$.  
The equality $\Eff(X_{T'}/T') = \Eff(X_{T'}/T')_+$ follows by the same argument.
\end{proof}

\begin{remark}
It remains open whether the geometric generic fiber of a fibration $X \to T$ whose fibers are MKD spaces is itself an MKD space (cf. Question \ref{que: fiber to space}).
\end{remark}

\section{Variants of cones}\label{sec: variant of cones}

For a variety $X$, the nef cone encodes contraction morphisms, the movable cone encodes birational models that are isomorphic in codimension $1$, and the effective cone encodes birational contractions of $X$. From this viewpoint, various geometric properties of $X$ are reflected in different cones of divisors, and one may consider further variants of these cones. These variants exhibit especially simple behavior for MKD spaces, since good minimal models exist for every effective divisor, rather than only for adjoint divisors of the form $K_X+\Delta$.

Recall the definitions of the generic nef cone (see Definition \ref{def: generic nef cone})
\[
{\rm GNef}(X/T) \coloneqq \{[D] \in \Eff(X/T) \mid [D_\eta] \in \Nef(X_\eta)\},
\] and the generic automorphism group (see \eqref{eq: generic auto})
\[
{\rm GAut}(X/T) \coloneqq \{g\in \PsAut(X/T) \mid g_\eta \in \Aut(X_\eta)\}.
\] As before, $\Gamma_{GA}(X/T)$ is set to be the image of ${\rm GAut}(X/T)$ under the natural group homomorphism $\rho: \PsAut(X/T) \to {\rm GL}(N^1(X/T))$.

\begin{proof}[Proof of Theorem \ref{thm: cone for GNef}]
    We show (1) first. By Theorem \ref{thm: 3 equivalences}, there exists a rational polyhedral cone $P$ such that $P\subset \Eff(X/T)$ and $\PsAut(X/T) \cdot P =\Eff(X/T)$. Let $P = \sqcup_i P_i$ be a disjoint union of finitely many relatively open rational polyhedral cones satisfying Theorem~\ref{thm: Shokurov poly for minimal models}. Let 
\[
\iota_\eta: N^1(X/T) \to N^1(X_\eta), \quad [D] \mapsto [D_\eta]
\]
be the natural surjective linear map (see Proposition \ref{prop: Generic property} and Remark \ref{rmk: iota-K not need fibration}).  Set $P_{i,\eta} \coloneqq \iota_\eta(P_i)$.

If $N^1(X_\eta)=0$, then $X \to T$ is a generically finite morphism. Hence, we have ${\rm GNef}(X/T)=\Eff(X/T)$ and ${\rm GAut}(X/T)= \PsAut(X/T)$. Then (1) follows from Theorem \ref{thm: 3 equivalences}.

From now on, we assume that $N^1(X_\eta)\neq 0$. As $\Amp(X/T) \subset {\rm GNef}(X/T)$, ${\rm GNef}(X/T)$ is a full-dimensional cone. Moreover, we have 
\begin{equation}\label{eq: include in amp}
\iota_{\eta} \left(\Int({\rm GNef}(X/T))\right) \subset \Amp(X_\eta).
\end{equation}
Indeed, if $[D] \in \Int({\rm GNef}(X/T))$ and $A$ is an ample$/T$ divisor, then $[D - \ep A] \in \Int({\rm GNef}(X/T))$ for some $\ep \in \Rr_{>0}$. Hence, $[D_\eta - \ep A_\eta] \in \Nef(X_\eta)$, which implies that $[D_\eta] \in \Amp(X_\eta)$.

First, we claim that if $(g\cdot P_i)\cap {\rm GNef}(X/T) \neq \emptyset$ for some $g\in \PsAut(X/T)$, then 
\begin{equation}\label{eq: closure included}
g\cdot \bar P_i \subset {\rm GNef}(X/T),
\end{equation} 
where $\bar P_i$ is the closure of $P_i$. 

Indeed, suppose that $[g\cdot D] \in {\rm GNef}(X/T)$ such that $D$ is an effective divisor with $[D]\in P_i$. In particular, $(g\cdot D)_\eta$ is a nef divisor. By Theorem~\ref{thm: MMP for MKD}, we can run a $(g \cdot D)$-MMP$/T$, which terminates to a good minimal model$/T$ denoted by $f: X \dto Y/T$. As $(g\cdot D)_\eta$ is nef, $f$ is an isomorphism over the generic point of $T$. By Lemma \ref{lem: iso in codim 1 preserves mm}, $f\circ g$ is a good minimal model$/T$ of $D$. Therefore, by the property of $P_i$, $f\circ g$ is a good minimal model$/T$ of each effective divisor $B$ with $[B]\in P_i$. This implies that $(g \cdot B)_\eta$ is nef, and thus $g\cdot P_i \subset {\rm GNef}(X/T)$. As a limit of nef divisors is still nef, we have $g\cdot \bar P_i \subset {\rm GNef}(X/T)$.

Second, we claim that if 
    \[
    (g \cdot P_i) \cap\Int({\rm GNef}(X/T)) \neq \emptyset, \quad  (h \cdot P_i)\cap\Int({\rm GNef}(X/T)) \neq \emptyset
      \] for some $g,h \in \PsAut(X/T)$, then 
      \begin{equation}\label{eq: in GAut}
      g \circ h^{-1} \in {\rm GAut}(X/T).
      \end{equation} Indeed, there exists some $[D]\in P_i$, such that $(g\cdot D)_\eta$ is an ample divisor by \eqref{eq: include in amp}. By \eqref{eq: closure included}, we have $h\cdot P_i \subset {\rm GNef}(X/T)$. In particular, $(h\cdot D)_\eta$ is a nef divisor. Let 
    $f: X \dto Y/T$ and $f': X \dto Y'/T$ be good minimal models/$T$ of $g\cdot D$ and $h\cdot D$, respectively. By Theorem~\ref{thm: MMP for MKD}, we can assume that these good minimal models$/T$ are obtained through MMP$/T$. Hence, $f_\eta$ and $f'_\eta$ are isomorphisms. Note that $f'_*(h\cdot D)$  maps to $f_*(g\cdot D)$ through the natural map
    \[
  f \circ g \circ h^{-1} \circ f'^{-1}: Y' \dto Y/T.
    \] As $\left(f_*(g\cdot D)\right)_\eta$ is ample and $\left(f'_*(h\cdot D)\right)_\eta$ is nef, we see that $f \circ g \circ h^{-1} \circ f'^{-1}$ is a morphism over the generic point $\eta$ by Lemma \ref{lem: common mm}~(2). As $Y$ is $\bQ$-factorial (see Theorem~\ref{thm: MMP for MKD}), we see that 
$f \circ g \circ h^{-1} \circ f'^{-1}$ is an isomorphism over the generic point. 
Since $f_\eta$ and $f'_\eta$ are isomorphisms, it follows that 
$h \circ g^{-1}$ is an isomorphism over the generic point $\eta$. 
Hence, we have $g \circ h^{-1} \in {\rm GAut}(X/T)$.

For each $P_i$, fix one $g_i$ if there exists some $g_i \in \PsAut(X/T)$ such that 
\[
(g_i \cdot P_i) \cap \Int({\rm GNef}(X/T)) \neq \emptyset.
\] Let 
    \[
    Q\coloneqq \Cone(g_i\cdot \bar P_i \mid (g_i \cdot P_i)\cap\Int({\rm GNef}(X/T)) \neq \emptyset)
    \] be the cone generated by these finitely many cones. Then $Q$ is a rational polyhedral cone in ${\rm GNef}(X/T)$ by \eqref{eq: closure included}. We claim that
    \begin{equation}\label{eq: desired inclusion}
    {\rm GAut}(X/T) \cdot Q \supset \Int({\rm GNef}(X/T)).
    \end{equation} Indeed, since $\PsAut(X/T) \cdot P \supset \Eff(X/T)$, for each 
$[D] \in \Int({\rm GNef}(X/T))$, there exist some $P_i$, $[B] \in P_i$, and 
$h \in \PsAut(X/T)$ such that $h \cdot [B] = [D]$. By \eqref{eq: in GAut}, we have $h \cdot g_i^{-1}\in {\rm GAut}(X/T)$, thus
    \[
    [D]=h \cdot [B] = (h \cdot g_i^{-1} ) \cdot g_i \cdot [B] \in {\rm GAut}(X/T) \cdot Q.
    \]
    The above shows (1).

    \medskip
    
    For (2), note that if $E$ is the maximal subspace inside $\bEff(X/T)$ then it is also the maximal subspace inside the closure of ${\rm GNef}(X/T)$. By Proposition \ref{prop: max vector defined over Q}, $E$ is defined over $\bQ$. Then (2) follows from  (1) and Proposition \ref{prop: degenerate cone}.
    
    \medskip
    
    For (3), note that if ${\rm GNef}(X/T)$ is non-degenerate, then
    \[
    {\rm GNef}(X/T)_+= {\rm GAut}(X/T) \cdot Q 
    \] by Proposition \ref{prop: prop-def}. By (1), as $Q \subset {\rm GNef}(X/T)$, we have 
    ${\rm GAut}(X/T) \cdot Q ={\rm GNef}(X/T)$.
    
    \medskip
    
    For (4), let $Q = \sqcup_j Q_j$ be a disjoint union of finitely many relatively open rational polyhedral cones satisfying Theorem~\ref{thm: Shokurov poly for minimal models}. For each $j$, take $[D]\in Q_j$ and let $f_j: X \dto Y_j/T$ be a good minimal model$/T$ of $D$ which is obtained through a $D$-MMP$/T$. In particular, $f_j$ is an isomorphism over the generic point of $T$. Since $\bar Q_j$ is a rational polyhedral cone, there are only finitely many contraction morphisms
\begin{equation}\label{eq: finitely contractions}
Y_j\to Z_j^m/T,\qquad 1\leq m\leq l,
\end{equation}
corresponding to the finitely many faces of $\bar Q_j$.
    
    Suppose that $f: X \dto Y/T$ is a map which is a contraction morphism $f_U: X_U \to Y_U/U$ for some non-empty open subset $U\subset T$. Let $A$ be an ample$/T$ divisor on $Y$. Let $p: W \to X$ be a projective birational morphism and $q: W \to Y$ be a projective morphism such that $f = q \circ p^{-1}$. Set $H \coloneqq p_* q^* A$. Then we have $[H] \in {\rm GNef}(X/T)$.  By (1), there exists some $\sigma\in {\rm GAut}(X/T)$ such that $[\sigma_*H]\in Q_j$.  Then $(f_j)_*(\sigma_*H)$ is semi-ample$/T$ on $Y_j$, and over the generic point of $T$, the contraction induced by $(f_j)_*(\sigma_*H)$ is among the finitely many contractions in \eqref{eq: finitely contractions}. By construction, $\sigma_\eta$ and $(f_j)_\eta$ are both isomorphisms. Therefore, we have $Y_\eta \simeq (Z_j^m)_\eta$ for some $m$.
\end{proof}

Similarly to the generic nef cone, we can consider the following cones:
\[
\begin{split}
{\rm GMov}(X/T) &\coloneqq \{[D] \in \Eff(X/T) \mid [D_\eta] \in \Mov(X_\eta)\},\\
{\rm MNef}(X/T) &\coloneqq \{[D] \in \Mov(X/T) \mid [D_\eta] \in \Nef(X_\eta)\}.
\end{split}
\] They are equipped with group actions by $\PsAut(X/T)$ and ${\rm GAut}(X/T)$, respectively.

Since these cones are not needed in the subsequent discussions, we simply record the following theorems, which are analogous to Theorem~\ref{thm: cone for GNef} and can be proved by a similar argument.

\begin{theorem}
    Let $X/T$ be an MKD fiber space. 
    \begin{enumerate}
    \item There is a rational polyhedral cone $Q\subset {\rm GMov}(X/T)$ such that $\Gamma_{B}(X/T) \cdot Q \supset {\rm GMov}(X/T)$. 
    \item ${\rm GMov}(X/T)_+$ admits a weak rational polyhedral fundamental domain under the action of $\Gamma_{B}(X/T)$. 
    \item If ${\rm GMov}(X/T)$ is non-degenerate, then ${\rm GMov}(X/T)_+={\rm GMov}(X/T)$.
    \end{enumerate}
\end{theorem}

\begin{theorem}
    Let $X/T$ be an MKD fiber space. 
    \begin{enumerate}
    \item There is a rational polyhedral cone $Q\subset {\rm MNef}(X/T)$ such that $\Gamma_{GA}(X/T) \cdot Q \supset {\rm MNef}(X/T)$. 
    \item ${\rm MNef}(X/T)_+$ admits a weak rational polyhedral fundamental domain under the action of $\Gamma_{GA}(X/T)$. 
    \item If ${\rm MNef}(X/T)$ is non-degenerate, then ${\rm MNef}(X/T)_+={\rm MNef}(X/T)$.
    \end{enumerate}
\end{theorem}

\begin{remark}
Theorem~\ref{thm: 3 equivalences} can also be extended to incorporate the cones ${\rm GNef}(X/T)$, ${\rm GMov}(X/T)$, and ${\rm MNef}(X/T)$, yielding analogous equivalence statements.
\end{remark}

\section{Examples and open questions on MKD spaces}\label{sec: examples and questions}

\subsection{Examples of MKD spaces}\label{subsec: examples}

To construct examples of MKD spaces that are neither Mori dream spaces nor of Calabi-Yau type, one may use the following proposition. The argument draws on ideas from \cite{Nam91} and \cite[Proposition~7.1]{GLSW26}.

\begin{proposition}\label{prop: product of MKD spaces}
Let $X, Y$ be MKD spaces with natural projections $p: X \times Y \to X$ and $q: X \times Y \to Y$. Suppose that $X \times Y$ is $\bQ$-factorial and $p^*N^1(X)+q^*N^1(Y) = N^1(X \times Y)$, then $X \times Y$ is an MKD space. 
\end{proposition}
\begin{proof}
We claim that 
\begin{equation}\label{eq: prod of eff}
p^*\Eff(X) +q^*\Eff(Y)= \Eff(X \times Y).
\end{equation} 

It suffices to show the inclusion ``$\supset$''. Let $D$ be an effective $\bQ$-Cartier divisor on $X \times Y$. By assumption, we have $D \equiv p^*A+q^*B$. For a projective variety $Z$, let $a_Z: Z \to {\rm Alb}(Z)$ denote the Albanese morphism. 
Since 
\[
 a_{X \times Y}: X \times Y \to {\rm Alb}(X \times Y)= {\rm Alb}(X) \times  {\rm Alb}(Y)
\] is the product of $a_X$ and $a_Y$, every numerically trivial divisor on $X \times Y$ pulls back from the product of Albanese varieties. Thus, we may write
\[
D \sim_{\bQ} (p^*A + q^*B) + a_{X \times Y}^*N
    = p^*\bigl(A + a_X^*N_X\bigr) 
      + q^*\bigl(B + a_Y^*N_Y\bigr),
\]
where $N_X \equiv 0$ on $X$ and $N_Y \equiv 0$ on $Y$. Replacing $A$ and $B$ by $A+a_X^*N_X$ and $B+a_Y^*N_Y$ respectively, we have $D \sim_\bQ p^*A+q^*B$. After replacing $D, A$, and $B$ by suitable positive multiples, we may assume that 
they are all Cartier divisors, and that
\[
D \sim p^*A + q^*B
\]
with $H^0(X \times Y, \mO(D)) \neq 0$. Denote by $\phi: X \times Y \to \Spec (\bC)$, $\psi: X \to \Spec (\bC)$ and $\nu: Y \to \Spec (\bC)$ the natural morphisms. By the projection formula, we have
\[
\begin{split}
0 \neq H^0(X \times Y, \mO(D)) &\simeq \phi_*(p^*\mO(A) \otimes q^*\mO(B))\\
&=\psi_*(\mO(A) \otimes p_*q^*\mO(B))=\psi_*(\mO(A) \otimes \psi^*\nu_*\mO(B))\\
&=\psi_*\mO(A) \otimes \nu_*\mO(B),
\end{split}
\] where $p_*q^*\mO(B) = \psi^*\nu_*\mO(B)$ follows from the flat base change theorem. In particular, we have $\psi_*\mO(A) \neq 0$ and $\nu_*\mO(B) \neq 0$ and thus $A$ and $B$ are effective divisors. This shows the claim.

By Theorem \ref{thm: 3 equivalences}, as $X$ and $Y$ are MKD spaces, there exist rational polyhedral cones $P_X\subset \Eff(X)$ and $P_Y\subset \Eff(Y)$ such that $\Gamma_B(X) \cdot P_X = \Eff(X)$ and $\Gamma_B(Y) \cdot P_Y = \Eff(Y)$. Let
\[
\Pi \coloneqq \Cone(p^*P_X,\; q^*P_Y) \subset \Eff(X \times Y)
\]
be the cone generated by $p^*P_X$ and $q^*P_Y$. Then $\Pi$ is a rational polyhedral cone.  By \eqref{eq: prod of eff}, we have
 \begin{equation}\label{eq: example cono action}
\Gamma_{B}(X \times Y) \cdot \Pi = \Eff(X \times Y).
 \end{equation}

Next, we check that every effective $\bR$-Cartier divisor on $X \times Y$ admits a good minimal model.

By \eqref{eq: prod of eff}, if $D$ is an effective $\Rr$-Cartier divisor on $X \times Y$, then there exist an effective $\Rr$-Cartier divisor $A$ on $X$ and an effective $\Rr$-Cartier divisor $B$ on $Y$, such that 
\[
D \sim_\bR p^*A+q^*B.
\]
Let $f: X \dashrightarrow X'$ and $g: Y \dashrightarrow Y'$ be good minimal models of $A$ and $B$, respectively. We claim that 
\begin{equation}\label{eq: splitting of MMP}
f \times g : X \times Y \dashrightarrow X' \times Y'
\end{equation}
is a good minimal model of $D$. Let $s: W \to X$ and $s': W \to X'$ be birational morphisms such that $f = s' \circ s^{-1}$. Similarly, let $t: V \to Y$ and $t': V \to Y'$ be birational morphisms such that $g = t' \circ t^{-1}$. Then
\[
s^*A = s'^*A' + E, \qquad t^*B = t'^*B' + F,
\]
where $A', B'$ are the strict transforms of $A, B$, and $E, F$ are effective divisors whose supports contain $\Exc(f), \Exc(g)$, respectively. Hence,
\[
(s \times t)^*(p^*A + q^*B)
= (s' \times t')^*(p'^*A' + q'^*B') + \Theta,
\]
where $p' : X' \times Y' \to X'$ and $q' : X' \times Y' \to Y'$ are the natural projections, and $\Theta$ is an effective divisor whose support contains 
\[
\Supp(E) \times Y \quad \text{and} \quad X \times \Supp(F).
\]
In particular, $\Supp(\Theta)$ contains $\Exc(f \times g)$. Therefore, $f \times g$ is a good minimal model of $D$.

Finally, local factoriality of canonical models for $\Eff(X\times Y)$ follows from the corresponding properties for $\Eff(X)$ and $\Eff(Y)$, together with \eqref{eq: prod of eff} and \eqref{eq: splitting of MMP}.

By Theorem \ref{thm: 3 equivalences} (3), it follows from \eqref{eq: example cono action} that $X \times Y$ is an MKD space.
\end{proof}

By Proposition~\ref{prop: product of MKD spaces}, one obtains many MKD spaces that are neither Mori dream spaces nor of Calabi-Yau type.

\begin{example}\label{eg: product}
Let $X$ be a smooth Mori dream space of general type. Such an $X$ can be taken as a smooth variety of general type with Picard number $1$, or as a smooth hypersurface in a certain product of Mori dream spaces (see \cite{Jow11}). Let $Y$ be a K3 surface such that $\Nef(Y)$ is not a polyhedral cone. Both $X$ and $Y$ are MKD spaces. As $h^1(Y, \Oo_Y)=0$, we have 
\begin{equation}\label{eq: split NS}
p^*N^1(X)+q^*N^1(Y) = N^1(X \times Y).
\end{equation} 
By Proposition \ref{prop: product of MKD spaces}, $X \times Y$ is still an MKD space. 

It follows from \eqref{eq: split NS} that $\Nef(X \times Y)=p^*\Nef(X)+q^*\Nef(Y)$. As $\Nef(X)$ is a polyhedral cone and $\Nef(Y)$ is not a polyhedral cone, we see that $\Nef(X \times Y)=p^*\Nef(X)+q^*\Nef(Y)$ is not a polyhedral cone. Thus, $X \times Y$ cannot be a Mori dream space. As the Kodaira dimension $\kappa(X \times Y)=\kappa(X)>0$, $X \times Y$ cannot be of Calabi-Yau type either.
\end{example}

One can also consider quotients of the above products under suitable conditions. This construction is analogous to the construction of bi-elliptic surfaces.

\begin{example}
Let 
\[
X = \{x_0^6+x_1^6+x_2^6+x_3^6=0\} \subset \mathbb{P}^3
\] 
be a Fermat-type hypersurface. We have $\rho(X)=1$ and $h^1(X, \mathcal{O}_X)=0$. Let $E_1, E_2$ be two non-isogenous elliptic curves without complex multiplication. Set $Y= E_1 \times E_2$, which is an abelian surface. Let $E_1(2)$ and $E_2(2)$ be torsion points of order $2$ on $E_1$ and $E_2$, respectively. Then 
\[
G \coloneqq E_1(2) \times E_2(2)
\] 
is a finite group isomorphic to $(\mathbb{Z}/2\Zz)^4$. 
Let $\langle\xi\rangle \simeq \mathbb{Z}/2\mathbb{Z}$ be a direct factor of $G \simeq (\mathbb{Z}/2\mathbb{Z})^4$. 
Let $\langle\xi\rangle$ act on $X$ by
\[
\xi\cdot[x_0:x_1:x_2:x_3]=[-x_0:x_1:x_2:x_3].
\]
This extends to an action of $G$ on $X$ by requiring that the remaining factors act trivially. Note that $G$ naturally acts on $Y$ by translation, denoted by
\[
t_\sigma: y \mapsto y+\sigma, \quad \sigma\in G, ~y\in Y.
\] Set
\[
W \coloneqq (X \times Y)/G,
\]
where the action
\[
G \times (X \times Y) \to X \times Y, \qquad 
\sigma \cdot (x,y) = (\sigma\cdot x,~ t_\sigma(y))
\]
is the diagonal $G$-action. 

We show that $W$ is an MKD space. 

As $G$ acts freely on $Y$, the diagonal action of $G$ on $X \times Y$ is also free. Hence, the quotient $W = (X \times Y)/G$ is a smooth variety.

We claim that 
\begin{equation}\label{eq: quotient split}
\Eff(W)=\ti p^*\Eff(X/G) + \ti q^*\Eff(Y/G),
\end{equation} where $\ti p: W \to X/G$ and $\ti q: W \to Y/G$ are natural morphisms. Indeed, it suffices to show the inclusion ``$\subset$''. Let $D$ be an effective $\Qq$-Cartier divisor on $W$, and set $\ti D = \phi^*D$, where $\phi: X \times Y \to W$ is the quotient map. As $X$ and $Y$ satisfy Proposition \ref{prop: product of MKD spaces}, by \eqref{eq: prod of eff}, we have
\[
\ti D \sim_{\mathbb{Q}} p^*A + q^*B,
\]
where $A, B$ are effective $\mathbb{Q}$-Cartier divisors on $X, Y$, respectively, and
\(p: X \times Y \to X\), \(q: X \times Y \to Y\) are the natural projections. Therefore, we have
\[
\sum_{g\in G} g\cdot \tilde D 
   \sim_\Qq 
   p^*\!\left(\sum_{g\in G} g\cdot A\right)
   + 
   q^*\!\left(\sum_{g\in G} g\cdot B\right).
\]
Since $\tilde D$ is $G$-invariant, we obtain 
\[
\sum_{g\in G} g\cdot \tilde D = 16\, \tilde D.
\]
As $\sum_{g\in G} g\cdot A$ and $\sum_{g\in G} g\cdot B$ are also $G$-invariant, they descend to effective $\Qq$-Cartier divisors on $X/G$ and $Y/G$, respectively. Hence, there exist effective $\Qq$-Cartier divisors $A'$ on $X/G$ and $B'$ on $Y/G$ such that
\[
D \sim_\Qq \tilde p^*A' + \tilde q^* B'.
\]
This shows \eqref{eq: quotient split}.

As $\rho(X)=1$, we see that $\rho(X/G)=1$. Indeed, if $\nu: X \to X/G$ is the quotient map, then $\nu_*(\nu^*A) = 2 A$, where $A$ is a Cartier divisor on $X/G$. As the cone conjecture holds for the abelian surface $Y/G$, there exists a rational polyhedral cone  $\ti \Pi \subset \Eff(Y/G)$ such that $\Aut(Y/G) \cdot \ti \Pi = \Eff(Y/G)$. Note that we have $\Aut(Y/G)= \PsAut(Y/G)$. Let 
\[
\Pi \coloneqq \Cone(\ti p^*\Eff(X/G), \ti q^*\ti \Pi) \subset \Eff(W)
\] be the cone generated by the rational polyhedral cones $\ti p^*(\Eff(X/G))$ and $\ti q^*\ti \Pi$. Hence, $\Pi$ is still a rational polyhedral cone. 

We claim that
\begin{equation}\label{eq: cone on W}
\PsAut(W) \cdot \Pi = \Eff(W).
\end{equation}

Note that $G$ is the kernel of the multiplication-by-$2$ map
\[
Y \to Y, \quad y \mapsto 2y.
\] Hence, we have $Y/G \simeq Y$. The group of automorphisms of the abelian surface $Y$ is
\[
\Aut_{\rm gp}(Y) = \{\pm 1\} \times \{\pm 1\}.
\] Thus, we also have $\Aut_{\rm gp}(Y/G) = \{\pm 1\} \times \{\pm 1\}$. As $\Aut_{\rm gp}(Y)$  is invariant under $G=E_1(2) \times E_2(2)$, we see that any element of $\Aut_{\rm gp}(Y)$ descends to an element of $\Aut_{\rm gp}(Y/G)$. Since $\Aut(Y/G) = \Aut^0(Y/G) \rtimes \Aut_{\rm gp}(Y/G)$, it induces a surjective map
\begin{equation}\label{eq: surj aut}
\Aut(Y) \to \Aut(Y/G).
\end{equation}
In other words, for any $\ti g\in \Aut(Y/G)$, there exists some $g\in \Aut(Y)$ such that $\ti g \circ \mu = \mu \circ g$, where $\mu: Y \to Y/G$ is the quotient map. This $g$ induces an automorphism
\[
g': X \times Y/G \to X \times Y/G, \quad [(x,y)] \mapsto [(x, g(y))],
\] which descends to $\ti g$ on $Y/G$ under the natural morphism $\ti q: X \times Y/G \to Y/G$. Indeed, for any $\sigma\in G$ and $g\in \Aut(Y)$, we have
\[
gt_\sigma = t_{g(\sigma)} g.
\] Since $\Aut_{\rm gp}(Y) \simeq (\pm 1) \times (\pm 1)$ and $\sigma$ is of order at most $2$, if $g\in \Aut_{\rm gp}(Y)$, then we have
\[
g(\sigma)=\sigma.
\] If $g\in \Aut^0(Y)$, then we have $gt_\sigma = t_{\sigma} g$. Hence, by $\Aut(Y) = \Aut^0(Y)\rtimes \Aut_{\rm gp}(Y)$, for any $g\in \Aut(Y)$, we have
\[
gt_\sigma = t_{\sigma} g.
\] Therefore, for any $x\in X$ and $y\in Y$, we have
\[
[(\sigma \cdot x, ~g(t_\sigma(y)))]= [(\sigma \cdot x, ~t_\sigma (g(y)))]=[(x, ~g(y))] \in X \times Y/G. 
\] That is, $g'$ is well-defined. From this, we see that $g'\in \Aut(W)$ and $\ti q \circ g' = \ti g \circ \ti q$. 

To show \eqref{eq: cone on W}, by \eqref{eq: quotient split}, we can assume that an effective divisor on $W$ is of the form $\tilde p^*A' + \tilde q^* B'$, where $A'$ and $B'$ are effective divisors on $X/G$ and $Y/G$, respectively. Hence, there exists some $\ti g \in \Aut(Y/G)$ such that $\ti g \cdot [B'] \in \ti \Pi$. Let $g'\in\Aut(W)$ be the automorphism constructed above. We have
\[
g'\cdot (\tilde p^*[A'] + \tilde q^* [B'])=g'\cdot (\tilde p^*[A'])+\ti q^*(\ti g \cdot [B'])= \tilde p^*[A'] +\ti q^*(\ti g \cdot [B']),
\] which lies in $\ti p^*\Eff(X/G)+\ti q^*\ti \Pi$. Thus, \eqref{eq: cone on W} follows. 

\medskip

Finally, we show that the existence of good minimal models and the local factoriality of canonical models for $\Eff(W)$ also hold. 

As $X$ and $Y$ satisfy Proposition \ref{prop: product of MKD spaces}, if 
$f: X \dashrightarrow X'$ and $g: Y \dashrightarrow Y'$ are good minimal models of $A$ and $B$, respectively 
(indeed, $X = X'$ as $\rho(X)=1$), then 
\[
f \times g : X \times Y \dashrightarrow X' \times Y'
\]
is a good minimal model of $p^*A+q^*B$ by \eqref{eq: splitting of MMP}.

Moreover, if $A = \nu^* A'$ and $B = \mu^* B'$, then, by running a $G$-equivariant MMP, we may assume that $X' \times Y'$ admits a $G$-action compatible with that of $X \times Y$ (see \cite[\S 2.2]{KM98} for a discussion of $G$-equivariant MMP). Hence, $X' \times Y'/G$ is a good minimal model of $\ti p^* A' + \ti q^* B'$. As the strict transform of $p^*(\nu^* A') + q^*(\mu^* B')$ is $G$-equivariant, the contraction 
\[
X' \times Y' \to Z
\] 
induced by this divisor is also $G$-equivariant, which gives the canonical model 
\[
W =X' \times Y'/G\dashrightarrow Z/G
\] 
of $\ti p^* A' + \ti q^* B'$. The local factoriality of canonical models for 
\[
\Eff(W) = \Eff(X \times Y)^G
\] 
also follows from this.

By Theorem \ref{thm: 3 equivalences} (3), this shows the claim that $W$ is an MKD space.

Note that $\kappa(W)=\kappa(X \times Y)=\kappa(X)=2$, so $W$ is not of Calabi-Yau type. Since the nef cone $\Nef(Y) \subset N^1(Y) \simeq \Rr^3$ is not a polyhedral cone, by $Y/G \simeq Y$,
$\Nef(Y/G)$ is not a polyhedral cone either. From \eqref{eq: quotient split}, we have 
\[
\Nef(W)=\tilde p^*\Nef(X/G) + \tilde q^*\Nef(Y/G).
\] 
Therefore, $\Nef(W)$ is not a polyhedral cone, which implies that $W$ is not a Mori dream space.
\end{example}

\subsection{Open questions on MKD fiber spaces}\label{subsec: questions about MKD}

\begin{question}
Is there an intrinsic way to characterize an MKD space by its Cox ring?
\end{question}

One possible way to characterize an MKD space in the spirit of the Cox ring is as follows.
\begin{enumerate}
    \item Every effective divisor admits a good minimal model.
    \item There exists a rational polyhedral cone $\Pi $ such that $\Pi$ is a fundamental domain of $\Mov(X)$ under the action of $\Gamma_B(X)$.
    \item $R(X, \Pi)$ is a finitely generated $\bC$-algebra.
\end{enumerate}
However, it would be desirable to have a more intrinsic characterization, that is, one not involving the choice of $\Pi$ explicitly.

\begin{question}
Let $X \to Y$ be a fibration from an MKD space $X$. Is $Y$ still an MKD space? 
\end{question}

Note that if $X$ is of Calabi-Yau type, then $Y$ is still of Calabi-Yau type by the canonical bundle formula, and if $X$ is a Mori dream space, then $Y$ is still a Mori dream space by \cite{Oka16}.

\begin{question}\label{que: fiber to space}
Let $X/T$ be a fiber space. If $X_s$ is an MKD space for each closed point $s\in T$, then is $X/T$ an MKD fiber space after a generically finite base change of $T$?
\end{question}

It is straightforward to show that if $X_s$ is a Calabi-Yau variety (resp. Fano type variety) for every $s\in T$, then $X/T$ is a Calabi-Yau fiber space (resp. Fano type fiber space) after shrinking $T$ (cf. Corollary \ref{cor: Mori dream space, CY are MKD space}). In Question \ref{que: fiber to space}, instead of considering all $s\in T$, one can also consider a Zariski dense subset of $T$. This relates to the moduli problem of MKD spaces. For a related question for Fano type varieties, see \cite[Question 1.1]{CLZ25}. Note that \cite{Lut24} studies the Morrison-Kawamata cone conjecture under deformations.

\begin{question}
Suppose that $X/T$ is an MKD fiber space (or $X_{T'}/T'$ is an MKD fiber space for any generically finite base change $T' \to T$). Then is there a Zariski open subset $T_0 \subset T$, such that each fiber $X_s$ is an MKD space for $s\in T_0$.
\end{question}

The above question is related to Theorem \ref{thm: fiber MKD} (see Remark \ref{rk: almost MKD}).

\begin{question}
How about the moduli problem of MKD spaces? Are there natural invariants that determine the boundedness of the moduli spaces of MKD spaces? 
\end{question}

\section{Applications of MKD fiber spaces}\label{sec: deformation of cones of MKD spaces}

In this section, we apply the theory of MKD fiber spaces to the deformation of various cones and to the boundedness of moduli problems.

\subsection{Deformation of cones of MKD spaces}

Our study of the deformation of cones of MKD spaces follows the lines of \cite{CLZ25}, with simplifications provided by the general theory developed in Section~\ref{sec: variant of cones}.

\subsubsection{Collected results on the deformation of N\'eron-Severi spaces}\label{subsubsec: def of NS spaces}

Recall that for a variety $X$ over a variety $T$, $\eta$ denotes the generic point of $T$ and $\bar\eta$ denotes $\Spec \overline{k(\eta)}$, where $\overline{k(\eta)}$ is the algebraic closure of the rational function field $k(\eta)$. Then $X_{\bar \eta} \coloneqq X \times_T \Spec \overline{k(\eta)}$ is the geometric generic fiber. A point is said to be a very general point of $T$ if it lies in the complement of countably many Zariski closed subsets of $T$.

We call a sequence of birational contractions
\[
X=X_0 \dto X_1 \dto \cdots \dto X_{n}/T
\] a partial MMP$/T$ (with respect to a divisor $D$) if each step consists of $D_i$-extremal divisorial or flipping contractions, and $D_i$-flips, where $D_i$ is the strict transform of $D$ on $X_i$. In the above, the corresponding extremal contraction has relative Picard number $1$. To be precise, if $X_i \to X_{i+1}/T$ is a $D_i$-flipping contraction, then $X_{i+2} \to X_{i+1}/T$ is its $D_i$-flip. Besides, in the above partial MMP, $D_n$ is not required to be nef over $T$. On the other hand, if $D_{n}$ becomes nef and big over $T$, and $X_{n} \to X_{n+1}$ is the birational contraction induced by $D_n$ (in this case, the relative Picard number of $\rho(X_n/X_{n+1})$ may be greater than $1$), then
\[
X=X_0 \dto X_1 \dto \cdots \dto X_{n}\to X_{n+1}/T
\] is also called a partial MMP$/T$. 

In the sequel, for simplicity, when we say that a sequence $X\dto X_n/T$ is an MMP$/T$ (with respect to a divisor $D$), we mean a partial MMP$/T$ such that the strict transform $D_n$ of $D$ on $X_n$ is nef$/T$. In other words, we also allow the last birational contraction $X_{n-1}\to X_n$, induced by the possibly big and semi-ample divisor $D_{n-1}$, to be included in an MMP$/T$.

Let $f: X \to T$ be a fibration. Let ${\mP}ic_{X/T}$ be the sheaf associated to the relative Picard functor
\[
S \mapsto \Pic(X_S)_{\bZ}/\Pic(S)_{\bZ}=\Pic(X_S/S)_{\bZ},
\] where $S\subset T$ is a Zariski open subset. See \cite[\S 9.2]{Kle05} for details. Note that ${\mP}ic_{X/T}$ is denoted by $\Pic_{(X/T)(\rm{zar})}$ in \cite[Definition 9.2.2]{Kle05}. In general, ${\mP}ic_{X/T}(U)$ may not be $\Pic(X_U/U)_{\bZ}$ for an open subset $U\subset T$ because of the sheafification.  By \cite[(9.2.11.2)]{Kle05}, we always have
\[
{\mP}ic_{X/T}(U) = H^0(U, R^1f_*\mO^*_{X_U}).
\]

\begin{theorem}[{\cite[Theorem 1.3]{CLZ25}}]\label{thm: CLZ-deform pic}
Assume that $f: X \to T$ is a fibration with $(X, \Delta)$ a klt pair for some effective $\mathbb{R}$-divisor $\Delta$ on $X$. Suppose that $S\subset T$ is a Zariski dense subset such that for any $s\in S$, the fiber $X_s$ satisfies
\[
H^1(X_s, \mO_{X_s})=H^2(X_s, \mO_{X_s})=0.
\] 
\begin{enumerate}
    \item If the natural restriction map $N^1(X/T) \to N^1(X_t)$ is surjective for very general $t\in T$, then there exists a non-empty open subset $T_0\subset T$, such that ${\mP}ic_{X_{T_0}/T_0} \otimes \bR$ is a constant sheaf in the Zariski topology.
    \item Up to a generically finite base change of $T$, there exists a non-empty open subset $T_0\subset T$, such that ${\mP}ic_{X_{T_0}/T_0} \otimes \bR$ is a constant sheaf in the Zariski topology. 
\end{enumerate} 
Moreover, in both of the above cases, for any open subset $U\subset T_0$, the natural restriction maps
\[
N^1(X_{T_0}/T_0) \to N^1(X_U/U) \to N^1(X_t),\quad [D] \mapsto [D|_{X_U}] \mapsto [D|_{X_t}]
\] are isomorphisms for any $t\in U$.
\end{theorem}

In the above theorem, ``up to a generically finite base change of $T$'' means that we take a generically finite morphism $T' \to T$ and replace $T$ by $T'$. The non-empty open subset $T_0$ is then taken inside this new base $T'$. Hence, after possibly further shrinking $T_0$, we are indeed performing an \'etale base change over an open subset of the original base.

\begin{lemma}\label{lem: total to fiber is inj}
Let $X \to T$ be a fibration. Then there exists an open subset $T_0 \subset T$ such that for any open subset $U \subset T_0$, the natural restriction map
\[
N^1(X_{U}/U) \longrightarrow N^1(X_t), \quad [D] \longmapsto [D|_{X_t}],
\]
is injective for every $t \in U$.
\end{lemma}
\begin{proof}
Let $g: Y \to X$ be a resolution. Shrinking $T$ if necessary, we may assume that $Y \to T$ is a smooth morphism by generic smoothness. Since
\[
N^1(X/T) \to N^1(Y/T),\quad [D] \mapsto [g^*D]
\qquad\text{and}\qquad
N^1(X_t) \to N^1(Y_t),\quad [B] \mapsto [g_t^*B]
\]
are both injective, and the following diagram of natural maps is commutative:
\[
\begin{tikzcd}
N^1(Y/T) \arrow[r] & N^1(Y_t) \\
N^1(X/T) \arrow[u, hookrightarrow] \arrow[r] & N^1(X_t) \arrow[u, hookrightarrow],
\end{tikzcd}
\]
after replacing $Y$ by $X$, it suffices to assume that $X$ is smooth and $f \colon X \to T$ is a smooth morphism.

Let $D_i, 1 \leq i \leq n$ be prime divisors on $X$ such that $\{[D_i] \mid 1 \leq i \leq n\}$ spans the vector space $N^1(X/T)$. By generic flatness, there exists an open subset $T_0\subset T$, such that each $D_i$ is flat over $T_0$. Note that for any open subset $U \subset T_0$, the natural map
\[
N^1(X/T) \to N^1(X_{U}/U), \quad [D] \mapsto [D|_{X_U}]
\] is surjective as $X$ is smooth. Hence, $\{[D_i|_{X_U}] \mid 1 \leq i \leq n\}$ still spans the vector space $N^1(X_U/U)$. 

We claim that the restriction map
\[
\iota: N^1(X_U/U) \longrightarrow N^1(X_t), \qquad [D] \longmapsto [D|_{X_t}],
\]
is injective for every $t \in U$. 

Since $\iota$ is a linear map defined over $\bQ$, we may assume that there exists a divisor
\[
D = \sum_{1 \leq i \leq n} a_i D_i|_{X_U}, \qquad a_i \in \Qq
\]
such that $[D|_{X_t}] = 0$ in $N^1(X_t)$. As $N^1(X_t)_\bQ = H^{1,1}(X_t, \Cc) \cap H^2(X_t, \Qq)$, we have $[D|_{X_t}]=0 \in H^2(X_t, \Qq)$. Since $f: X \to T$ is a smooth morphism, by Ehresmann's fibration theorem, the map $f$ is a locally trivial fibration in the Euclidean topology. For any $t' \in U$, by connecting $t$ and $t'$ through a closed path, we deduce that
\[
[D|_{X_{t'}}] = 0 \in H^2(X_{t'}, \Qq).
\] Indeed, since each $D_i$ is flat over $U$, the classes $[D_i|_{X_t}]$ and $[D_i|_{X_{t'}}]$ can be identified under the natural identification of $H^2(X_t,\Qq)$ and $H^2(X_{t'},\Qq)$ along the path. This implies that $[D|_{X_{t'}}] = 0 \in N^1(X_{t'})$ for every $t' \in U$. Consequently, $[D] = 0 \in N^1(X_U/U)$. This proves the claim and completes the proof.
\end{proof}

The following result is proved in \cite[Lemma~5.3~(2)]{CLZ25}. Note that the extra assumption in \cite[Lemma~5.3~(2)]{CLZ25} that $X$ is of Fano type over $T$ guarantees the existence of a universal open subset $T_0\subset T$ for all partial MMPs over $T$. Since we only consider a fixed partial MMP$/T$, this assumption is not needed here.

\begin{lemma}\label{lem: CLZ-surj preserved for bir contractions}
Let $X \to T$ be a fibration. Let $X \dashrightarrow Y/T$ be a partial MMP$/T$. Suppose that the natural map $N^1(X/T) \to N^1(X_t)$ is surjective for very general $t \in T$. Then there exists an open subset $T_0 \subset T$ such that the natural map
\[
N^1(Y_{T_0}/T_0) \to N^1(Y_t), \quad [D] \mapsto [D|_{Y_t}]
\]
is surjective for very general $t \in T_0$.
\end{lemma}

\subsubsection{Deformation of nef cones}\label{subsubset: def of nef cone}

\begin{proposition}\label{prop: iso of nef}
Let $X/T$ be an MKD fiber space such that $X\to T$ is a fibration. Assume that $\Eff(X/T)$ is a non-degenerate cone. If the natural map $N^1(X/T) \to N^1(X_t)$ is an isomorphism for each $t\in T$, then there exists an open subset $T_0 \subset T$ such that for any open subset $U\subset T_0$, the induced map
\[
\Nef(X_U/U) \to \Nef(X_t), \quad [D] \mapsto [D|_{X_t}]
\] is an isomorphism for each $t\in U$.
\end{proposition}
\begin{proof}
First, we have the natural inclusion map $\Nef(X_U/U) \hookrightarrow \Nef(X_t)$ for any open subset $U\subset T$ by the definition of nef cones. To show the surjectivity of this map, assume that $D_t$ is a $\bQ$-Cartier divisor such that $[D_t] \in \Nef(X_t)$. By assumption, there exists a $\bQ$-Cartier divisor $D$ on $X$ such that $[D|_{X_t}]=[D_t]$. As ampleness is a Zariski open condition on the base, the set
\begin{equation}\label{eq: v.g. points are nef}
\{s \in T \mid  D|_{X_s} \text{~is nef}\}
\end{equation} consists of very general points of $T$.

Note that $D$ is a pseudo-effective divisor$/T$ as $D+\ep A$ is a big$/T$ divisor for any $\ep>0$, where $A$ is an ample$/T$ divisor on $X$. Therefore, we have $[D] \in \Eff(X/T)_+=\Eff(X/T)$ by the non-degeneracy of $\Eff(X/T)$ (see Corollary \ref{cor: of 3 equiv} (3)). Hence, we can assume that $D$ is an effective $\bQ$-Cartier divisor over $T$. 

As $X/T$ is an MKD fiber space, by Theorem \ref{thm: MMP for MKD}, we can run a $D$-MMP$/T$ with scaling of an ample divisor which terminates with a $D$-good minimal model$/T$. By \eqref{eq: v.g. points are nef}, this MMP is an isomorphism over a non-empty Zariski open subset $V \subset T_0$. In particular, this means that $D|_{X_V}$ is nef$/V$. Note that $V$ depends on $D$. 

In the sequel, we will show that there exists a universal open subset $V \subset T$ such that $D|_{X_V}$ is nef whenever $D|_{X_t}$ is a nef $\bQ$-Cartier divisor for some $t \in T$. This implies the surjectivity of 
\[
\Nef(X_U/U) \to \Nef(X_t)
\] for any open subset $U \subset V$.

Recall that in Definition \ref{def: generic nef cone}, the generic nef cone is defined as
\[
{\rm GNef}(X/T) \coloneqq \{[B] \in \Eff(X/T) \mid [B_\eta] \in \Nef(X_\eta)\}.
\]  By \eqref{eq: v.g. points are nef}, if $D|_{X_t}$ is nef for some $t\in T$, then $D_\eta$ is nef on $X_\eta$. Moreover, the above discussion shows that $[D] \in \Eff(X/T)$. Hence, we have
\[
[D] \in {\rm GNef}(X/T).
\] By Theorem \ref{thm: cone for GNef}, there exists a rational polyhedral cone $\Pi \subset {\rm GNef}(X/T)$ such that 
\[
\Gamma_{GA}(X/T) \cdot \Pi = {\rm GNef}(X/T),
\]  where $\Gamma_{GA}(X/T)$ is the image of 
\[
{\rm GAut}(X/T) = \{g\in \PsAut(X/T) \mid g_\eta \in \Aut(X_\eta)\}
\] under the natural group homomorphism $\rho_T: \PsAut(X/T) \to {\rm GL}(N^1(X/T))$. 

As $\Eff(X/T)$ is non-degenerate by assumption,  ${\rm GNef}(X/T)$ is also non-degenerate. Hence, $\Gamma_{GA}(X/T)$ is finitely presented by Theorem \ref{thm: finite presented}. Assume that $\gamma_1, \cdots, \gamma_l \in {\rm GAut}(X/T)$ are generic automorphisms such that $\rho_T(\gamma_1), \cdots, \rho_T(\gamma_l)$ generate $\Gamma_{GA}(X/T)$. As $\gamma_i \in \Aut(X_\eta)$, there exists an open subset $V_i \subset T$ and an automorphism $\tilde\gamma_i \in \Aut(X_{V_i}/V_i)$ such that $(\tilde\gamma_i)_\eta=\gamma_i$. For simplicity, we still denote $\tilde\gamma_i$ by $\gamma_i$. Replacing $T$ by $\cap_{i=1}^l V_i$, we can assume that $\gamma_i \in \Aut(X/T)$ for each $1 \leq i \leq l$. 

Next, let $D_i, 1 \leq i \leq v$ be effective $\bQ$-Cartier divisors such that 
\[
\Cone([D_i] \mid 1 \leq i \leq v) = \Pi.
\] As $(D_i)_\eta$ is nef, we see that $D_i|_{X_t}$ is nef for very general $t \in T$. By the previous discussion, there exists a non-empty Zariski open subset $U_i \subset T$ such that $D|_{X_{U_i}}$ is nef over $U_i$. Let $V=\cap_{i=1}^l U_i$. Then the image $\Pi'$ of $\Pi$ under the natural restriction map $N^1(X/T) \to N^1(X_V/V)$ lies in $\Nef^e(X_V/V)$. 

We claim that
\begin{equation}\label{eq: 3}
\langle \gamma'_i \mid 1 \leq i \leq l \rangle \cdot \Pi' = {\rm GNef}(X_V/V),
\end{equation} where $\langle \gamma'_i \mid 1 \leq i \leq l \rangle \subset \Aut(X_V/V)$ is the subgroup generated by $\gamma'_i= \gamma_i|_{X_U}, 1 \leq i \leq l$. Indeed, if $B$ is an effective $\bQ$-Cartier divisor such that $[B] \in {\rm GNef}(X_V/V)$, then let $\ti B$ be the Zariski closure of $B$ on $X$. Since $X$ is $\bQ$-factorial, $\ti B$ remains an effective $\bQ$-Cartier divisor. Moreover, we still have $[\ti B] \in {\rm GNef}(X/T)$. Hence, we have
\[
[\ti B] \in \langle \gamma_i \mid 1 \leq i \leq l \rangle \cdot \Pi.
\]
Restricting the above to $X_V/V$ yields~\eqref{eq: 3}.

(On the other hand, if $V \subset T$ is an open subset, a pseudo-automorphism of $X_V/V$ need not extend to a pseudo-automorphism of $X/T$. Consequently, $\langle \rho_V(\gamma'_i) \mid 1 \leq i \leq l \rangle$ may no longer generate $\Gamma_{GA}(X_V/V)$.)

Finally, as $\langle \gamma'_i \mid 1 \leq i \leq l \rangle \subset \Aut(X_V/V)$ and $\Pi' \subset \Nef^e(X_V/V)$, \eqref{eq: 3} implies 
\[
\Nef^e(X_V/V) = {\rm GNef}(X_V/V).
\] In particular, we have $[D] \in \Nef^e(X_V/V)$. Therefore, this $V$ is the desired universal open subset of $T$. This completes the proof.
\end{proof}

\begin{theorem}\label{thm: def of nef cone}
Let $f: X \to T$ be a fibration. Suppose that $S \subset T$ is a Zariski dense subset such that for any $s \in S$, the fiber $X_s$ satisfies
\[
H^1(X_s, \mO_{X_s}) = H^2(X_s, \mO_{X_s}) = 0.
\]
Assume further that the geometric generic fiber $X_{\bar\eta}$ is a klt MKD space. Then, after a generically finite base change of $T$, there exists a non-empty Zariski open subset $T_0 \subset T$ such that for any Zariski open subset $U \subset T_0$, the natural maps
\[
\begin{split}
& N^1(X_U/U) \to N^1(X_t), \quad [D] \mapsto [D|_{X_t}],\\
& \Nef(X/U) \to \Nef(X_t), \quad [D] \mapsto [D|_{X_t}],
\end{split}
\]
are isomorphisms for all $t \in U$. Moreover, $X_U/U$ is an MKD fiber space.
\end{theorem}
\begin{proof}
By Theorem \ref{thm: geometric MKD space}, there exists a generically finite base change $u: T' \to T$ such that, after shrinking $T'$, the morphism $X_{T'} \to T'$ becomes a klt MKD fiber space. Moreover, these properties are preserved for any generically finite base change factor through $T' \to T$. Hence, by replacing $T$ with $T'$ and $S$ with $u^{-1}(S)$, we can assume that $X/T$ is an MKD fiber space with klt singularities. By Proposition \ref{prop: generic property for MKD space}, $X_U/U$ is still an MKD fiber space for any non-empty open subset $U \subset T$.

By Theorem \ref{thm: CLZ-deform pic} (2), after replacing $T$ by a further generically finite base change, there exists an open subset $T_0 \subset T$ such that the natural restriction maps
\begin{equation}\label{eq: iso of N^1}
N^1(X_{T_0}/T_0) \to N^1(X_U/U) \to N^1(X_t), \quad [D] \mapsto [D|_{X_U}] \mapsto [D|_{X_t}]
\end{equation} are isomorphisms for any $t\in U$, where $U\subset T_0$ is an arbitrary open subset. Shrinking $T$ further, we may assume that the conclusion of Theorem \ref{thm: geometric MKD space} still holds. Then the desired isomorphism
\[
\Nef(X/T) \simeq \Nef(X_t)
\] follows from Proposition \ref{prop: iso of nef}. 
\end{proof}

A projective variety $X$ is called a klt weak Fano variety if $X$ has klt singularities and the anti-canonical divisor $-K_X$ is nef and big. A klt weak Fano variety is of Fano type and thus it is also an MKD space.

\begin{corollary}\label{cor: CY def of nef cone}
Let $f: X \to T$ be a fibration with $X$ a $\bQ$-Gorenstein variety. 
Suppose that $S \subset T$ is a Zariski dense subset such that, for any $s \in S$, 
the fiber $X_s$ is a Calabi-Yau variety (resp. a klt weak Fano variety). 
Then $X_{\bar\eta}$ is also a Calabi-Yau variety (resp. a klt weak Fano variety).

In the case where $X_{\bar\eta}$ is a Calabi-Yau variety, assume further that 
$X_{\bar\eta}$ satisfies the Morrison-Kawamata cone conjecture and the good minimal model conjecture, and that 
\[
H^1(X_s, \mO_{X_s}) = H^2(X_s, \mO_{X_s}) = 0
\]
for each fiber $X_s$ with $s \in S$. 

Then, in both cases, after a generically finite base change of $T$, 
there exists a non-empty Zariski open subset $T_0 \subset T$ such that, 
for any Zariski open subset $U \subset T_0$, the natural maps
\[
\Nef(X/U) \to \Nef(X_t), \quad [D] \mapsto [D|_{X_t}]
\]
are isomorphisms for all $t \in U$.
\end{corollary}
\begin{proof}
If $X_s$ is a Calabi-Yau variety for each $s\in S$, we see that $\pm K_{X_\eta}$ are nef. Hence, $\pm K_{X_{\bar\eta}}$ are also nef. This implies that  $K_{X_{\bar\eta}} \equiv 0$. As $X_s$ has klt singularities over a Zariski dense subset $S \subset T$, $X_{\bar\eta}$ also has klt singularities. Thus, $K_{X_{\bar\eta}} \sim_{\bQ} 0$ by \cite{Gon13}, which implies that $X_{\bar\eta}$ is a Calabi-Yau variety.

In the case where $X_s$, $s \in S$, are klt weak Fano varieties, by \cite[Theorem~1.2 (i)]{CLZ25}, we see that $X$ is of Fano type over $T$ after possibly shrinking $T$. Thus, $X_{\bar\eta}$ has klt singularities and $-K_{\bar\eta}$ is big. By the same reasoning as above, $-K_{X_{\bar\eta}}$ is nef, and thus $X_{\bar\eta}$ is a klt weak Fano variety.

Note that for a klt weak Fano variety $X_s$, we always have $H^1(X_s, \mO_{X_s}) = H^2(X_s, \mO_{X_s}) = 0$. Therefore, in both cases, the hypotheses of Theorem~\ref{thm: def of nef cone} are fulfilled, from which the desired statement follows.
\end{proof}

\begin{remark}
We list several remarks on Corollary~\ref{cor: CY def of nef cone}.
\begin{enumerate}
\item The assumption that $X$ is $\bQ$-Gorenstein in Corollary~\ref{cor: CY def of nef cone} can be removed by an argument similar to that in \cite[Lemma~1.12]{Kaw88}. This minor issue, however, requires a significant extension of the proof, and we therefore omit it here.
\item For the Calabi-Yau case, the assumption $H^1(X_s, \mO_{X_s}) = H^2(X_s, \mO_{X_s}) = 0$ is indispensable. Indeed, \cite{Ogu00} exhibits a family of K3 surfaces where the Picard numbers jump on a Zariski dense subset.
\item If the fibers $X_s$, $s\in S$, are only assumed to be of Fano type, then it is conjectured that $X_{\bar\eta}$ remains of Fano type (see \cite[Question~1.1]{CLZ25}). This is known in several special cases (see \cite[Theorem~1.2]{CLZ25}).
\end{enumerate}
\end{remark}

\subsubsection{Structure of birational models and birational maps}\label{subsec: structure of b(X/T) and bc(X/T)}

\begin{definition}\label{def: bc and b}
Let $X$ be a variety over $T$. We set
\[
\begin{split}
{\rm b}(X/T)& \coloneqq \{Y/T \mid X \dasharrow Y/T \text{~is a birational contraction up to isomorphism of~} Y/T\},\\
{\rm bc}(X/T) & \coloneqq \{h \mid h: X \dasharrow Y/T \text{~is a birational contraction up to isomorphism of~} Y/T\}.\\
\end{split}
\]
\end{definition}

In the above definition, for two birational contractions $h, g: X \dto Y/T$, we write $h=g$ if they agree up to an isomorphism of the target. In other words, there exists an isomorphism $\theta: Y \simeq Y/T$ such that $h=\theta \circ g$.

For completeness, we summarize the structures of ${\rm b}(X/T)$ and ${\rm bc}(X/T)$ for an MKD fiber space in the following proposition.

\begin{proposition}\label{prop: structure of b and bc}
Let $X/T$ be an MKD fiber space. 
\begin{enumerate}
\item ${\rm b}(X/T)$ is a finite set.
\item There exist finitely many birational contractions 
\[ 
f_i: X \dashrightarrow Y_i/T, \quad 1 \leq i \leq m, 
\] such that any birational contraction $h \in {\rm bc}(X/T)$ is isomorphic to $f_i \circ \mu$, where $\mu \in \PsAut(X/T)$.
\item Under the notation of (2), if $\Eff(X/T)$ is a non-degenerate cone, then there exist finitely many pseudo-automorphisms 
\[ 
\gamma_j \in \PsAut(X/T), \quad 1 \leq j \leq l, 
\] 
such that any birational contraction $h \in {\rm bc}(X/T)$ is isomorphic to $f_i\circ \gamma $, where $\gamma$ is a finite product of $\gamma_j, 1 \leq j \leq l$.
\item If $f: X \to \ti W/T$ is a  contraction morphism (not necessarily birational), then there exist finitely many contraction morphisms
\[
g_s: \ti W \to W_s/T, \quad 1 \leq s \leq p,
\] 
such that any contraction morphism $g:  \ti W \to W/T$ is isomorphic to $g_s \circ \ti\sigma$, where $\ti\sigma \in \Aut(\ti W/T)$.
\end{enumerate}
\end{proposition}
\begin{proof}
Statements (1) and (2) have already been proved in Theorem \ref{thm: 3 equivalences} (2); see \eqref{eq: finiteness of bir contractions} and the subsequent discussion for details.

\medskip

For (3), we use the same notation as in (2). By Theorem \ref{thm: finite presented}, $\Gamma_B(X/T)$ is finitely presented. Let $\gamma_j \in \PsAut(X/T)$, $1 \leq j \leq l$, be pseudo-automorphisms such that $\rho(\gamma_j)$, $1 \leq j \leq l$, generate $\Gamma_B(X/T)$. Then there exists $\gamma \in \PsAut(X/T)$, which is a finite product of the $\gamma_j$, such that $\rho(\gamma) = \rho(\mu)$, where $\mu \in \PsAut(X/T)$ satisfies $h = f_i \circ \mu$ as in (2).  As $f_i \circ \gamma$ and $f_i \circ \mu$ induce the same linear map $N^1(X/T) \to N^1(Y/T)$, there exists an isomorphism $\theta \in \Aut(Y_i/T)$ such that 
\[
\theta \circ f_i \circ \mu = f_i \circ \gamma.
\] 
This proves (3).

\medskip

For (4), by Proposition \ref{prop: face is of polyhedral type}, we see that
\[
(f^*\Nef^e(\ti W/T), \; {\rm Stab}_{f^*\Nef^e(\ti W/T)} \Gamma_{A}(X/T))
\]
is of polyhedral type and $\Nef^e(\ti W/T) = \Nef(\ti W/T)_+$. Hence, there exists a rational polyhedral cone $\Pi \subset \Nef^e(\ti W/T)$ such that ${\rm Stab}_{f^*\Nef^e(\ti W/T)} \Gamma_{A}(X/T) \cdot f^*\Pi =f^*\Nef^e(\ti W/T)$. Let 
\[
g_s: \ti W \to W_s/T, \quad 1 \leq s\leq p,
\] be finitely many contractions corresponding to the faces of $\Pi$. If $g: \ti W \to W/T$ is a contraction morphism, then let $A$ be an ample$/T$ divisor on $W$. Then there exists some $\sigma \in \Aut(X/T)$ with $\rho(\sigma)\in {\rm Stab}_{f^*\Nef^e(\ti W/T)} \Gamma_{A}(X/T)$ such that $\sigma \cdot [f^*(g^*A)] \in f^*\Pi$. Hence, there exists $\tilde \sigma \in \Aut(\tilde W/T)$ such that
\begin{equation}\label{eq: descend}
f \circ \sigma = \tilde \sigma \circ f.
\end{equation}
Let $F$ be the face of $\Pi$ such that $\sigma \cdot [f^*(g^*A)] \in f^* \Int(F)$. If $g_s: \ti W \to W_s/T$ is the contraction corresponding to $F$, then there exists an isomorphism $\theta: W \to W_s/T$ such that
\[
\theta \circ g\circ f = g_s \circ f \circ \sigma: X \to W_s/T. 
\] By \eqref{eq: descend}, we have $\theta \circ g\circ f = g_s \circ  \tilde \sigma \circ f$ which implies that
\[
\theta \circ g = g_s \circ  \tilde \sigma:  \ti W \to W_s/T.
\] In other words, $g_s \circ \ti\sigma$ is isomorphic to $g$. This shows (4).
\end{proof}

\subsubsection{Deformation invariance of $\Eff(X/T)$}\label{subsec: MMP and deformation of effective cones}

To establish the deformation invariance of cones for MKD spaces, we first prepare several lemmas.

\begin{lemma}\label{lem: b consists of MKD fiber spaces}
Let $X/T$ be an MKD fiber space. 
\begin{enumerate}
\item For any open subset $V\subset T$, each element $[h: X_V \dto Y/V ] \in {\rm bc}(X_V/V)$ with $Y$ a $\bQ$-factorial variety is still an MKD fiber space. 
\item For any $[h': X_V \dto Y'/V ] \in {\rm bc}(X_V/V)$, there exists some  $[h: X_V \dto Y/V ] \in {\rm bc}(X_V/V)$ with $Y/V$ an MKD fiber space and a birational contraction morphism $\sigma: Y \to Y'/V$ such that $h'=\sigma\circ h$.
\end{enumerate}
\end{lemma}
\begin{proof}
By Proposition \ref{prop: generic property for MKD space}, $X_V/V$ is an MKD fiber space. By Theorem \ref{thm: bir contraction is MKD}, $Y/V$ is an MKD fiber space. This shows (1).

\medskip

By Lemma \ref{lem: factor bir contraction}, $h'$ factors as a small $\bQ$-factorial modification $h: X_V \dashrightarrow Y/V$ followed by a birational contraction morphism $\sigma: Y \to Y'/V$. By (1), we see that $Y/V$ is an MKD fiber space. This proves (2).
\end{proof}

The following generalizes \cite[Proposition 5.2]{CLZ25} from Fano type varieties to MKD fiber spaces.

\begin{proposition}\label{prop: stable bdd}
Let $X/T$ be an MKD fiber space. Then, for any open subset $V\subset T$ and $h \in {\rm bc}(X_V/V)$, there exists an element $H \in {\rm bc}(X/T)$ such that $H|_{X_V} = h$.
\end{proposition}
\begin{proof}
Suppose that $h: X_V \dto Y/V$ is the birational contraction. Let $A$ be an ample divisor on $Y$ over $V$. Let $A'$ be the strict transform of $A$ on $X_V$ and let $B$ be the Zariski closure of $A'$ on $X$. Note that $B_V\coloneqq B|_{X_V}=A'$ and $B$ is a $\bQ$-Cartier divisor as $X$ is $\bQ$-factorial. Let $p: W \to X_V, q: W \to Y$ be projective birational morphisms such that $h=q\circ p^{-1}$. Then we have
\[
p^*B_V=q^*A+E,
\] where $E$ is an effective $p$-exceptional divisor. In particular, $h$ is the canonical model$/V$ of $B_V$ on $X_V$. As $X/T$ is an MKD fiber space, $B$ admits the canonical model $H: X \dto Z$ over $T$. Therefore, we have $H|_{X_V} = h$ by construction.
\end{proof}

\begin{lemma}\label{lem: iso for bir model}
Let $X/T$ be an MKD fiber space such that $X\to T$ is a fibration. Suppose that the natural map $N^1(X/T) \to N^1(X_t)$ is surjective for very general $t \in T$. Then there exists an open subset $T_0 \subset T$ such that for any open subset $U \subset T_0$ and $Y/U \in {\rm b}(X_{U}/U)$, the natural map
\[
N^1(Y/U) \to N^1(Y_t)
\]
is an isomorphism for any $t \in U$.
\end{lemma}
\begin{proof}
By Proposition \ref{prop: structure of b and bc} (1), ${\rm b}(X/T)$ is a finite set. By Proposition \ref{prop: stable bdd}, for any open subset $V \subset T$, each element of ${\rm b}(X_{V}/V)$ is the restriction of some element from ${\rm b}(X/T)$. Hence, it suffices to show that for a fixed $Y/T \in {\rm b}(X/T)$, there exists an open subset $T_0 \subset T$ such that the natural map 
\[
N^1(Y_U/U) \to N^1(Y_t)
\]
is an isomorphism for any $t \in U$, where $U \subset T_0$ is an open subset.

Suppose that $f: X \dto Y/T$ is a birational contraction. Let $A$ be an ample$/T$ divisor on $Y/T$ with $A_X$ the strict transform of $A$ on $X/T$. By Theorem \ref{thm: MMP for MKD}, there exists an $A_X$-MMP $h: X \dto X'/T$ such that the strict transform of $A_X$ on $X'$, denoted by $A'$, is semi-ample over $T$. Let $g: X' \to Z/T$ be the contraction morphism induced by $A'$. Since $f$ and $g \circ h$ are both canonical models$/T$ of $A_X$, there exists an isomorphism $\theta: Z \simeq Y/T$ such that $f = \theta \circ g \circ h$. In other words, $f$ can be decomposed into a sequence of partial MMP steps. By Lemma \ref{lem: CLZ-surj preserved for bir contractions}, $N^1(Y/T) \to N^1(Y_t)$ is surjective for very general $t \in T$. Then the desired claim follows from Lemma \ref{lem: total to fiber is inj}.
\end{proof}

The following can be shown along the lines of \cite[\S 5]{CLZ25}.

 \begin{lemma}\label{lem: uniform behavior in bc}
Let $f: X \to T$ be a fibration. Suppose that $S \subset T$ is a Zariski dense subset such that for any $s \in S$, the fiber $X_s$ satisfies
\[
H^1(X_s, \mO_{X_s}) = H^2(X_s, \mO_{X_s}) = 0.
\]
Assume further that the geometric generic fiber $X_{\bar\eta}$ is a klt MKD space. Then, up to a generically finite base change of $T$, there exists an open subset $T_0\subset T$ such that the following properties hold.
\begin{enumerate}
\item For any open subset $V\subset T_0$, if $Y/V\in {\rm b}(X_{V}/V)$, then the natural maps
\[
N^1(Y/V) \to N^1(Y_t), \quad \Nef(Y/V) \to \Nef(Y_t), \quad \overline{\NE}(Y_t) \to \overline{\NE}(Y/V)
\] are isomorphisms for any $t\in V$.
\item For any $Y/T_0\in {\rm b}(X_{T_0}/T_0)$, $Y$ is flat over $T_0$, and $Y_t$ is an irreducible normal variety for each $t\in T_0$. 
\item If $V\subset T_0$ is an open subset and $h: Y \to Z/V$ is a contraction morphism (not necessarily birational) for some $Y/V \in {\rm b}(X_{V}/V)$, then $h_t: Y_t \to Z_t$ is still a contraction for each $t\in V$. Moreover, if $h$ is a divisorial contraction (resp. a small contraction, an extremal contraction, a Mori fiber space) if and only if $h_t$ for $t\in V$ is a divisorial contraction (resp. a small contraction, an extremal contraction, a Mori fiber space).
\end{enumerate}
 \end{lemma}
 \begin{proof}
By Theorem \ref{thm: def of nef cone}, there exists a generically finite base change $T' \to T$ and an open subset $T_0'\subset T'$ such that for any Zariski open subset $U \subset T_0'$, the natural maps
\[
\begin{split}
& N^1(X_U/U) \to N^1(X_t), \quad [D] \mapsto [D|_{X_t}],\\
& \Nef(X/U) \to \Nef(X_t), \quad [D] \mapsto [D|_{X_t}],
\end{split}
\]
are isomorphisms for all $t \in U$. Note that $X_U/U$ is still an MKD fiber space.

By Lemma \ref{lem: iso for bir model}, after possibly shrinking $T_0'$, for any $Y/V \in {\rm b}(X_{V}/V)$, the natural map
\begin{equation}\label{eq: iso 2}
N^1(Y/V) \to N^1(Y_t)
\end{equation}
is an isomorphism for any $t \in V$, where $V \subset T_0'$ is an open subset. By Proposition \ref{prop: structure of b and bc} (1) and Proposition \ref{prop: stable bdd}, applying Proposition \ref{prop: Generic property} (2) to each element of ${\rm b}(X_{T'_0}/T'_0)$, we may, after further shrinking $T'_0$, assume that $\Eff(Y/V)$ is non-degenerate for every $Y/V\in {\rm b}(X_{V}/V)$.

Let $V \subset T_0'$ be a non-empty open subset. Choose $Y'/V \in {\rm b}(X_V/V)$, and suppose that $f: X_V \dto Y'/V$ is a birational contraction map. Let $A'$ be an ample$/V$ divisor on $Y'/V$, and let $A$ denote the strict transform of $A'$ on $X_V$. Let $B$ be the Zariski closure of $A$ on $X_{T_0'}$. Hence, we have $B|_{X_V}=A$. By Theorem \ref{thm: MMP for MKD}, there exists a $B$-MMP $h: X_{T_0'} \dto \ti Y/T_0'$ such that the strict transform of $B$ on $\ti Y$, denoted by $\ti B$, is semi-ample over $T_0'$. Let 
\[
\ti Y \to \ti Y'/T_0'
\] be the contraction morphism induced by $\ti B$. Since the natural map restricted to $V$,
\[
X_{V} \dto \ti Y|_V \to \ti Y'|_V,
\]
is the canonical model of $A$ over $V$, it follows that this map is isomorphic to $f: X_V \dto Y'/V$. Moreover, $\ti Y/T_0'$ is an MKD fiber space by Theorem \ref{thm: MMP for MKD} and thus $\ti Y|_V$ is still an MKD fiber space by Proposition \ref{prop: generic property for MKD space}. By Proposition \ref{prop: iso of nef} and \eqref{eq: iso 2}, after shrinking $T_0'$, the natural map 
\[
\Nef(\ti Y_V/V) \to \Nef(\ti Y_t)
\] is an isomorphism for each $t\in V$. This yields the following commutative diagram
\[
\begin{tikzcd}
 \Nef(\ti Y_V/V) \arrow[r, "\simeq"] & \Nef(\ti Y_t)  \\
 \Nef(Y'/V) \arrow[u, hookrightarrow] \arrow[r] & \Nef(Y'_t) \arrow[u, hookrightarrow]
\end{tikzcd}
\] for any $t\in V$. As $N^1(\ti Y_V/V) \simeq N^1(\ti Y_t)$ and $N^1(Y'/V) \simeq N^1(Y'_t)$ for any $t\in V$ by \eqref{eq: iso 2}, the natural map
\[
\Nef(Y'/V) \to \Nef(Y'_t)
\] is an isomorphism for any $t\in V$. As $\overline{\NE}(Y'/V)$ and $\overline{\NE}(Y'_t)$ are dual to $\Nef(Y'/V)$ and $\Nef(Y'_t)$, respectively, we have
\[
\overline{\NE}(Y'/V)\simeq \overline{\NE}(Y'_t)  
\]
for any $t \in V$. This completes the proof of (1).

\medskip

For (2), by Lemma \ref{prop: stable bdd} and Proposition \ref{prop: structure of b and bc} (1), after shrinking $T$, we may assume that for any $Y/T \in {\rm b}(X/T)$, the morphism $Y \to T$ is flat by generic flatness, and each fiber $Y_t$, $t \in T$, is an irreducible normal variety, since the set 
\[
\{p \in T \mid Y_p \text{~is normal}\}
\]
is constructible. Note that in the above set, $p$ is not restricted to closed points, and the generic point of $T$ belongs to it, as $Y$ is normal.

\medskip

To show (3), first note that by (1), we have 
\begin{equation}\label{eq: for 3}
\Nef(Y/V) \simeq \Nef(Y_t)
\end{equation} 
for any $Y/V\in {\rm b}(X_{V}/{V})$, where $V\subset {T}$ is an open subset (here we replace $T_0$ by $T$).  

If $h: Y' \to Z'/V$ is a contraction for some $Y'/V \in {\rm b}(X_V/V)$, then let $A'\geq 0$ be an ample$/V$ divisor on $Z'$. By Lemma \ref{prop: stable bdd}, there exists $Y/T \in {\rm b}(X/T)$ such that $\theta: Y_V \simeq Y'/V$. Let $B$ be the Zariski closure of $\theta^*(h^*A')$ on $Y$. By \eqref{eq: for 3}, $B$ is nef over $T$. Although $Y/T$ may not be an MKD fiber space because $Y$ may not be $\bQ$-factorial, the divisor $B$ is still semi-ample over $T$. Indeed, by Lemma \ref{lem: b consists of MKD fiber spaces}~(2), there exists an MKD fiber space $W/T$ and a birational morphism $\sigma: W \to Y/T$. Hence, $\sigma^*B$ is semi-ample$/T$ by the definition of MKD fiber spaces. This implies that $B$ is also semi-ample$/T$. Let $g: Y \to Z/T$ be the contraction induced by $B$. Since $B|_V=\theta^*(h^*A')$, we have 
\begin{equation}\label{eq: morphi_V}
g|_V = h.
\end{equation}
By Lemma \ref{lem: b consists of MKD fiber spaces}~(2) and Proposition \ref{prop: structure of b and bc}~(4), for each $W/T\in {\rm b}(X/T)$, there are finitely many contraction morphisms
\[
g_s: W \to W_s/T, \quad 1 \leq s \leq p
\] 
such that any contraction morphism $g: W \to W'/T$ is isomorphic to $g_s \circ \sigma$, where $\sigma \in \Aut(W/T)$. Collecting these contraction morphisms for different $W/T$ together, and since ${\rm b}(X/T)$ is a finite set by Proposition \ref{prop: structure of b and bc} (1), there are only finitely many contraction morphisms
\[
g_s, \quad 1 \leq s \leq n,
\]
in total.

For any $Y/V \in {\rm b}(X_V/V)$, let $h: Y \to Z/V$ be a contraction morphism. By \eqref{eq: morphi_V}, we see that $h=(g_s \circ \sigma)|_V$ for some $1 \leq s \leq n$ and $\sigma \in \Aut(W/T)$, where $g_s: W \to W_s/T$. Since 
\[
{\overline\NE}(W/T) \simeq {\overline\NE}(W_V/V) \simeq {\overline\NE}(W_{t}),\quad t\in T
\] by (1), we see that $g_s$, $(g_s)_V$, and $(g_s)_t$ are extremal contractions if one of them is an extremal contraction. After further shrinking $T$, we may assume that each $g_s$, $1 \leq s \leq n$, is a divisorial contraction (resp. a small contraction, a Mori fiber space) if and only if $(g_s)_t$, $t \in T$, is a divisorial contraction (resp. a small contraction, a Mori fiber space). Since $\sigma \in \Aut(W/T)$, $g_s \circ \sigma$ is a divisorial contraction (resp. a small contraction, a Mori fiber space) if and only if $(g_s \circ \sigma)_t$, $t \in T$, is a divisorial contraction (resp. a small contraction, a Mori fiber space). This completes the proof of (3).
 \end{proof}
In what follows, for a Cartier divisor $\mathcal{D}$ on a fiber space $X/T$, on the fiber $X_t$, we write $\mathcal{D}|_{X_t}$ or $\mathcal{D}_t$ for any Cartier divisor on $X_t$ such that 
\[
\mathcal{O}_{X_t}(\mathcal{D}|_{X_t}) = \mathcal{O}_X(\mathcal{D})|_{X_t}
\] 
as line bundles. Since $\mathcal{D}$ may contain $X_t$ in its support, the restriction $\mathcal{D}|_{X_t}$ is only well-defined at the level of line bundles. We use the same notation for $\mathbb{R}$-linear combinations of Cartier divisors (i.e., $\mathbb{R}$-Cartier divisors).

\begin{theorem}\label{thm: MMP}
Let $f: X \to T$ be a fibration. Suppose that $S \subset T$ is a Zariski dense subset such that for any $s \in S$, the fiber $X_s$ satisfies
\[
H^1(X_s, \mO_{X_s}) = H^2(X_s, \mO_{X_s}) = 0.
\]
Assume further that the geometric generic fiber $X_{\bar\eta}$ is a klt MKD space. Then, after a generically finite base change of $T$, there exists a non-empty Zariski open subset $T_0 \subset T$ such that any Zariski open subset $U \subset T_0$ satisfies the following properties.
\begin{enumerate}
\item Suppose that $X'/T \in {\rm b}(X/T)$. If 
$g: X' \dashrightarrow Y/T$ is a birational contraction between $\bQ$-factorial varieties, 
then for every $t \in T_0$, the induced map 
$g_t: X'_t \dashrightarrow Y_t$ is again a birational contraction. Moreover, for any 
$\bR$-Cartier divisor $\mD$ on $X'$, we have
\[
(g_t)_*(\mD|_{X'_t}) \sim_{\bR} (g_*\mD)|_{Y_t}, \quad t \in T_0.
\]
\item For each $X'/T \in {\rm b}(X/T)$ and any effective $\bR$-divisor $\mD$ on $X'_U$, a sequence of $\mD$-MMP$/U$ induces a sequence of $\mD|_{X'_t}$-MMP of the same type for each $t\in U$. 
\item Conversely, for each $X'/T \in {\rm b}(X/T)$ with $X'$ $\bQ$-factorial, and any effective $\bR$-divisor $D$ on $X'_t$ with $t \in U$, any sequence of $D$-MMP on $X'_t$ is induced by a sequence of $\mD$-MMP$/U$ on $X'_U/U$ of the same type, where $\mD$ is an effective divisor satisfying $[\mD|_{X'_t}] = [D]$.
\item For each $X'/T \in {\rm b}(X/T)$, we have $\Eff(X'_U/U) \simeq \Eff(X'_t)$ for any $t\in U$.
\end{enumerate}
\end{theorem} 
\begin{proof}
By Theorem \ref{thm: geometric MKD space}, there exists a generically finite base change $u: T' \to T$ such that, after shrinking $T'$, the morphism $X_{T'} \to T'$ becomes a klt MKD fiber space. Moreover, these properties are preserved under any generically finite base change that factors through $T' \to T$. Hence, by replacing $T$ with $T'$ and $S$ with $u^{-1}(S)$, we can assume that $X/T$ is a klt MKD fiber space. By Proposition \ref{prop: generic property for MKD space}, $X_U/U$ is still an MKD fiber space for any non-empty open subset $U \subset T$. After replacing $T$ by a generically finite base change and possibly shrinking it, we may assume that the properties listed in Lemma \ref{lem: uniform behavior in bc} hold. By Proposition \ref{prop: Generic property} (2), after further shrinking $T$, we may assume that $\Eff(X_U/U)$ is non-degenerate for any open subset $U\subset T$.

\medskip

For (1), since ${\rm b}(X/T)$ is a finite set by Proposition~\ref{prop: structure of b and bc}~(1), and $X'/T$ is an MKD fiber space by Theorem~\ref{thm: bir contraction is MKD}, 
we may assume that $X'/T$ is simply $X/T$. By Proposition \ref{prop: structure of b and bc}~(3), there exist finitely many birational contractions 
\[ 
f_i: X \dashrightarrow Y_i/T, \quad 1 \leq i \leq m, 
\] and finitely many pseudo-automorphisms 
\[
\gamma_j \in \PsAut(X/T), \quad 1 \leq j \leq l, 
\]
such that any birational contraction $g \in {\rm bc}(X/T)$ is isomorphic to $f_i\circ \gamma $, where $\gamma$ is a finite product of $\gamma_j, 1 \leq j \leq l$. In other words, there exists an isomorphism $\theta: Y_i \simeq Y/T$, such that $g$ can be decomposed into
\begin{equation}\label{eq: dec bir}
g: X \stackrel{\gamma_{i_1}}{\dto} X \stackrel{\gamma_{i_2}}{\dto} X \dto \cdots \dto X \stackrel{\gamma_{i_j}}{\dto}  X \stackrel{{f_i}}{\dto} Y_i  \stackrel{\theta}{\simeq} Y/T.
\end{equation} 
For each $\gamma_j$, fix projective birational morphisms $p\colon W \to X$ and $q\colon W \to X$ such that $\gamma_j = q \circ p^{-1}$. Then there exists a non-empty open subset $T_0 \subset T$ such that for every $t \in T_0$, the induced map
\[
\gamma_{j,t} \colon X_t \dashrightarrow X_t
\]
is still a birational contraction, and that $p_t$ and $q_t$ are projective birational morphisms. Moreover, we may further assume that if a divisor $E$ on $W$ is $p$-exceptional (resp. $q$-exceptional), then for every $t \in T_0$, the divisor $E|_{W_t}$ is $p_t$-exceptional (resp. $q_t$-exceptional). A similar construction also applies to each $f_i$. Since the sets $\{\gamma_j\mid 1 \leq j \leq l\}$ and $\{f_i \mid 1 \leq i \leq m\}$ are finite, there exists a non-empty open subset $T_0 \subset T$ on which the above property holds simultaneously for all $f_i$ and $\gamma_j$.

By construction, for any $\Rr$-Cartier divisor $\Dd$ on $X$, we have 
\[
p^*\mD = q^*\mD' +E,
\] where $\mD' = \gamma_{j*}\mD$ and $E$ is $q$-exceptional. Hence, we have
\[
p_t^*\mD_t \sim_\bR q_t^*\mD'_{t} + E|_{W_t}.
\] As $\gamma_{j, t}$ is a birational contraction and $E|_{W_t}$ is $q_t$-exceptional, we see 
\[
\gamma_{j,t*}(\mD_t) \sim_\Rr \mD'_{t}.
\] Since $\theta_t\colon Y_t \simeq Y_t$ is an isomorphism, repeating the above argument along the sequence \eqref{eq: dec bir} yields the desired claim.

\medskip

For (2), we show that an MMP of the fiber space is an MMP of each fiber of the same type. For any $[\mD] \in \Eff(X'/T)$, let 
\[
X'=X'_0 \dasharrow X'_1 \dasharrow \cdots \dto X'_{n-1}\dasharrow X'_n \to X'_{n+1}
\] be a $\mD$-MMP over $T$,  where $X'_i \dasharrow X'_{i+1}, i=0, \ldots, n-1$ are birational contractions, and $X'_n \to X'_{n+1}$ is a contraction induced by the semi-ample$/T$ divisor $\mD_{X'_{n}}$. Thus, the natural birational contraction $X \dto X' \dasharrow X'_i/T$ belongs to $\bc(X/T)$. By Lemma \ref{lem: uniform behavior in bc} (3), if $g: X_i \dasharrow X_{i+1}$ is a divisorial contraction (resp. a small contraction, an extremal contraction, a Mori fiber space) if and only if $g_t$ is a divisorial contraction  (resp. a small contraction, an extremal contraction, a Mori fiber space) for each $t\in T$. Hence,
\[
X'_t=X'_{0,t} \dasharrow X'_{1,t} \dasharrow \cdots \dasharrow X'_{n,t} \to X'_{n+1, t}
\] is a $\mD_t$-MMP on $X'_t$ of the same type.

\medskip

Next, we show (4) first. For simplicity, we set $U=T$. The following argument is inspired by the proof of \cite[Theorem 4.2 (3)]{FHS24}.

Let $\iota: N^1(X'/T) \simeq N^1(X'_t)$ be the natural linear map as in Lemma \ref{lem: uniform behavior in bc} (1). By the upper-semicontinuity of cohomologies \cite[Chapter III, Theorem 12.8]{Har77}, we have  
\begin{equation}\label{eq: eff inclusion}
\iota(\Eff(X'/T)) \subset \Eff(X'_t).
\end{equation} This implies $\iota(\bEff(X'/T)) \subset \bEff(X'_t)$.

In the sequel, we show the converse inclusion of \eqref{eq: eff inclusion}. 

We first show the converse inclusion of \eqref{eq: eff inclusion} when $X'/T$ is an MKD fiber space. In this case, it suffices to show that $\iota$ induces the isomorphism $\bEff(X'/T) \simeq \bEff(X'_t)$. Indeed, from this isomorphism, we have 
\[
\bEff(X'/T)_+ \simeq \bEff(X'_t)_+ \supset \Eff(X'_t).
\] As $X'/T$ is an MKD fiber space and $\Eff(X'/T)$ is non-degenerate, we have $\Eff(X'/T)=\Eff(X'/T)_+$ by Corollary \ref{cor: of 3 equiv} (3). This implies that $\iota(\Eff(X'/T)) \supset \Eff(X'_t)$. 

As we have $\iota: \Nef(X'/T) \simeq \Nef(X'_t)$ by Lemma \ref{lem: uniform behavior in bc} (1), the cone
\[
\iota(\Eff(X'/T)) \cap \Eff(X'_t)
\] contains the full-dimensional cone $\Amp(X'_t)$. If $\iota(\bEff(X'/T)) \subsetneqq \bEff(X'_t)$, then there exists some $[\mD]$ lying on the boundary of $\bEff(X'/T)$ such that $[\mD_t] \in \Int(\Eff(X'_t))$. Indeed, take 
\[
y \in \partial \bEff(X_t)\setminus \iota(\bEff(X/T)), \quad x\in \Amp(X_t),
\] then the interval $[x, y]$ intersects with $\partial \bar(\iota(\bEff(X/T)))$ which is the desired $\iota([\mD])$. Let 
\[
[\mD_i] \in \Int(\Eff(X'/T)),\quad i\in \Nn
\] be a sequence of divisors such that $[\mD_i] \to [\mD]$. By (2), for each $i$, there exists a $\mD_i$-MMP/$T$, $X' \dto Y/T$, inducing a $\mD_{i,t}$-MMP of the same type. Let $\mD_{i,Y}$ be the strict transform of $\mD_{i}$ on $Y$. Since $\mD_{i,Y}$ is nef$/T$, we have
\[
\vol((\mD_{i,Y})_\eta) = \vol(\mD_{i,Y}|_{Y_t}) \quad t\in T.
\] As volumes are preserved under the MMP, we have
\[
\vol((\mD_{i})_\eta)=\vol((\mD_{i,Y})_\eta), \quad \vol(\mD_{i}|_{X'_t})=\vol(\mD_{i,Y}|_{Y_t}).
\] Since the volume function is a continuous function on the real N\'eron-Severi space (see \cite[Corollary 2.2.45]{Laz04a}), we have
\[
0<\vol(\mD|_{X'_t}) = \lim_{i \to \infty} \vol(\mD_{i}|_{X'_t}) =  \lim_{i \to \infty}\vol((\mD_{i})_\eta) = \vol((\mD)_\eta). 
\] This implies that $(\mD)_\eta$ is a big divisor. Thus $\mD$ is a big divisor over $T$. This contradicts the choice of $[\mD]$. Hence, we have $\iota(\bEff(X'/T))= \bEff(X'_t)$.

Next, we show (4) for an arbitrary $X'/T \in {\rm b}(X/T)$. By Lemma \ref{lem: b consists of MKD fiber spaces}~(2), $X \dto X'/T$ can be factored into a birational contraction $X \dto \ti X/T$ with $\ti X/T$ an MKD fiber space followed by a birational morphism $g: \ti X \to X'/T$. This yields the following commutative diagram
\[
\begin{tikzcd}
 \Eff(\ti X/T) \arrow[r, "\simeq"] & \Eff(\ti X_t)  \\
 \Eff(X'/T) \arrow[u, hookrightarrow] \arrow[r] & \Eff(X'_t) \arrow[u, hookrightarrow]
\end{tikzcd}
\] for any $t\in T$. By chasing the diagram and using the isomorphism $N^1(X'/T)\simeq N^1(X'_t)$ from Lemma~\ref{lem: uniform behavior in bc}~(1), we see that the natural injective map
\[
\Eff(X'/T)\to \Eff(X'_t)
\]
is surjective. This completes the proof of (4).

\medskip

Finally, we show (3), that is, an MMP of the fiber is induced from an MMP$/T$ of the total fiber space. Suppose that 
\begin{equation}\label{eq: MMP fiber}
X'_t=Y_{0} \dasharrow Y_{1} \dasharrow \cdots \dasharrow Y_{n} \to W
\end{equation} 
is a $D$-MMP on $X'_t$ for some $[D]\in\Eff(X'_t)$. By Lemma \ref{lem: uniform behavior in bc}~(4), there exists $[\mD]\in \Eff(X'/T)$ such that $[\mD_t]=[D]$. Hence, after possibly replacing $\mD$ by a numerically equivalent divisor, we may assume that $\mD$ is effective.

 If $\sigma_0: Y_0 \to Y_1$ is a divisorial contraction (resp. a small contraction, an extremal contraction, a Mori fiber space), then by Lemma \ref{lem: uniform behavior in bc} (1) and Theorem \ref{thm: MMP for MKD}, there exists a contraction morphism $g_0: X' \to X'_1$ that contracts the same extremal ray as $\sigma_0$, and thus $g_{0,t}=\sigma_0$. By Lemma \ref{lem: uniform behavior in bc} (3), $g_0$ is still a divisorial contraction (resp. a small contraction, an extremal contraction, a Mori fiber space). When $\sigma_0$ is a flipping contraction, let $\sigma_1: Y_2 \to Y_1$ be its flip. Let $g_1: X'_2 \to X'_1$ be the flip of the flipping contraction $g_0: X' \to X'_1$. By Lemma \ref{lem: uniform behavior in bc} (2) (3), $g_{1,t}: X'_{2,t} \to X'_{1,t}$ is an extremal contraction between normal varieties. Then $\mD_{X'_2,t}$ is ample over $X'_{1,t}$ as $\mD_{X'_2}$ is ample over $X'_1$, where $\mD_{X'_2}$ is the strict transform of $\mD$ on $X'_2$. As
 \[
 (\sigma^{-1}_1 \circ \sigma_0)_*\mD_t \sim_\Rr \mD_{X'_2, t}
 \] by (1), $g_{1,t}$ is exactly $\sigma_1$. Repeating this process, we obtain a $\mD$-MMP on $X'/T$ whose restriction to the fiber $X'_t$ is exactly \eqref{eq: MMP fiber}. Moreover, this $\mD$-MMP terminates when \eqref{eq: MMP fiber} terminates by Lemma \ref{lem: uniform behavior in bc} (1) (3).
\end{proof}

\begin{remark}\label{rmk: model may not satisfy coh conditions}
Unlike the case where $X \to T$ is a Calabi-Yau type fibration, under the hypotheses of Theorem \ref{thm: MMP}, even if $X'/T$ is obtained from $X/T$ by an MMP$/T$, we still do not know whether 
\[
H^1(X'_s, \mathcal{O}_{X'_s}) = H^2(X'_s, \mathcal{O}_{X'_s}) = 0
\]
holds over a Zariski dense subset of $T$. Indeed, $X'/T$ and $X'_s$ may no longer have klt singularities.
\end{remark}

\subsubsection{Specialization of an MKD fiber space}\label{subsubsec: specialization}

In this section, we show that the general fibers of a geometrically generic MKD space are almost MKD spaces. This fact will be used later in the study of the Mori chamber decomposition.

We first recall the following result on fiberwise small $\bQ$-factorial modifications.

\begin{theorem}[{\cite[Theorem 1.4]{CLZ25}}]\label{thm: fiberwise q-fact}
Let $f: X \to T$ be a fibration with $(X, \Delta)$ a klt pair for some effective $\mathbb{R}$-divisor $\Delta$ on $X$. Suppose that $S\subset T$ is a Zariski dense subset such that for any $s\in S$, the fiber $X_s$ satisfies
\[
H^1(X_s, \mO_{X_s})=H^2(X_s, \mO_{X_s})=0.
\] 
Then, up to a generically finite base change of $T$, there exist a birational morphism $Y \to X$ and a non-empty open subset $T_0 \subset T$ such that $Y \to X$ and each fiber $Y_t \to X_t$ for $t \in T_0$ are small $\bQ$-factorial modifications.

Moreover, for any open subset $U\subset T_0$, the natural restriction maps
\[
\begin{split}
    &N^1(Y_{T_0}/T_0) \to N^1(Y_U/U) \to N^1(Y_t)\\
    &N^1(X_{T_0}/T_0) \to N^1(X_U/U) \to N^1(X_t)
\end{split}
\] are isomorphisms for any $t\in U$.
\end{theorem}

\begin{theorem}\label{thm: fiber MKD}
Let $f: X \to T$ be a fibration. Suppose that $S \subset T$ is a Zariski dense subset such that for any $s \in S$, the fiber $X_s$ satisfies
\[
H^1(X_s, \mO_{X_s}) = H^2(X_s, \mO_{X_s}) = 0.
\]
Assume further that the geometric generic fiber $X_{\bar\eta}$ is a klt MKD space. Then, there exists a non-empty Zariski open subset $T_0 \subset T$ such that each fiber $X_t$, $t \in T_0$ satisfies the following properties.    
\begin{enumerate}
    \item $X_t$ is a $\bQ$-factorial variety.  
   \item For any effective $\bR$-Cartier divisor $B$ on $X_t$, there exists an effective $\bR$-Cartier divisor $B'$ such that $B' \equiv B$ and $B'$ admits a good minimal model.
        \item There exists a rational polyhedral cone $\Pi_t \subset \Eff(X_t)$ such that $\PsAut(X_t) \cdot \Pi_t = \Eff(X_t)$.
        \item Any set of divisors $\bS \subset {\rm CDiv}(X_t)$ satisfies the local factoriality of canonical models.
    \end{enumerate}
\end{theorem} 
\begin{proof}
If there exists a generically finite morphism $T' \to T$ such that the fibers of $X_{T'} \to T'$ over a non-empty open subset of $T'$ satisfy the claim, then the desired statement follows. Hence, in the arguments below, we are allowed to take generically finite base changes.

\medskip

Replacing $T$ by a generically finite base change, we may assume that Theorem~\ref{thm: geometric MKD space}, Theorem \ref{thm: def of nef cone}, Theorem~\ref{thm: MMP}, and Theorem~\ref{thm: fiberwise q-fact} hold. In particular, $X/T$ is $\bQ$-factorial and admits a fiberwise small $\bQ$-factorial modification $Y \to X/T$ over a non-empty open subset of $T_0\subset T$. Moreover, $Y \to X/T$ is still a small $\bQ$-factorial modification (see Theorem~\ref{thm: fiberwise q-fact}). Hence $Y \simeq X/T$, which implies that the fibers of $X \to T$ are $\bQ$-factorial over $T_0$. This shows (1). 

\medskip

By Theorem~\ref{thm: MMP} (4), possibly shrinking $T_0$, we see that for any effective divisor $B$ on $X_t$, there exists an effective divisor $\mD$ on $X/T$ such that $\mD|_{X_t} \equiv B$. Let $B' = \mD|_{X_t}$. Then Theorem~\ref{thm: MMP} (2) implies that $B'$ admits a good minimal model. This shows (2). 

\medskip

As $X/T$ is an MKD fiber space, there exists a rational polyhedral cone $\Pi \subset \Eff(X/T)$ such that $\PsAut(X/T) \cdot \Pi =\Eff(X/T)$ by Theorem \ref{thm: 3 equivalences} (3). By Theorem~\ref{thm: MMP} (1), for any $\sigma \in \PsAut(X/T)$ and $t\in T_0$, we have $\sigma_t \coloneqq \sigma|_{X_t} \in \PsAut(X_t)$. Let 
\[
\Pi_t \coloneqq \Cone([D_t ]\mid [D] \in \Pi) \subset \Eff(X_t),
\] which is a rational polyhedral cone. Then we have
\[
\{\sigma_t \mid \sigma \in \PsAut(X/T)\} \cdot \Pi_t = \Eff(X_t)
\] by Theorem~\ref{thm: MMP} (1) and (4). This implies that
\[
\PsAut(X_t) \cdot \Pi_t = \Eff(X_t).
\] This shows (3).

\medskip

By Theorem \ref{thm: 3 equivalences} (2), if $X \dashrightarrow Z/T$ is the conical model of $\mathcal{D}$ over $T$, then for each $t \in T_0$, the induced map $X_t \dashrightarrow Z_t$ is the canonical model of $\mathcal{D}_t$. Since $\Eff(X/T) \simeq \Eff(X_t)$ by Theorem \ref{thm: 3 equivalences} (4), the local factoriality of canonical models for $\bS$ follows from that for $\Eff(X/T)$. This establishes (4).
\end{proof}

\begin{remark}\label{rk: almost MKD}
Theorem~\ref{thm: fiber MKD} shows that, for a fiber space $X/T$, if the geometric generic fiber is an MKD space, then the fibers over a Zariski open subset of $T$ are almost MKD spaces under certain cohomological conditions. Here, ``almost'' means that, instead of knowing that every effective divisor admits a good minimal model (see Definition~\ref{def: MKD spaces}~(2)), we only know that, for each effective divisor, there exists a numerically equivalent divisor that admits a good minimal model (see Theorem~\ref{thm: fiber MKD}~(2)). In particular, in (4), we state local factoriality of canonical models for $\bS\subset \operatorname{CDiv}(X_t)$ rather than for $\Eff(X_t)$.
\end{remark}

\subsubsection{Deformation invariance of movable cones and Mori chamber decompositions}\label{subsubsec: def mov and Mori chamber}

Suppose that $X\to T$ is a fibration. If $X/T$ is an MKD fiber space and $D$ is a movable divisor, then there exists a $D$-MMP$/T$, $\phi: X \dto Y/T$, which terminates at $Y$ by Theorem \ref{thm: MMP for MKD}. Note that $X \dto Y$ is isomorphic in codimension $1$, and if $D_Y$ is the strict transform of $D$ on $Y$, then $D_Y$ is a semi-ample divisor over $T$. Thus we have $[D_Y]\in \Nef^e(Y/T)$.

Conversely, suppose that $\phi\colon X\dto Y/T$ is an MMP$/T$ which is isomorphic in codimension $1$. By Lemma~\ref{lem: b consists of MKD fiber spaces}~(2), $\phi$ factors as a birational contraction $X\dto Y'/T$, where $Y'/T$ is an MKD fiber space, followed by a morphism $\theta\colon Y'\to Y/T$. Hence every effective nef$/T$ divisor on $Y'$ is semi-ample$/T$. This implies that every effective nef$/T$ divisor on $Y$ is also semi-ample$/T$. Thus,
\[
\Mov(X/T)\supset \phi_*^{-1}\Nef^e(Y/T).
\]

The above discussion implies the following Mori chamber decomposition 
\begin{equation}\label{eq: Mori chamber decomp}
\Mov(X/T)=\bigcup_{\substack{\phi \colon X \dto Y \text{~isom.~in codim.~} 1,\\ \phi \text{~is an MMP$/T$}}} \phi^{-1}_{*}\Nef^e(Y/T).
\end{equation}
In the above union, each subcone $\phi^{-1}_{*}\Nef^e(Y/T)$ is called a Mori chamber (cf. \cite[\S 5.4]{CLZ25}). Note that different Mori chambers are disjoint in their interiors. 

On the other hand, assume that $X/T$ is an MKD fiber space satisfying the conclusions of Theorem \ref{thm: MMP} (that is, $X/T$ is the $X'/T$ appearing in Theorem \ref{thm: MMP} (1)--(4)). We claim that, for each $t\in T_0$, $X_t$ still admits the Mori chamber decomposition as \eqref{eq: Mori chamber decomp}:
\begin{equation}\label{eq: Mori chamber on fiber}
\Mov(X_t)=\bigcup_{\substack{\phi \colon X_t \dto Y \text{~isom.~in codim.~} 1,\\ \phi \text{~is an MMP}}} \phi^{-1}_{*}\Nef^e(Y).
\end{equation}

For any $[D] \in \Mov(X_t)$, by Theorem \ref{thm: MMP} (4), there exists an effective divisor $\mathcal{D}$ on $X$ such that $[D] = [\mathcal{D}_t]$. By Theorem \ref{thm: MMP} (2), there exists a $\mathcal{D}$-MMP over $T$,
\[
h \colon X \dashrightarrow W/T,
\]
which induces a $D$-MMP on the fiber $X_t$. By the same argument as above, we may assume that $\mathcal{D}_W \coloneqq h_*\mathcal{D}$ is semi-ample over $T$. By Theorem \ref{thm: MMP}~(1), we have $\mathcal{D}_{W,t} \sim_{\mathbb{R}} (h_t)_*\mathcal{D}_t$. Hence, we have 
\[
[(h_t)_*D] = [(h_t)_*\mathcal{D}_t] \in \Nef^e(W_t).
\]
This proves the inclusion “$\subset$’’ in \eqref{eq: Mori chamber on fiber}.

Conversely, suppose that $\phi \colon X_t \dashrightarrow Y$ is an MMP which is isomorphic in codimension~$1$. We need to show that
\[
\Mov(X_t)\supset \phi_*^{-1}\Nef^e(Y).
\]
Let $B_Y$ be a divisor on $Y$ with $[B_Y]\in \Nef^e(Y)$, and let $B$ be its strict transform on $X_t$. Then $[B]\in \bMov^e(X_t)$. By Theorem~\ref{thm: MMP} (4), there exists an effective divisor $\mathcal{B}$ on $X$ such that $[\mathcal{B}_t]=[B]$. Let 
\[
X \dashrightarrow W/T
\]
be a $\mathcal{B}$-MMP$/T$ which is a good minimal model of $\mathcal B$ by Theorem \ref{thm: MMP for MKD}. By Theorem~\ref{thm: MMP} (2), the induced map
\[
\psi \colon X_t \dashrightarrow W_t
\]
is a $B$-MMP which is a good minimal model of $\mathcal B_t$. In particular, $\mathcal{B}_{W,t}$ is semi-ample and satisfies $[\mathcal{B}_{W,t}]=[B_{W_t}]$ by Theorem~\ref{thm: MMP} (1), where $B_{W_t}$ is the strict transform of $B$ on $W_t$. Since $[B]\in \bMov^e(X_t)$, the map $\psi$ is isomorphic in codimension~$1$ by Lemma \ref{lem: movable give iso in codim 1}. Thus,
\[
\phi \circ \psi^{-1} \colon W_t \dashrightarrow Y
\]
is also isomorphic in codimension~$1$. Under this map, the semi-ample divisor $\mathcal{B}_{W,t}$ is sent to a divisor $\Theta$ satisfying $[\Theta]=[B_Y]$. Hence, $\Theta$ is nef. By Lemma~\ref{lem: common mm} (1), $\Theta$ is semi-ample. This implies that
\[
[B_Y]=[\Theta]\in \Mov(Y).
\] Therefore, we obtain $\phi_*^{-1}(B_Y)\in \Mov(X_t)$, which establishes the desired inclusion.

By the above discussion, although the fibers are not known to be MKD spaces, under the assumptions of Theorem~\ref{thm: MMP}, it still makes sense to compare the Mori chamber decomposition of the total fiber space, namely \eqref{eq: Mori chamber decomp}, with the Mori chamber decompositions of the fibers, namely \eqref{eq: Mori chamber on fiber}. The following result shows that these Mori chamber decompositions coincide under the natural identification.

\begin{theorem}\label{thm: mov}
Let $f: X \to T$ be a fibration. Suppose that $S \subset T$ is a Zariski dense subset such that for any $s \in S$, the fiber $X_s$ satisfies
\[
H^1(X_s, \mO_{X_s}) = H^2(X_s, \mO_{X_s}) = 0.
\]
Assume further that the geometric generic fiber $X_{\bar\eta}$ is a klt MKD space. Then, up to a generically finite base change of $T$, there exists a Zariski open subset $T_0\subset T$ such that for any $U\subset T_0$ and each MKD fiber space $X'/T_0\in {\rm b}(X_{T_0}/T_0)$, we can identify the Mori chamber decompositions of $\Mov(X'/T_0), \Mov(X'_{U}/U)$ and $ \Mov(X'_t), t\in U$ under the natural restriction maps. In particular, we have isomorphisms among movable cones 
\[
\Mov(X'/T_0) \to \Mov(X'_U/U) \to \Mov(X'_t), \quad t\in U
\] under the natural restriction maps.
\end{theorem}
\begin{proof}
After performing a generically finite base change and shrinking $T$ if necessary, we may assume that Theorem \ref{thm: geometric MKD space}, Lemma \ref{lem: uniform behavior in bc}, and Theorem \ref{thm: MMP} hold over $T$, and that $T$ coincides with $T_0$. In the following argument, since we only use the conclusions of Lemma \ref{lem: uniform behavior in bc} and Theorem \ref{thm: MMP}, we may assume that $X'/T$ is exactly $X/T$ (cf.\ Remark \ref{rmk: model may not satisfy coh conditions}).

By Lemma \ref{lem: uniform behavior in bc} (1), for any open subset $V\subset T$, if $Y/V\in {\rm b}(X_{V}/V)$, then the natural maps
\[
N^1(Y/V) \to N^1(Y_t), \quad \Nef(Y/V) \to \Nef(Y_t)
\] are isomorphisms for any $t\in V$. By Theorem \ref{thm: MMP} (4), this implies the natural identification of effective nef cones
\[
\Nef^e(Y/V)= \Nef(Y/V) \cap \Eff(Y/V)\simeq \Nef(Y_t) \cap \Eff(Y_t) =\Nef^e(Y_t).
\] By Theorem \ref{thm: MMP} (2) and (3), an MMP$/T$ on $X$ restricts to an MMP of the same type on each fiber $X_t$, and the converse also holds. Moreover, by Theorem \ref{thm: MMP} (1), divisors are compatible with respect to restrictions. Therefore, $\Mov(X/T)$ and $\Mov(X_t)$ share the same chambers under the identification $N^1(X/T) \simeq N^1(X_t)$. 
\end{proof}

\subsection{Boundedness of moduli spaces}\label{subsec: moduli}

This section aims to apply the theory of MKD fiber spaces to the boundedness problem of MKD spaces via their birational boundedness. Although the overall strategy is now standard (see \cite{HMX14, HX15, HMX18, FHS24}, etc.), in the setting of MKD spaces, one must carefully lay the necessary foundations to ensure that the arguments carry over as expected.

The following result is an analogue, in the setting of MKD fiber spaces, of the main technical theorem \cite[Theorem~6.18]{FHS24} for elliptic fibrations (cf. \cite[Remark~6.19]{FHS24}). We also note that the assumption on singularities is relaxed from terminal to klt.

\begin{theorem}\label{thm: bir to bdd}
Let $\mathcal S=\{X_i \mid i\in I\}$ be a set of projective varieties and $f: \Yy \to T$ be a fibration between varieties. Suppose that
\begin{enumerate}
\item the geometric generic fiber $\Yy_{\bar\eta}$ is a klt MKD space,
\item $H^1(\Yy_s, \Oo_{\Yy_s})=H^2(\Yy_s, \Oo_{\Yy_s})=0$ for the fibers of $f$ over a Zariski open subset of $T$, and 
\item for each $X_i \in \mathcal S$, there exist some $t\in T$ and a sequence of MMP $ \Yy_t \dto X_i$ of an effective divisor.
\end{enumerate}
Then there exist a non-empty Zariski open subset $T_0 \subset T$ and a fibration between schemes of finite type
\[
g : \mathcal X \to T',
\]
such that if $X_i \in \mathcal S$ is connected to $\mathcal Y_t$ with $t \in T_0$ by an MMP as in (3), then $X_i$ is isomorphic to some fiber of $g$.
\end{theorem}
\begin{proof}
After a generically finite base change of $T$, we may assume that Theorem \ref{thm: geometric MKD space} and Theorem \ref{thm: MMP} hold for $\Yy \to T$. It suffices to prove the statement under this assumption. By Theorem~\ref{thm: MMP}~(3), for $t\in T_0$, any sequence of $D$-MMP
\[
\Yy_t \dto X_i
\]
with $D\geq 0$ is induced by restricting a sequence of $\mD$-MMP over $T_0$ on $\Yy_{T_0}/T_0$, denoted by
\[
\Yy_{T_0}\dto \Xx'/T_0.
\]
In particular, $X_i\simeq \Xx'_t$. As $\Yy_{T_0}/T_0$ is an MKD fiber space by Proposition \ref{prop: generic property for MKD space}, the set ${\rm b}(\Yy_{T_0}/T_0)$ is finite by Proposition \ref{prop: structure of b and bc} (1). Then, the natural map
\[
g: \Xx \coloneqq \bigsqcup_{\Xx'/T_0 \in {\rm b}(\Yy_{T_0}/T_0)} \Xx' \to T_0
\] satisfies the claim.
\end{proof}

We now prove Theorem \ref{thm: bdd for rationally connected CY}, which follows directly from \cite[Theorem~1.6]{Bir23} together with Theorem \ref{thm: bir to bdd}.

\begin{proof}[Proof of Theorem \ref{thm: bdd for rationally connected CY}]
By \cite[Theorem 1.1]{HMX14}, there exists an $\ep > 0$ such that every $X \in \mathcal S_n$ has $\ep$-lc singularities; equivalently, if $E$ is any exceptional divisor over $X$, then its discrepancy $a(E,X) \geq \ep - 1$. Indeed, if there is no such $\ep > 0$, then we could choose a sequence of varieties $\{X_i\}_{i\in \Nn} \subset \mathcal S_n$ and an exceptional divisor $E_i$ over each $X_i$ such that 
\[
\lim_{i \to \infty} a(E_i, X_i)=-1.
\] Since each $X_i$ has klt singularities, we have 
$a(E_i, X_i) > -1$. After passing to a subsequence, we may assume that $0 > a(E_i, X_i) > -1$ for every $i$ and $\{a(E_i, X_i) \}_{i\in \Nn}$ is strictly decreasing.
By \cite{BCHM10}, there exists a birational morphism $f_i: Y_i \to X_i$ that extracts exactly the divisor $E_i$. 
Thus,
\[
K_{Y_i} + (-a(E_i, X_i))E_i = f_i^* K_{X_i}.
\]
Hence $\{-a(E_i, X_i)\}_{i\in\Nn}$ is strictly increasing, which contradicts \cite[Theorem 1.1]{HMX14}. 

By \cite[Theorem~1.6]{Bir23}, there exists a fibration $f \colon \mathcal Y \to T$ between schemes of finite type such that for each $X \in \mathcal S_n$, there exists some fiber
$\mathcal Y_t$ with the property that $X$ and $\mathcal Y_t$ are isomorphic in codimension~$1$. Moreover, we know that $\mathcal Y_t$ is a $\bQ$-factorial variety with klt singularities. Indeed, by the last paragraph in the proof of \cite[Theorem~1.7]{Bir23} (note that \cite[Theorem~1.6]{Bir23} is a special case of \cite[Theorem~1.7]{Bir23}), the fiber $\mathcal Y_t$ has $\frac{\epsilon}{2}$-lc singularities. Here $\mathcal Y_t$ corresponds to the variety denoted by $\tilde X$ in the proof of \cite[Theorem~1.7]{Bir23}. Note that $\mathcal Y_t$ is $\mathbb Q$-factorial since it is obtained by running an MMP from a $\mathbb Q$-factorial variety (namely, $X''$ in the proof of \cite[Theorem~1.7]{Bir23}). 

We replace $T$ by the Zariski closure of 
\[
\Tt \coloneqq \{t \in T \mid \text{$\mathcal Y_t$ and $X$ are isomorphic in codimension~$1$ for some $X \in \mathcal S_n$}\}.
\] Since $T$ has only finitely many irreducible components, by restricting to one fixed irreducible component, we may assume that $T$ is irreducible. Moreover, by Noetherian induction and repeating the above process, we will freely replace $T$ by a Zariski open subset of $T$ in the following argument. Note that for each $t\in \Tt$, $\Yy_t$ is still a rationally connected Calabi-Yau variety with klt singularities by the previous discussion. By \cite[Corollary 1.8]{Bir23}, there exists some $I\in\Zz_{>0}$ such that 
\[
IK_{\Yy_t} \sim 0 \text{~for any~} t\in \Tt.
\]
Shrinking $T$ and applying the upper-semicontinuity of cohomology to the coherent sheaves $g_*\Oo_{\Yy}(I K_{\Yy})$ and $g_*\Oo_{\Yy}(-I K_{\Yy})$, we see that $g_*\Oo_{\Yy}(I K_{\Yy}) \neq 0$ and $g_*\Oo_{\Yy}(-I K_{\Yy}) \neq 0$. This implies that
\begin{equation}\label{eq: global CY}
I K_{\Yy} \sim 0 / T.
\end{equation} As $\Tt$ is Zariski dense in $T$, shrinking $T$ further, we see that $\Yy$ has klt singularities. 

\medskip

In the following, we show that the geometric generic fiber $\Yy_{\bar\eta}$ is a klt MKD space. 

First, we show that $\Yy_{\bar\eta}$ is rationally connected. By \cite[IV, Theorem~3.5.3]{Kol96}, after shrinking $T$, every fiber $\mathcal Y_t$ for $t \in T$ is rationally connected. Let $\tilde{\mathcal Y} \to \mathcal Y$ be a resolution. Then, by \cite[Corollary~1.6]{HM07}, after further shrinking $T$, we may assume that every fiber $\tilde{\mathcal Y}_t$ for $t \in T$ is rationally connected. We claim that $\ti \Yy_{\bar\eta}$ is still rationally connected. Let $\ti \Yy_{\bar\eta} \dto \ti\Ww$ be the MRC-fibration of $\ti \Yy_{\bar\eta}$ (see \cite[IV.5]{Kol96}). In particular, $\ti\Ww$ is not uniruled by \cite[Corollary 1.4]{GHJ03}. By the standard spreading-out and specialization techniques (see, for example, Lemma~\ref{lem: spread out and specialization}), there exist a generically finite morphism $T' \to T$ and maps
\[
\mathcal Y' \dto \mathcal W' \to T',
\]
such that over the geometric generic point $\bar\eta'$ of $T'$, the natural map
\[
\mathcal Y'_{\bar\eta'} \dto \mathcal W'_{\bar\eta'}
\]
coincides with the map $\tilde{\mathcal Y}_{\bar\eta} \dto \ti{\mathcal W}$. By \cite[IV, Theorem~1.8.2]{Kol96}, there exists a closed subvariety $Z \subset T'$ such that a fiber $\mathcal W_z$ is uniruled if and only if $z \in Z$ (note that $z$ is not necessarily a closed point). By \cite[IV, Proposition (1.3.1) and (1.3.2)]{Kol96}, since $\mathcal W'_{\bar\eta'}$ is not uniruled, the generic fiber $\mathcal W'_{\eta'}$ is also not uniruled. Therefore, we see that $\eta' \not\in Z$. Replacing $T'$ by $T' \setminus Z$, we see that each fiber $\mathcal W'_t$ for $t \in T'$ is not uniruled by the property of $Z$. After shrinking $T$, this implies that each fiber $\mathcal W_t$ for $t \in T$ is not uniruled. Since $\tilde{\mathcal Y}_t$ for $t \in T$ is rationally connected, it follows that $\mathcal W'_t$ must be a point. Hence, $\ti{\mathcal W}$ is also a point, which shows that $\tilde{\mathcal Y}_{\bar\eta}$ is rationally connected. Therefore, $\mathcal Y_{\bar\eta}$ is also rationally connected.

Next, we show that $\Yy_{\bar\eta}$ is $\bQ$-factorial. By \cite{BCHM10}, there exists a small $\bQ$-factorial modification $\ti\Yy'' \to \Yy_{\bar\eta}$. Again, by the standard spreading-out and specialization techniques, there exist a generically finite morphism $T'' \to T$ and morphisms
\[
\mathcal Y'' \to \mathcal Y_{T''} \to T'',
\]
where $\mathcal Y_{T''} \coloneqq \mathcal Y \times_{T} T''$, such that over the geometric generic point $\bar\eta''$ of $T''$, the natural map
\[
\mathcal Y''_{\bar\eta''} \to \mathcal Y_{T'',\bar\eta''}
\]
coincides with the morphism $\tilde{\mathcal Y}'' \to \mathcal Y_{\bar\eta}$. Shrinking $T''$, we can assume that $\mathcal Y'' \to \Yy_{T''}$ is a fiberwise small modification. As $\Yy_t$ is $\bQ$-factorial over a Zariski dense subset of $T''$, we see that $\Yy''_t \to \Yy_{t}$ is an isomorphism over this set. Therefore, after shrinking $T''$ further, $\mathcal Y'' \to \Yy_{T''}$ is an isomorphism. This implies that $\ti\Yy'' \to \Yy_{\bar\eta}$ is an isomorphism. In particular, $\Yy_{\bar\eta}$ is $\bQ$-factorial.

Combining the above discussion with \eqref{eq: global CY}, we see that $\mathcal Y_{\bar\eta}$ is a $\bQ$-factorial $n$-dimensional rationally connected Calabi-Yau variety with klt singularities. By the assumption of the theorem, the Morrison-Kawamata cone conjecture holds for $\mathcal Y_{\bar\eta}$, and every effective $\bR$-Cartier divisor on $\mathcal Y_{\bar\eta}$ admits a good minimal model. By Corollary \ref{cor: Mori dream space, CY are MKD space} (2), $\mathcal Y_{\bar\eta}$ is a klt MKD space.

\medskip

Finally, since $\mathcal Y_s$ for any $s \in T$ is a rationally connected variety with klt singularities, we have 
\[
H^1(\mathcal Y_s, \mathcal O_{\mathcal Y_s}) = H^2(\mathcal Y_s, \mathcal O_{\mathcal Y_s}) = 0.
\] 
Moreover, for each $X \in \mathcal S_n$, there exists some $t \in T$ such that $X$ and $\mathcal Y_t$ are isomorphic in codimension $1$. Hence, $X$ can be obtained from $\mathcal Y_t$ by running an MMP with respect to some movable divisor. Therefore, all the assumptions of Theorem~\ref{thm: bir to bdd} are fulfilled. The desired result then follows from Theorem~\ref{thm: bir to bdd} by Noetherian induction.
\end{proof}

\bibliographystyle{alpha}

\bibliography{reference.bib}

\end{document}